\documentclass[a4paper,11pt]{article}
\usepackage[english]{babel}
\usepackage{amsthm,amsmath,amssymb}
\usepackage{amsfonts} 
\usepackage[dvips]{epsfig}
\usepackage[dvips]{graphicx}

\newcommand{\meil}[1]{\mathbf{#1}} 
\newcommand{\latin}[1]{\emph{#1}}

\theoremstyle{plain}
\newtheorem{theorem}{Theorem}
\newtheorem{lemma}{Lemma}

\newtheorem{proposition}[lemma]{Proposition}
\newtheorem{claim}{Claim}
\newcommand{\hypHist}{\hyp{AS}\ensuremath{_{\mathbf{hist}}}}
\newtheorem*{assumptionset}{Assumption set \hypHist}
\theoremstyle{definition}

\theoremstyle{remark}
\newtheorem{remark}{Remark}

\newcommand{\egaldef}{:=} 
\newcommand{\defegal}{=:} 
\newcommand{\flens}{\mapsto} 
\newcommand{\flapp}{\mapsto} 
\newcommand{\telque}{\, \mbox{ s.t. } \,} 

\newcommand{\un}{\mathbf{1}} 

\newcommand{\R}{\mathbb{R}} 
\newcommand{\N}{\mathbb{N}} 
\newcommand{\X}{\mathcal{X}}
\newcommand{\Y}{\mathcal{Y}}

\DeclareMathOperator{\card}{Card} 
\DeclareMathOperator{\argmin}{argmin} 

\newcommand{\paren}[1]{\left( \left. #1 \right. \right)} 
\newcommand{\sparen}[1]{(#1)} 
\newcommand{\croch}[1]{\left[ \left. #1 \right. \right]} 
\newcommand{\scroch}[1]{[#1]} 
\newcommand{\set}[1]{\left\{ \left. #1 \right. \right\}}
\newcommand{\sset}[1]{\{ #1 \}}
\newcommand{\absj}[1]{\left\lvert #1 \right\rvert} 
\providecommand{\norm}[1]{\left \lVert #1 \right\rVert}
\newcommand{\carre}[1]{\left(#1\right)^2}

\renewcommand{\P}{\mathbb{P}}
\newcommand{\Prob}{\mathbb{P}} 
\newcommand{\E}{\mathbb{E}} 
\DeclareMathOperator{\var}{var} 
\newcommand{\sachant}{\, \right| \left. \,} 
\newcommand{\loi}{\mathcal{L}} 
\DeclareMathOperator{\Leb}{Leb} 

\newcommand{\bayes}{s}
\newcommand{\perte}[1]{\ell\paren{\bayes , #1 }}
\newcommand{\ERM}{\widehat{s}}

\newcommand{\M}{\mathcal{M}}
\newcommand{\Mt}{\widetilde{\M}}
\newcommand{\mt}{\widetilde{m}}
\newcommand{\mM}{m \in \M}

\newcommand{\mo}{\ensuremath{m^{\star}}} 
\newcommand{\Did}{\ensuremath{\widehat{D}^{\star}}}
\newcommand{\Kid}{\ensuremath{\widehat{K}^{\star}}}
\newcommand{\mideale}{\ensuremath{\widehat{m}^{\star}}}
\newcommand{\midlin}{\ensuremath{\widehat{m}^{\star}_{\textrm{lin}}}}
\newcommand{\middim}{\ensuremath{\widehat{m}^{\star}_{\textrm{dim}}}}
\newcommand{\mh}{\widehat{m}}

\newcommand{\mhpenVF}{\mh_{\mathrm{penVF}}}

\newcommand{\Mdim}{\ensuremath{\M_{\mathrm{dim}}}} 

\DeclareMathOperator{\pen}{pen}
\DeclareMathOperator{\crit}{crit}
\newcommand{\penVF}{\pen_{\mathrm{VF}}} 
\newcommand{\penHO}{\pen_{\mathrm{HO}}} 
\newcommand{\penid}{\pen_{\mathrm{id}}} 
\newcommand{\penMal}{\pen_{\mathrm{Cp}}} 

\newcommand{\Pnb}[1][]{{P_n^{W #1}}}

\newcommand{\sfloor}[1]{\ensuremath{\lfloor #1 \rfloor}}
\newcommand{\grandO}{\ensuremath{\mathcal{O}}}

\newcommand{\hyp}[1]{\ensuremath{\mathbf{(#1)}}} 
\newcommand{\hypAb}{\hyp{Ab}} 
\newcommand{\hypAn}{\hyp{An}} 
\newcommand{\hypAp}{\hyp{Ap}} 
\newcommand{\hypArXl}{\hyp{Ar^{X}_{\ell}}} 
\newcommand{\hypPpoly}{\hyp{P1}} 
\newcommand{\hypPrich}{\hyp{P2}} 

\newcommand{\crXl}{\ensuremath{c_{\mathrm{r},\ell}^X}}
\newcommand{\cbiasmaj}{C_{\mathrm{b}}^{+}} 
\newcommand{\cbiasmin}{C_{\mathrm{b}}^{-}} 
\newcommand{\betamaj}{\beta_{+}} 
\newcommand{\betamin}{\beta_{-}} 
\newcommand{\aM}{\alpha_{\M}}
\newcommand{\cM}{c_{\M}}
\newcommand{\sigmin}{\sigma_{\min}} 

\newcommand{\delc}{\overline{\delta}} 

\newcommand{\Il}{I_{\lambda}}
\newcommand{\lamm}{\lambda \in \Lambda_m} 

\newcommand{\pl}{p_{\lambda}} 
\newcommand{\sigl}{\sigma_{\lambda}}
\newcommand{\sigla}{\sigma_{\lambda}^r}
\newcommand{\sigld}{\sigma_{\lambda}^d}

\newcommand{\phl}{\widehat{p}_{\lambda}} 

\newcommand{\Cor}{\ensuremath{C_{\mathrm{or}}}}
\newcommand{\epsCor}{\ensuremath{\varepsilon_{\Cor,N}}}
\newcommand{\Cov}{\ensuremath{\mathcal{C}_{\mathrm{ov}}}}

\newcommand{\siga}{\ensuremath{\sigma_{a}}}
\newcommand{\sigb}{\ensuremath{\sigma_{b}}}

\newcommand{\hypThm}{\hyp{HThm}}
\newcommand{\LThm}{L_{\hypThm}}
\newcommand{\hypProA}{\hyp{HPro2}}
\newcommand{\LProA}{L_{\hypProA}}
\newcommand{\hypProB}{\hyp{HPro3}}
\newcommand{\LProB}{L_{\hypProB}}
\newcommand{\KoracleGenerique}{C}
\newcommand{\KThmPenVFOKproba}{K_0}
\newcommand{\KThmDimFAILproba}{K_1}
\newcommand{\KThmDimFAILoracle}{\mathbb{C}_1}
\newcommand{\KProLinOKproba}{K_2}
\newcommand{\KProLinOKoracle}{\mathbb{C}_2}
\newcommand{\KProCpFAILproba}{K_3}
\newcommand{\KProCpFAILdim}{K_4}
\newcommand{\KProCpFAILoracle}{K_5}

\newcommand{\mob}{\ensuremath{m_{\ast}}} 

\newcommand{\Mreg}{\M^{(\mathrm{reg})}}
\newcommand{\Mdeuxpas}{\M^{(\mathrm{reg},1/2)}}
\newcommand{\mMdeuxpas}{m \in \Mdeuxpas}
\newcommand{\Mdeuxpasjoint}[1]{\M^{(\mathrm{reg},#1)}}
\newcommand{\Mdeuxpasjointvar}{\M^{(\mathrm{reg},var)}}

\newcommand{\CadreEx}{X1--005}

\newcommand{\Kh}{\widehat{K}}
\newcommand{\Fh}{\widehat{F}}
\newcommand{\sighsq}{\ensuremath{\widehat{\sigma^2}}}

\newcommand{\Dimset}{\ensuremath{\mathcal{D}}} 

\renewcommand{\Il}{\lambda}

\newcommand{\refapp}{Appendix~\ref{sec.app}}
\newcommand{\refTabFigapp}{Table~\ref{LL.tab.un} and Figures~\ref{fig.res.boxB-div1}--\ref{fig.res.boxB-div2} in \refapp}
\newcommand{\refTabapp}{Table~\ref{LL.tab.un} in \refapp}
\newcommand{\refFigappid}{Figures~\ref{fig.res.boxB-Id1}--\ref{fig.res.boxB-Id2} in \refapp}
\newcommand{\refFigpath}{Figure~\ref{fig.path-density} in \refapp}
\newcommand{\Figpathmain}{Figure~\ref{fig.th1.path-density}}
\newcommand{\refeqmiddim}{\eqref{eq.middim}}

\begin{document}
\title{Choosing a penalty for model selection in heteroscedastic regression}
\author{Sylvain Arlot
\thanks{http://www.di.ens.fr/$\sim$arlot/} \\
CNRS ; Willow Project-Team \\
Laboratoire d'Informatique de l'Ecole Normale Superieure\\
(CNRS/ENS/INRIA UMR 8548)\\
INRIA - 23 avenue d'Italie - CS 81321 \\
75214 PARIS Cedex 13 - France \\
\texttt{sylvain.arlot@ens.fr}}
\date{\today}
\maketitle
\section{Introduction} \label{sec.intro}
Penalization is a classical approach to model selection. 
In short, penalization chooses the model minimizing the sum of the empirical risk (how well the model fits data) and of some measure of complexity of the model (called penalty); see FPE \cite{Aka:1969}, AIC \cite{Aka:1973}, Mallows' $C_p$ or $C_L$ \cite{Mal:1973}. 
A huge amount of literature exists about penalties proportional to the dimension of the model in regression, showing under various assumption sets that dimensionality-based penalties like $C_p$ are asymptotically optimal \cite{Shi:1981,KCLi:1987,Pol_Tsy:1990}, and satisfy non-asymptotic oracle inequalities \cite{Bar_Bir_Mas:1999,Bar:2000,Bar:2002,Bir_Mas:2006}.
Nevertheless, all these results assume data are homoscedastic, that is, the noise-level does not depend on the position in the feature space, an assumption often questionable in practice. 
Furthermore, $C_p$ is empirically known to fail with heteroscedastic data, as showed for instance by simulation studies in \cite{Arl:2009:RP,Arl_Cel:2009:segm}.

In this paper, it is assumed that data {\em can} be heteroscedastic, but not necessary with certainty. 
Several estimators adapting to heteroscedasticity have been built thanks to model selection (see \cite{Gen:2008} and references therein), but always assuming the model collection has a particular form. 
Up to the best of our knowledge, only cross-validation or resampling-based procedures are built for solving a general model selection problem when data are heteroscedastic. This fact was recently confirmed, since resampling and $V$-fold penalties satisfy oracle inequalities for regressogram selection when data are heteroscedastic \cite{Arl:2009:RP,Arl:2008a}. 
Nevertheless, adapting to heteroscedasticity with resampling usually implies a significant increase of the computational complexity.

\medskip

The main goal of the paper is to understand whether the additional computational cost of resampling can be avoided, when, and at which price in terms of statistical performance. 
Let us emphasize that determining from data only whether the noise-level is constant is a difficult question, since variations of the noise can easily be interpretated as variations of the smoothness of the signal, and conversely. 
Therefore, the problem of choosing an appropriate penalty---in particular, between dimensionality-based and resampling-based penalties---must be solved unless homoscedasticity of data is not questionable at all. The answer clearly depends at least on what is known about variations of the noise-level, and on the computational power available.

The framework of the paper is least-squares regression with a random design, see Section~\ref{HP.sec.cadre}. 
We assume the goal of model selection is efficiency, that is, selecting a least-squares estimator with minimal quadratic risk, without assuming the regression function belongs to any of the models. 
Since we deal with a non-asymptotic framework, where the collection of models is allowed to grow with the sample size, a model selection procedure is said to be optimal (or efficient) when it satisfies an oracle inequality with leading constant (asymptotically) one. 
A classical approach to design optimal procedures is the unbiased risk estimation principle, recalled in Section~\ref{sec.urep}.

\medskip

The main results of the paper are stated in Section~\ref{sec.dimbased}. 
First, all dimensionality-based penalties are proved to be suboptimal---that is, the risk of the selected estimator is larger than the risk of the oracle multiplied by a factor $\KThmDimFAILoracle>1$---as soon as data are heteroscedastic, for selecting among regressogram estimators (Theorem~\ref{th.shape}). Note that the restriction to regressograms is merely technical, and we expect a similar result holds for general heteroscedastic model selection problems. 
Compared to the oracle inequality satisfied by resampling-based penalties in the same framework (Theorem~\ref{th.penVF+RP}, recalled in Section~\ref{sec.urep}), Theorem~\ref{th.shape} shows what is lost when using dimensionality-based penalties with heteroscedastic data: at least a constant factor $\KThmDimFAILoracle>1$.

Second, Proposition~\ref{pro.oracle.lin.Klarge} shows that a well-calibrated penalty proportional to the dimension of the models does not loose more than a constant factor $\KProLinOKoracle>\KThmDimFAILoracle$ compared to the oracle. 
Nevertheless, $C_p$ strongly overfits for some heteroscedastic model selection problems, hence loosing a factor tending to infinity with the sample size compared to the oracle (Proposition~\ref{pro.overfit.Mal}). 
Therefore, a proper calibration of dimensionality-based penalties is absolutely required when heteroscedasiticy is suspected.

These theoretical results are completed by a simulation experiment (Section~\ref{sec.simus}), showing a slightly more complex finite-sample behaviour. In particular, when the signal-to-noise ratio is rather small, improving a well-calibrated dimensionality-based penalty requires a significant increase of the computational complexity. 

Finally, from the results of Sections~\ref{sec.dimbased} and~\ref{sec.simus}, Section~\ref{sec.concl} tries to answer the central question of the paper: How to choose the penalty for a given model selection problem, taking into account prior knowledge on the noise-level and the computational power available?

All the proofs are made in Section~\ref{LL.sec.proof}.

\section{Framework} \label{HP.sec.cadre}
In this section, we describe the least-squares regression framework, model selection and the penalization approach. Then, typical examples of collections of models and heteroscedastic data are introduced.

\subsection{Least-squares regression} \label{HP.sec.regression}
Suppose we observe some data $(X_1,Y_1), \ldots (X_n,Y_n) \in \X \times \R$, independent with common distribution $P$, where the feature space $\X$ is typically a compact subset of $\R^k$.
The goal is to predict $Y$ given $X$, where $(X,Y) \sim P$ is a new data point independent of $(X_i,Y_i)_{1 \leq i \leq n}$. 
Denoting by $\bayes$ the regression function, that is $\bayes(x) = \E\croch{Y \sachant X=x}$, we can write 
\begin{equation} 
\label{HP.eq.donnees.reg} 
Y_i = \bayes(X_i) + \sigma(X_i) \varepsilon_i 
\end{equation} 
where $\sigma: \X \mapsto \R$ is the heteroscedastic noise level and $(\varepsilon_i)_{1 \leq i \leq n}$ are i.i.d. centered noise terms; $\varepsilon_i$ may depend on $X_i\,$, but has mean 0 and variance 1 conditionally on $X_i \,$. 

The quality of a predictor $t: \X \mapsto \Y$ is measured by the quadratic prediction loss
\[ \E_{(X,Y) \sim P} \croch{ \gamma(t,(X,Y)) } \defegal P \gamma(t) \qquad \mbox{where} \quad \gamma(t,(x,y)) = \paren{ t(x) - y }^2 \]
is the least-squares contrast. 
The minimizer of $P \gamma(t)$ over the set of all predictors, called Bayes predictor, is the regression function $\bayes$. 
Therefore, the excess loss is defined as  
\[ \perte{t} \egaldef P\gamma\paren{t} - P\gamma\paren{\bayes} =  \E_{(X,Y)\sim P} \carre{t(X)-\bayes(X)} \enspace . \]
Given a particular set of predictors $S_m$ (called a {\em model}), the best predictor over $S_m$ is defined by 
\[ \bayes_m \egaldef \argmin_{t \in S_m} \set{ P \gamma(t) } \enspace . \]
The empirical counterpart of $\bayes_m$ is the well-known {\em empirical risk minimizer}, defined by 
\[ \ERM_m \egaldef \argmin_{t \in S_m} \set{ P_n \gamma(t) } \] (when it exists and is unique), where $P_n = n^{-1} \sum_{i=1}^n \delta_{(X_i,Y_i)}$ is the empirical distribution function; 
$\ERM_m$ is also called least-squares estimator since $\gamma$ is the least-squares contrast.

\subsection{Model selection, penalization} \label{HP.sec.cadre.mod_selec}
Let us assume that a family of models $(S_m)_{\mM_n}$ is given, hence a family of empirical risk minimizers $(\ERM_m)_{\mM_n}$.
The model selection problem consists in looking for some data-dependent $\mh \in \M_n$ such that $\perte{\ERM_{\mh}}$ is as small as possible.
For instance, it would be convenient to prove an oracle inequality of the form
\begin{equation}  \label{eq.oracle}
\perte{\ERM_{\mh}} \leq \KoracleGenerique \inf_{\mM_n} \set{\perte{\ERM_m}} + R_n
\end{equation}
in expectation or with large probability, with leading constant $\KoracleGenerique$ close to 1 and $R_n = \grandO(n^{-1})\,$.

This paper focuses more precisely on model selection procedures by penalization, which can be described as follows.
Let $\pen: \M_n \flapp \R^+$ be some penalty function, possibly data-dependent, and define  \begin{equation} \label{HP.eq.penalization} \mh \in \argmin_{\mM_n} \set{\crit(m)} \qquad \mbox{with} \qquad \crit(m) \egaldef P_n \gamma(\ERM_m) + \pen(m) \enspace . \end{equation} 
The penalty $\pen(m)$ can usually be interpretated as a measure of the size of $S_m\,$.
Since the ideal criterion $\crit(m)$ is the true prediction error $P \gamma\paren{\ERM_m}$, the {\em ideal penalty} is 
\[ \penid(m) \egaldef P \gamma(\ERM_m) - P_n \gamma(\ERM_m) \enspace . \]
This quantity is unknown because it depends on the true distribution $P$. 
A natural idea is to choose $\pen(m)$ as close as possible to $\penid(m)$ for every $\mM_n \,$. This idea leads to the well-known unbiased risk estimation principle, which is properly introduced in Section~\ref{sec.urep.principe}. 
For instance, when each model $S_m$ is a finite dimensional vector space of dimension $D_m$ and the noise-level is constant equal to $\sigma$, Mallows \cite{Mal:1973} proposed the $C_p$ penalty defined by 
\[ \penMal(m) = \frac{2 \sigma^2 D_m}{n} \enspace . \]
Penalties proportional to $D_m$, like $C_p$, are extensively studied in Section~\ref{sec.dimbased}.

\medskip

Among the numerous families of models that can be used, this paper mostly considers ``histogram models'', where for every $\mM_n\,$, $S_m$ is the set of piecewise constant functions w.r.t. some fixed partition $\Lambda_m$ of $\X\,$. 
Note that the least-squares estimator $\ERM_m$ on some histogram model $S_m$ is also called a regressogram.
Then, $S_m$ is a vector space of dimension $D_m = \card(\Lambda_m)$ generated by the family $\paren{\un_{\Il}}_{\lamm} \,$.
Model selection among a family $\paren{S_m}_{\mM_n}$ of histogram models amounts to select a partition of $\X$ among $\sset{\Lambda_m}_{\mM_n}\,$.

Three arguments motivate the choice of histogram models for this theoretical study.
First, better intuitions can be obtained on the role of variations of the noise-level $\sigma(\cdot)$ over $\X\,$---or variations of the smoothness of $\bayes\,$---because an histogram models is generated by a localized basis $\paren{\un_{\Il}}_{\lamm} \,$.
Second, histograms have good approximation properties when the regression function $\bayes$ is $\alpha$-H\"olderian with $\alpha\in(0,1] \,$.
Third, all important quantities for understanding the model selection problem can be precisely controlled and compared, see \cite{Arl:2008a}.

\subsection{Examples of histogram model collections} \label{sec.cadre.collec}
Let us assume in this section for simplicity that $\X=[0,1) \,$. We define in this section several collections of models $\paren{S_m}_{\mM_n}\,$, always assuming that each $S_m$ is the histogram model associated to some partition $\Lambda_m$ of $\X\,$.

The most natural (and simple) collection of histogram models is the collection of {\em regular histograms} $\paren{S_m}_{m\in\Mreg_n}$ defined by
\[ \forall m \in \Mreg_n \egaldef \set{1, \ldots, M_n}, \quad \Lambda_m = \paren{ \left[ \frac{k-1}{m} , \, \frac{k}{m} \right) }_{1 \leq k \leq m} \enspace , \]
where the maximal dimension $M_n \leq n$ usually grows with $n$ slightly slower than $n \,$; reasonable choices are $M_n=\sfloor{n/2} \,$, $M_n = \sfloor{n/(\ln(n))}$ or $M_n = \sfloor{n/(\ln(n))^2} \,$.

Model selection among the collection of regular histograms then amounts to selecting the number $D_m \in \set{1, \ldots, M_n}$ of bins, or equivalently, selecting the bin size $1/D_m$ among $\set{1/M_n, \ldots, 1/2, 1}\,$.
The regular collection $\Mreg_n$ is a good choice when the distribution of $X$ is close to the uniform distribution on $[0,1]\,$, the noise-level $\sigma(\cdot)$ is almost constant on $[0,1]$, and the variations of $\bayes$ (measured by $\absj{\bayes^{\prime}}$) are almost constant over $\X\,$.

\medskip

Since we can seldom be sure these three assumptions are satisfied by real data, considering other collection of histograms models can be useful in general, in particular for adapting to possible heteroscedasticity of data, which is the main topic of the paper.
The simplest case of collection of histogram models with variable bin size is the collection of {\em histograms with two bin sizes and split at $1/2\,$}, $\paren{S_m}_{\mMdeuxpas_n}\,$, defined by
\[ 
\Mdeuxpas_n \egaldef \set{ \paren{D_1,D_2} \telque 1 \leq D_1 , \, D_2 \leq \frac{M_n}{2} } \cup \set{1} \enspace , \]
where $S_1$ is the set of constant functions on $\X$ and for every $m =(D_{m,1}, D_{m,2}) \in \paren{\N\backslash\set{0}}^2 \,$,
\[ \Lambda_m \egaldef 
\set{ \left[ \frac{k-1}{2 D_{m,1}} ; \frac{k}{2 D_{m,1}} \right) }_{1 \leq k \leq D_{m,1}} \cup 
\set{ \left[ \frac{D_{m,2} + k-1}{2 D_{m,2}} ; \frac{D_{m,2} + k}{2 D_{m,2}} \right) }_{1 \leq k \leq D_{m,2}} 
\enspace . \]

Note that using a collection of models such as $\paren{S_m}_{\mMdeuxpas_n}$ does not mean that data are {\em known} to be heteroscedastic; $\paren{S_m}_{\mMdeuxpas_n}$ can also be useful when one only {\em suspects} that at least one quantity among $\sigma$, $\absj{\bayes^{\prime}}$ and the density of $X$ w.r.t. the Lebesgue measure ($\Leb$) significantly varies over $\X\,$.
Distinguishing between the three phenomena precisely is the purpose of model selection: Overfitting occurs when a large noise level is interpretated as large variations of $\bayes$ with low noise. 
The interest of using $\Mdeuxpas_n$ is illustrated by Figure~\ref{fig.interet.Mdeuxpas}, where two data samples and the corresponding oracle estimators are plotted (left: heteroscedastic with $\bayes^{\prime}$ constant; right: homoscedastic with $\bayes^{\prime}$ variable).

\begin{figure}
\begin{minipage}[b]{.48\linewidth}
\includegraphics[width=\textwidth]{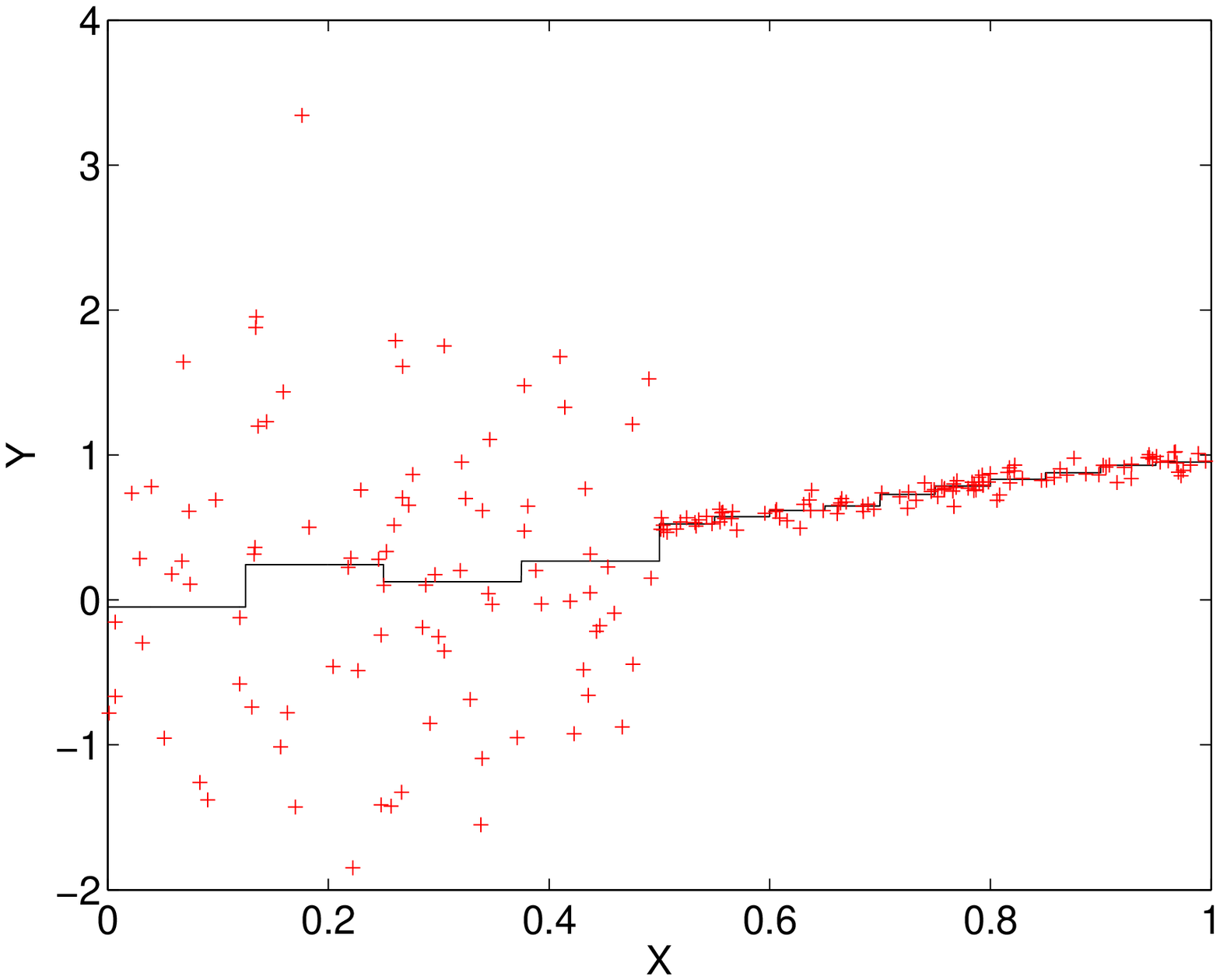}
\end{minipage}
\hspace{.025\linewidth}
\begin{minipage}[b]{.48\linewidth}
\includegraphics[width=\textwidth]{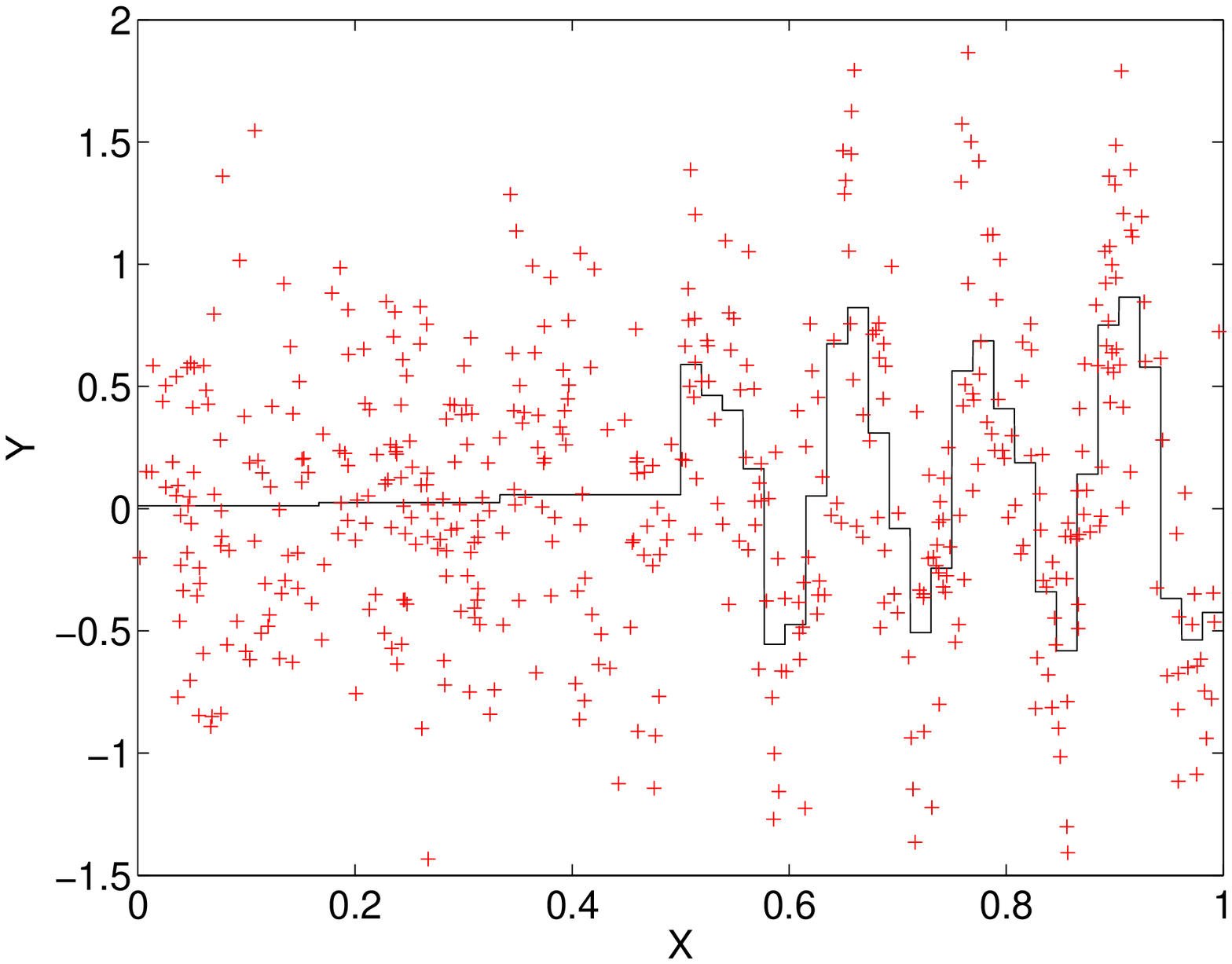}
\end{minipage}
\caption{One data sample (red `$+$') and the corresponding oracle estimators (black `--') among the family $\paren{\ERM_m}_{\mMdeuxpas_n}\,$.
Left: heteroscedastic data ($\bayes(x)=x\,$; $\sigma(x)=1$ if $x\leq 1/2$ and $\sigma(x)=1/20$ otherwise) with sample size $n=200\,$.
Right: homoscedastic data ($\bayes(x)=x/4$ if $x\leq 1/2$ and $\bayes(x)=1/8 + 2/3\times \sin(16 \pi x) $ otherwise; $\sigma(x)=1/2$) with sample size $n=500\,$.
\label{fig.interet.Mdeuxpas}
}
\end{figure}

\medskip

Using only two different bin sizes, with a fixed split at $1/2$, obviously is not the only collection of histograms that may be used. 
Let us mention here a few examples of alternative histogram collections:
\begin{itemize}
\item the split can be put at any fixed position $t \in (0,1)$ (possibly with different maximal number of bins $M_{n,1}$ and $M_{n,2}$ on each side of $t$), leading to the collection $\paren{S_m}_{m\in\Mdeuxpasjoint{t}_n }\,$.
\item the position of the split can be variable: 
\[ \Mdeuxpasjointvar_n = \bigcup_{t \in \mathcal{T}_n} \Mdeuxpasjoint{t}_n \quad \mbox{where} \quad \mathcal{T}_n \subset (0,1) \enspace , \quad 
\mbox{for instance} \quad  \mathcal{T}_n = \set{\frac{k}{\sqrt{n}} \telque 1 \leq k \leq \sqrt{n}-1 } \enspace . \] 
\item instead of a single split, one could consider collections with several splits (fixed or not), such that $\set{ 1/3 , 2/3 }$ or $\set{ 1/4,1/2,3/4 }$ for instance.
\end{itemize}
Remark that $\card(\Mdeuxpas_n) \leq M_n^2 \leq n^2 \,$, and the cardinalities of all other collections are smaller than some power of $n$.
Therefore, as explained in Section~\ref{sec.urep} below, penalization procedures using an estimator of $\penid(m)$ for every $\mM_n$ as a penalty are relevant.
This paper does not consider collections whose cardinalities grow faster than some power of $n$, such as the ones used for multiple change-point detection. 
Indeed, the model selection problem is of different nature for such collections, and requires the use of different penalties; see for instance \cite{Arl_Cel:2009:segm} about this particular problem. 

\medskip

Most results of the paper are proved for model selection among $\paren{S_m}_{m\in\Mdeuxpas_n}\,$, which already captures most of the difficulty of model selection when data are heteroscedastic.
The simplicity of $\paren{S_m}_{m\in\Mdeuxpas_n}$ may be a drawback for analyzing real data; in the present theoretical study, simplicity helps developing intuitions about the general case. 
Note that all results of the paper can be proved similarly when $\M_n=\Mdeuxpasjoint{t}$ and $t\in(0,1)$ is fixed; we conjecture these results can be extended to $\M_n=\Mdeuxpasjointvar_n\,$, at the price of additional technicalities in the proofs.

\section{Unbiased risk estimation principle for heteroscedastic data} \label{sec.urep}
The unbiased risk estimation principle is among the most classical approaches for model selection \cite{Ste:1981}. Let us first summarize it in the general framework. 
\subsection{General framework} \label{sec.urep.principe}
Assume that for every $\mM_n \, $, $\crit(m,(X_i,Y_i)_{1\leq i \leq n})$ estimates unbiasedly the risk $P\gamma\paren{\ERM_m}$ of the estimator $\ERM_m \,$. 
Then, an oracle inequality like \eqref{eq.oracle} with $\KoracleGenerique\approx 1$ should be satisfied by any minimizer $\mh$ of $\crit(m,(X_i,Y_i)_{1\leq i \leq n})$ over $\mM_n\,$.
For instance, FPE \cite{Aka:1969}, SURE \cite{Ste:1981} and cross-validation \cite{All:1974,Sto:1974,Gei:1975} are model selection procedures built upon the unbiased risk estimation principle.

\medskip

When $\crit(m,(X_i,Y_i)_{1\leq i \leq n})$ is a penalized empirical criterion given by \eqref{HP.eq.penalization}, the unbiased risk estimation principle can be rewritten as 
\[ \forall \mM_n \, , \quad \pen(m) \approx \E\croch{\penid(m)} = \E\croch{ \paren{P - P_n} \gamma\paren{\ERM_m} } \enspace , \] 
which is also known as Akaike's heuristics  or Mallows' heuristics.
For instance, AIC \cite{Aka:1973}, $C_p$ or $C_L$ \cite{Mal:1973} (see Section~\ref{sec.dimbased.ex}), covariance penalties \cite{Efr:2004} and resampling penalties \cite{Efr:1983,Arl:2009:RP} are penalties built upon the unbiased risk estimation principle.

\medskip

The unbiased risk estimation principle can lead to oracle inequalities with leading constant $\KoracleGenerique=1+o(1)$ when $n$ tends to infinity, by proving that deviations of $P\gamma\paren{\ERM_m}$ around its expectation are uniformly small with large probability. Such a result can be proved in various frameworks as soon as the number of models grows at most polynomially with $n$, that is, $\card(\M_n) \leq \cM n^{\aM}$ for some $\cM,\aM>0\,$; see for instance \cite{Bir_Mas:2006,Arl_Mas:2009:pente} and references therein for recent results in this direction in the regression framework.

\subsection{Histogram models} \label{sec.urep.histos}
Let $S_m$ be the histogram model associated with a partition $\Lambda_m$ of $\X\,$.
Then, the concentration inequalities of Section~\ref{sec.proof.conc} show that for most models, the ideal penalty is close to its expectation.
Moreover, the expectation of the ideal penalty can be computed explicitly thanks to Proposition~\ref{VFCV.pro.EcritVFCV-EcritID}, first proved in a previous paper \cite{Arl:2008a}:
\begin{equation} \label{LL.eq.Epenid}
\E\croch{ \penid(m) } = \frac{1}{n} \sum_{\lamm} \paren{ 2 + \delta_{n,\pl} } \paren{\carre{\sigla} + \carre{\sigld}} 
\end{equation}
where for every $\lamm\,$,
\begin{gather*}
\carre{\sigla} \egaldef \E \croch{ \paren{Y - \bayes(X)}^2 \sachant X \in \Il } = \E \croch{ \paren{\sigma(X)}^2 \sachant X \in \Il} \qquad
\carre{\sigld} \egaldef \E \croch{ \paren{\bayes(X) - \bayes_m(X)}^2 \sachant X \in \Il } 
\\
\pl \egaldef \P(X \in \Il) \qquad
\mbox{and} \qquad \forall n \in \N \, , \, \forall p \in (0,1] \, , \quad \absj{ \delta_{n,p} } \leq \min\set{ L_1, \frac{L_2}{(np)^{1/4}}}
\end{gather*}
for some absolute constants $L_1,L_2>0\,$.

\medskip

When data are homoscedastic, \eqref{LL.eq.Epenid} shows that if $D_m \, , \min_{\lamm}\set{n\pl} \rightarrow \infty$ and if $\norm{\bayes-\bayes_m}_{\infty} \to 0\,$,
\[ \forall \mM_n \, , \quad \E\croch{\penid(m)} \approx \frac{2 \sigma^2 D_m}{n} + \frac{2}{n} \sum_{\lamm} \carre{\sigld} \approx \frac{2 \sigma^2 D_m}{n} = \penMal(m) \enspace , \]
so that $C_p$ should yield good model selection performances by the unbiased risk estimation principle.
When the constant noise-level $\sigma$ is unknown, $\penMal$ can still be used by replacing $\sigma^2$ by some unbiased estimator of $\sigma^2\,$; see for instance \cite{Bar:2002} for a theoretical analysis of the performance of $C_p$ with some classical estimator of $\sigma^2\,$. 

On the contrary, when data are heteroscedastic, \eqref{LL.eq.Epenid} shows that applying the unbiased risk estimation principle requires to take into account the variations of $\sigma$ over $\X$. 
Without prior information on $\sigma(\cdot)\,$, building a penalty for the general heteroscedastic framework is a challenging problem, for which resampling methods have been successful.

\subsection{Resampling-based penalization} \label{sec.urep.RP}
The resampling heuristics \cite{Efr:1979} provides a way of estimating the distribution of quantities of the form $F(P,P_n)\,$, by building randomly from $P_n$ several ``resamples'' with empirical distribution $\Pnb\,$. Then, the distribution of $F(P_n,\Pnb)$ conditionally on $P_n$ mimics the distribution of $F(P,P_n)\,$. We refer to \cite{Arl:2009:RP} for more details and references on the resampling heuristics in the context of model selection.
Since $\penid(m)=F_m(P,P_n)\,$, the resampling heuristics can be used for estimating $\E\croch{\penid(m)}$ for every $\mM_n\,$.
Depending on how resamples are built, we can obtain different kinds of resampling-based penalties, in particular the following three ones.

First, bootstrap penalties \cite{Efr:1983} are obtained with the classical bootstrap resampling scheme, where the resample is an $n$-sample i.i.d. with common distribution $P_n\,$.
Second, general exchangeable resampling schemes can be used for defining the family of (exchangeable) resampling penalties \cite{Arl:2009:RP,Ler:2009:omsde}.
Third, $V$-fold penalties \cite{Arl:2008a} are a computationally efficient alternative to bootstrap and other exchangeable resampling penalties; they follow from the resampling heuristics with a subsampling scheme inspired by $V$-fold cross-validation.

Let us define here $V$-fold penalties, which are of particular interest because of their smaller computational cost when $V$ is small. 
Let $V \in \set{2, \ldots, n}\,$ and $\paren{B_j}_{1 \leq j \leq V}$ be a fixed partition of $\set{1, \ldots, n}$ such that $\sup_j \absj{\card(B_j) - n/V} < 1 \,$. 
For every $j$, define 
\begin{gather*} P_n^{(-j)} = \frac{1}{n - \card(B_j)} \sum_{i \notin B_j} \delta_{(X_i,Y_i)} \\
\mbox{and} \qquad \forall \mM_n \, , \quad \ERM_m^{(-j)} \in \argmin_{t \in S_m} \set{P_n^{(-j)} \gamma\paren{t}} \enspace . 
\end{gather*}
Then, the $V$-fold penalty is defined by 
\begin{equation} \label{eq.penVF} \penVF(m) \egaldef \frac{V-1}{V} \sum_{j=1}^V \paren{ P_n - P_n^{(-j)}} \gamma\paren{\ERM_m^{(-j)}} \enspace . \end{equation}

In the least-squares regression framework, exchangeable resampling and $V$-fold penalties have been proved in \cite{Arl:2009:RP,Arl:2008a} to satisfy an oracle inequality of the form \eqref{eq.oracle} with leading constant $\KoracleGenerique=\KoracleGenerique(n) \rightarrow 1$ when $n \rightarrow \infty\,$.
In order to state precisely one of these results, let us introduce a set of assumptions, called \hypHist. 
\begin{assumptionset}
For every $\mM_n \, $, $S_m$ is the set of piecewise constants functions on some fixed partition $\Lambda_m$ of $\X\,$, and $\M_n$ satisfies:
\begin{enumerate}
\item[\hypPpoly] Polynomial complexity of $\M_n\,$\textup{:} $\card(\M_n) \leq \cM n^{\aM}\,$.
\item[\hypPrich] Richness of $\M_n\,$\textup{:} $\exists m_0 \in \M_n$ s.t. $D_{m_0} = \card(\Lambda_{m_0}) \in \croch{\sqrt{n},c_{\mathrm{rich}} \sqrt{n}}\,$.
\end{enumerate}
Moreover, data $(X_i,Y_i)_{1 \leq i \leq n}$ are i.i.d. and satisfy:
\begin{enumerate}
\item[\hypAb] Data are bounded\textup{:} $\norm{Y_i}_{\infty} \leq A < \infty\,$.
\item[\hypAn] Uniform lower-bound on the noise level\textup{:} $\sigma(X_i) \geq \sigmin>0$ a.s.
\item[\hypAp] The bias decreases like a power of $D_m\,$\textup{:} constants $\betamin\geq\betamaj>0$ and $\cbiasmaj,\cbiasmin>0$ exist such that 
\[ \forall \mM_n \, , \quad \cbiasmin D_m^{-\betamin} \leq \perte{\bayes_m} \leq \cbiasmaj D_m^{-\betamaj} \enspace . \]
\item[\hypArXl] Lower regularity of the partitions for $\loi(X)\,$\textup{:} $ D_m \min_{\lamm} \set{\Prob\paren{X \in \Il}} \geq \crXl\,$.
\end{enumerate}
\end{assumptionset}
\begin{remark} \label{rk.hypHist}
Assumption set \hypHist\ is shown to be mild and discussed extensively in \cite{Arl:2009:RP}; we do not report such a discussion here because it is beyond the scope of the paper.  
In particular, when $\bayes$ is non-constant, $\alpha$-H\"olderian for some $\alpha \in (0,1]$ and $X$ has a lower bounded density with respect to the Lebesgue measure on $\X=[0,1]$, assumptions \hypPpoly, \hypPrich, \hypAp\ and \hypArXl\ are satisfied by all the examples of model collections given in Section~\ref{sec.cadre.collec} (see in particular \cite{Arl:2008a} for a proof of the lower bound in \hypAp\ for regular partitions, which applies to the examples of Section~\ref{sec.cadre.collec} since they are ``piecewise regular'').
Note also that {\em all the results} of the present paper relying on \hypHist\ also hold under various alternative assumption sets. For instance, \hypAb\ and \hypAn\ can be relaxed, see \cite{Arl:2009:RP} for details.
\end{remark}
\begin{theorem}[Theorem~2 in \cite{Arl:2008a}] \label{th.penVF+RP}
Assume that \hypHist\ holds true. 
Then, for every $V\geq 2$, a constant $\KThmPenVFOKproba(V)$ \textup{(}depending only on $V$ and on the constants appearing in \hypHist\textup{)} and an event of probability at least $1-\KThmPenVFOKproba(V) n^{-2}$ exist on which, for every 
\begin{gather}
\notag \mhpenVF \in \argmin_{\mM_n} \set{ P_n \gamma\paren{\ERM_m} + \penVF(m) } \enspace , \\
\perte{\ERM_{\mhpenVF}} \leq \paren{1 + \paren{\ln n}^{-1/5}} \inf_{\mM_n} \set{ \perte{\ERM_m} } \enspace . \notag
\end{gather}
\end{theorem}
In particular, $V$-fold penalization is asymptotically optimal: when $n$ tends to infinity, the excess loss of the estimator $\ERM_{\mhpenVF}$ is equivalent to the excess loss of the oracle estimator $\ERM_{\mo}\,$, defined by $\mo \in \argmin_{\mM_n}\set{\perte{\ERM_m}}\,$.
A result similar to Theorem~\ref{th.penVF+RP} has also been proved for exchangeable resampling penalties in \cite{Arl:2009:RP}, under the same assumption set \hypHist. In particular, Theorem~\ref{th.penVF+RP} is still valid when $V=n\,$. Let us emphasize that {\em general unknown variations of the noise-level $\sigma(\cdot)\,$} are allowed in Theorem~\ref{th.penVF+RP}.

Theorem~\ref{th.penVF+RP}---as well as its equivalent for exchangeable resampling penalties---mostly follows from the unbiased risk estimation principle presented in Section~\ref{sec.urep.principe}: For every model $\mM_n\,$, $\E\croch{ \penVF(m)}$ is close to $\E\croch{\penid(m)}$ whatever the variations of $\sigma(\cdot)\,$, and deviations of $\penVF(m)$ around its expectation can be properly controlled. The oracle inequality follows, thanks to \hypPpoly. 

\medskip

The main drawback of exchangeable resampling penalties, and even $V$-fold penalties, is their computational cost. Indeed, computing these penalties requires to compute for every $\mM_n$ a least-squares estimator $\ERM_m$ several times: $V$ times for $V$-fold penalties, at least $n$ times for exchangeable resampling penalties.
Therefore, except in particular problems for which $\ERM_m$ can be computed fastly, all resampling-based penalties can be untractable when $n$ is too large, except maybe $V$-fold penalties with $V=2$ or $3$.
Note that ($V$-fold) cross-validation methods suffer from the same drawback, in addition to their bias which makes them suboptimal when $V$ is small, see \cite{Arl:2008a}.

Furthermore, Theorem~\ref{th.penVF+RP} could suggest that the performance of $V$-fold penalization does not depend on $V$, so that the best choice always is $V=2$ which minimizes the computational cost. 
Although this asymptotically holds true at first order, quite a different picture holds when the signal-to-noise ratio is small, according to the simulation studies of \cite{Arl:2008a} and of Section~\ref{sec.simus} below. 
Indeed, the amplitude of deviations of $\penVF(m)$ around its expectation decreases with $V$, so that the statistical performance of $V$-fold penalties can be much better for large $V$ than for $V=2\,$. 

Remark that one could also define {\em hold-out penalties} by 
\begin{gather*} 
\forall \mM_n \, , \quad 
\penHO(m) \egaldef \frac{\card(I)}{n-\card(I)} \paren{ P_n - P_n^{(I)}} \gamma\paren{\ERM_m^{(I)}}
\quad \mbox{where } I \subset\set{1, \ldots, n} \mbox{ is deterministic,} \\
P_n^{(I)} = \frac{1}{\card(I)} \sum_{i \in I} \delta_{(X_i,Y_i)} \qquad 
\mbox{and} \qquad \forall \mM_n \, , \quad \ERM_m^{(I)} \in \argmin_{t \in S_m} \set{P_n^{(I)} \gamma\paren{t}} \enspace , 
\end{gather*}
which only requires to compute once $\ERM_m$ for each $\mM_n\,$.
The proof of Theorem~\ref{th.penVF+RP} can then be extended to hold-out penalties provided that $\min\set{\card(I),n-\card(I)}$ tends to infinity with $n$ fastly enough, for instance when $\card(I) \approx n/2\,$.
Nevertheless, hold-out penalties suffer from a larger variability than 2-fold penalties, which leads to quite poor statistical performances.

\medskip

Therefore, when computational power is strongly limited and the signal-to-noise ratio is small, it may happen that none of the above resampling-based model selection procedures is satisfactory in terms of both computational cost and statistical performance.
The purpose of the next two sections is to investigate whether the dimensionality of the models, which is freely available in general, can be used for building a computationally cheap model selection procedure with reasonably good statistical performance, in particular compared to $V$-fold penalties with $V$ small.

\section{Dimensionality-based model selection} \label{sec.dimbased}
Dimensionality as a vector space is the only information about the size of the models that is freely available in general.
So, when some penalty must be proposed, functions of the dimensionality $D_m$ of model $S_m$ are the most natural (and classical) proposals.
This section intends to measure the statistical performance of such procedures for least-squares regression with heteroscedastic data. 

\subsection{Examples} \label{sec.dimbased.ex}
As previously mentioned in Section~\ref{HP.sec.cadre.mod_selec}, $C_p$  defined by 
$\penMal(m) = 2 \sigma^2 D_m / n$ is the among most classical penalties for least-squares regression \cite{Mal:1973}. 
$C_p$ belongs to the family of {\em linear penalties}, that is, of the form 
\[ \Kh D_m  \enspace , \]
where $\Kh$ can either depend on prior information on $P$ (for instance, the value $\sigma$ of the---constant---noise-level) or on the sample only. 
A popular choice is $\Kh=2 \sighsq/n\,$, where $\sighsq$ is an estimator of the variance of the noise, see Section~6 of \cite{Bar:2000} for instance. Birg\'e and Massart \cite{Bir_Mas:2006} recently proposed an alternative procedure for choosing $\Kh$, based upon the ``slope heuristics''.

Refined versions of $C_p$ have been proposed, for instance in \cite{Bir_Mas:2006,Bar:2002,Sau:2006}---always assuming homoscedasticity.
Most of them are of the form 
\begin{equation} \label{eq.pen.F(D)}
\pen(m) = \Fh(D_m)
\end{equation}
where $\Fh$ depends on $n$ and $\sigma^2\,$, or an estimator $\sighsq$ of $\sigma^2\,$ when $\sigma^2$ is unknown.
The rest of the section focuses on {\em dimensionality-based penalties}, that is, penalties of the form \eqref{eq.pen.F(D)}.

\subsection{Characterization of dimensionality-based penalties} \label{sec.dimbased.charact}
Let us define, for every $D \in \Dimset_n = \set{D_m \telque \mM_n } \, $, 
\[ \Mdim(D) \egaldef \argmin_{\mM_n \telque D_m = D} \set{ P_n \gamma\paren{ \ERM_m } } \qquad \mbox{and} \qquad \Mdim \egaldef \bigcup_{D \in \Dimset_n} \Mdim(D) \enspace . \]
The following lemma shows that any dimensionality-based penalization procedure actually selects $\mh \in \Mdim\,$.
\begin{lemma} \label{claim:1}
For every function $F: \M_n \flens \R $ and any sample $(X_i,Y_i)_{1\leq i \leq n}\,$,
\[ \argmin_{\mM_n} \set{ P_n \gamma\paren{\ERM_m} + F(D_m) } \subset \Mdim \enspace . \]
\end{lemma}
\begin{proof}[proof of Lemma~\ref{claim:1}]
Let $\mh_F \in \argmin_{\mM_n} \set{ P_n \gamma\paren{\ERM_m} + F(D_m) } \,$. 
Then, whatever $\mM_n \,$, 
\begin{equation} \label{eq.pr.claim:1} P_n \gamma\paren{\ERM_{\mh_F}} + F(D_{\mh_F}) \leq P_n \gamma\paren{\ERM_m} + F(D_m) \enspace . \end{equation}
In particular, \eqref{eq.pr.claim:1} holds for every $\mM_n$ such that $D_m=D_{\mh_F}\,$, for which $F(D_{\mh_F})=F(D_m)\,$. Therefore, \eqref{eq.pr.claim:1} implies that $\mh_F\in \Mdim(D_{\mh_F})\,$, hence $\mh_F \in \Mdim\,$.
\end{proof}

Lemma~\ref{claim:1} shows that despite the variety of functions $F$ that can be used as a penalty, using a function of the dimensionality as a penalty always imply selecting among $\paren{S_m}_{m\in \Mdim}$ (keeping in mind that $\Mdim$ is random). 
Indeed, penalizing with a function of $D$ means that all models of a given dimension $D$ are penalized in the same way, so that the empirical risk alone is used for selecting among models of the same dimension.
By extension, we will call {\em dimensionality-based model selection procedure} any procedure selecting a.s. $\mh \in \Mdim\,$.

Breiman \cite{Bre:1992} previously noticed that only a few models---called ``RSS-extreme submodels''---can be selected by penalties of the form $F(D_m)=K D_m$ with $K\geq 0\,$. Although Breiman stated this limitation can be benefic from the computational point of view, results below show that this limitation precisely makes the quadratic risk increase when data are heteroscedastic.

\subsection{Pros and cons of dimensionality-based model selection}
As shown by equation \eqref{LL.eq.Epenid}, when data are heteroscedastic, $\E\croch{\penid(m)}$ is no longer proportional to the dimensionality $D_m\,$. The expectation of the ideal penalty actually is even not a function of $D_m$ in general. Therefore, the unbiased risk estimation principle should prevent anyone from using dimensionality-based model selection procedures.

Nevertheless, dimensionality-based model selection procedures are still used for analyzing heteroscedastic data for at least three reasons:
\begin{itemize}
\item by ignorance of any other trustable model selection procedure than $C_p\,$, or of the assumptions of $C_p\,$; 
\item because data are (wrongly) assumed to be homoscedastic; 
\item because they are simple and have a mild computational cost, no other measure of the size of the models being available.
\end{itemize}
The last two points can indeed be good reasons, provided that we know what we can loose---in terms of quadratic risk---by using a dimensionality-based model selection procedure instead of, for instance, some resampling-based penalty.
The purpose of the next subsections is to estimate theoretically the price of violating the unbiased risk estimation principle in heteroscedastic regression.

\subsection{Suboptimality of dimensionality-based penalization} \label{sec.main.subopt}
Theorem~\ref{th.shape} below shows that any dimensionality-based penalization procedure fails to attain asymptotic optimality for model selection among $\paren{S_m}_{m\in\Mdeuxpas_n}$ when data are heteroscedastic.
\begin{theorem} \label{th.shape}
Assume that data $(X_1,Y_1), \ldots, (X_n,Y_n) \in [0,1]\times\R$ are independent and identically distributed \textup{(}i.i.d.\textup{)}, $X_i$ has a uniform distribution over $\X$ and $\forall i = 1, \ldots, n\,$, $Y_i = X_i + \sigma(X_i) \varepsilon_i$ where $(\varepsilon_i)_{1 \leq i \leq n}$ are independent such that $\E\croch{\varepsilon_i \sachant X_i}=0$ and $\E\croch{\varepsilon_i^2 \sachant X_i} =1 \, $.
Assume moreover that $\bayes$ is twice continuously differentiable, 
\begin{gather*} 
\norm{\varepsilon_i}_{\infty} \leq E < \infty \enspace , \qquad \min\set{\carre{\siga}, \carre{\sigb}} >  0 \quad \mbox{and} \quad \carre{\siga} \neq \carre{\sigb}  \\
\mbox{where} \qquad \carre{\siga} \egaldef \int_{0}^{1/2} \paren{\sigma(x)}^2 dx \quad \mbox{and} \quad  \carre{\sigb} \egaldef \int_{1/2}^{1} \paren{\sigma(x)}^2 dx \enspace .
\end{gather*}
Let $\M_n = \Mdeuxpas_n$ be the model collection defined in Section~\ref{sec.cadre.collec}, with a maximal dimension $M_n = \sfloor{n/(\ln(n))^2}\,$.
Then, constants $\KThmDimFAILproba,\KThmDimFAILoracle>0$ and an event of probability at least $1 - \KThmDimFAILproba n^{-2}$ exist on which, for every function $F: \M_n \flens \R$ and every $\mh_F \in \argmin_{\mM_n} \set{ P_n \gamma\paren{\ERM_m} + F(D_m) } \,$,
\begin{equation}
\label{eq.shape}
\perte{\ERM_{\mh_F}} \geq \KThmDimFAILoracle \inf_{\mM_n} \set{ \perte{\ERM_m} } \qquad \mbox{with} \quad \KThmDimFAILoracle > 1 \enspace .
\end{equation}
The constant $\KThmDimFAILoracle$ may only depend on $\carre{\siga} / \carre{\sigb}$; the constant $\KThmDimFAILproba$ may only depend on $E$, $\carre{\siga}$, $\carre{\sigb}$, $\norm{\bayes^{\prime}}_{\infty}$ and $\norm{\bayes^{\prime\prime}}_{\infty}\,$.
\end{theorem}
Theorem~\ref{th.shape} is proved in Section~\ref{sec.proof.main}.
\begin{remark}
\begin{enumerate}
\item The right-hand side of \eqref{eq.shape} is of order $n^{-2/3}\,$. Hence, no oracle inequality \eqref{eq.oracle} for $\mh_F$ can be proved with a constant $\KoracleGenerique$ tending to one when $n$ tends to infinity and a remainder term $R_n \ll n^{-2/3}\,$.
\item Results similar to Theorem~\ref{th.shape} can be proved similarly with other model collections (such as the nonregular ones defined Section~\ref{sec.cadre.collec}) and with unbounded noises (thanks to concentration inequalities proved in \cite{Arl:2009:RP}). The choice $\M_n=\Mdeuxpas_n$ in the statement of Theorem~\ref{th.shape} only intends to keep the proof as simple as possible.
\end{enumerate}
\end{remark}
Theorem~\ref{th.shape} is a quite strong result, implying that no dimensionality-based penalization procedure can satisfy an oracle inequality with leading constant $1 \leq \KoracleGenerique < \KThmDimFAILoracle \,$, even a procedure using the knowledge of $\bayes$ and $\sigma\,$! The proof of Theorem~\ref{th.shape} even shows that the {\em ideal dimensionality-based model selection procedure}, defined by  
\begin{equation}
\label{eq.middim}
\middim \in \mh\paren{\Did} \qquad \mbox{where} \qquad \Did \in \argmin_{D\in\Dimset_n} \set{ \perte{ \ERM_{\mh(D)}}} \enspace ,
\end{equation}
is suboptimal with a large probability. 

The combination of Theorems~\ref{th.penVF+RP} and~\ref{th.shape} shows that when data are heteroscedastic, the price to pay for using a dimensionality-based model selection procedure instead of some resampling-based penalty is (at least) an increase of the quadratic risk by some multiplying factor $\KThmDimFAILoracle>1$ (except maybe for small sample sizes). 
Therefore, the computational cost of resampling has its counterpart in the quadratic risk.
Empirical evidence for the same phenomenon in the context of multiple change-points detection can be found in \cite{Arl_Cel:2009:segm}.

\subsection{Illustration of Theorem~\ref{th.shape}} \label{sec.main.subopt.illus}
Let us illustrate Theorem~\ref{th.shape} and its proof with a simulation experiment, called `X1--005': 
The model collection is $\paren{S_m}_{m\in\Mdeuxpas_n}$, and data $(X_i,Y_i)_{1\leq i \leq n}$ are generated according to \eqref{HP.eq.donnees.reg} with $\bayes(x)=x\,$, $\sigma(x)=(1+ 19 \times \un_{x  \leq 1/2 })/20\,$, $X_i \sim \mathcal{U}([0,1])\,$, $\varepsilon_i \sim \mathcal{N}(0,1)$ and $n=200$ data points.
An example of such data sample is plotted on the left of Figure~\ref{fig.interet.Mdeuxpas}, together with the oracle estimator $\ERM_{\mo}\,$.

Then, as remarked previously from \eqref{LL.eq.Epenid}, the ideal penalty is clearly not a function of the dimensionality (Figure~\ref{fig.X1-005.penid+path} left).
\begin{figure}
\begin{minipage}[b]{.48\linewidth}
\includegraphics[width=\textwidth]{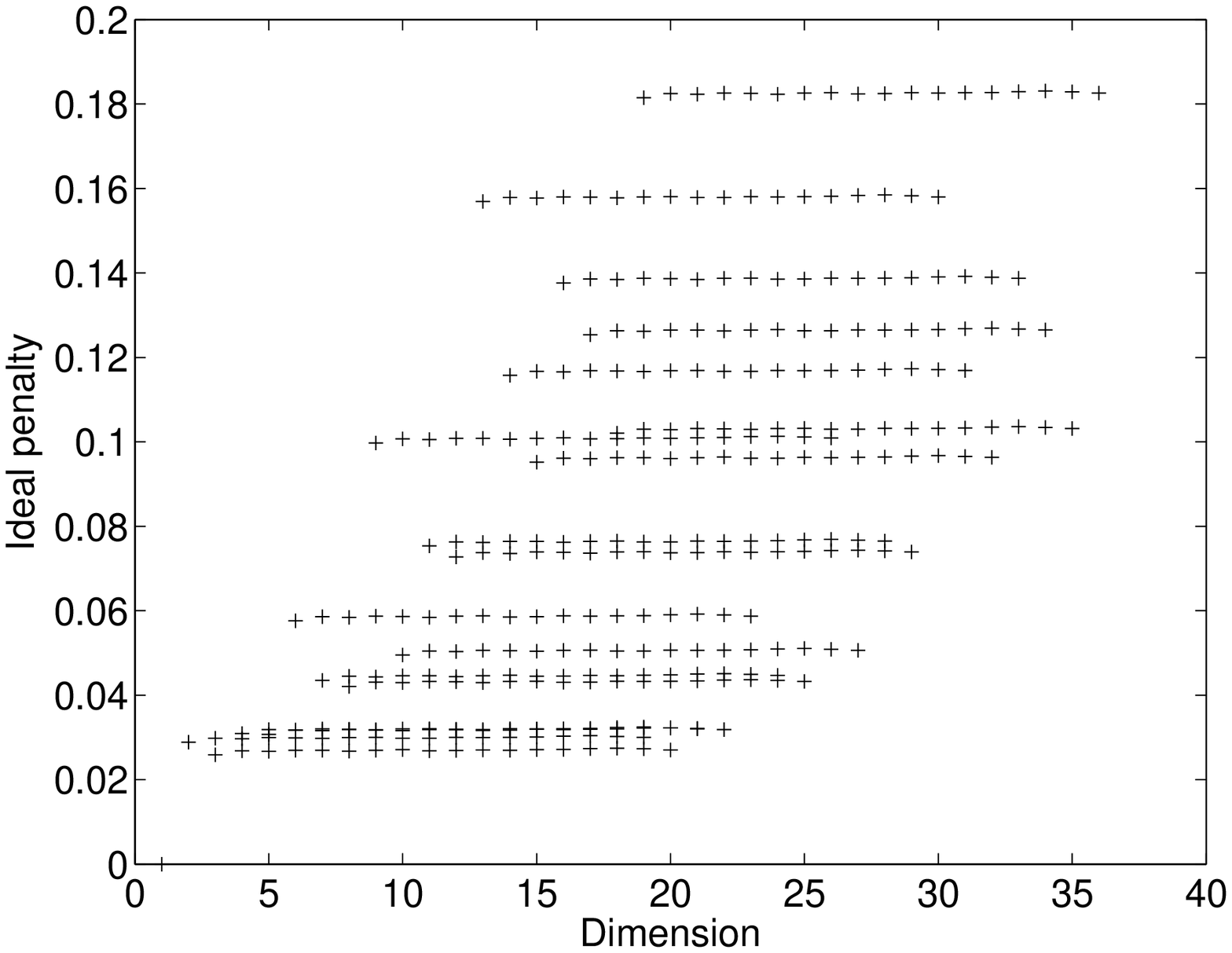}
\end{minipage}
\begin{minipage}[b]{.48\linewidth}
\includegraphics[width=\textwidth]{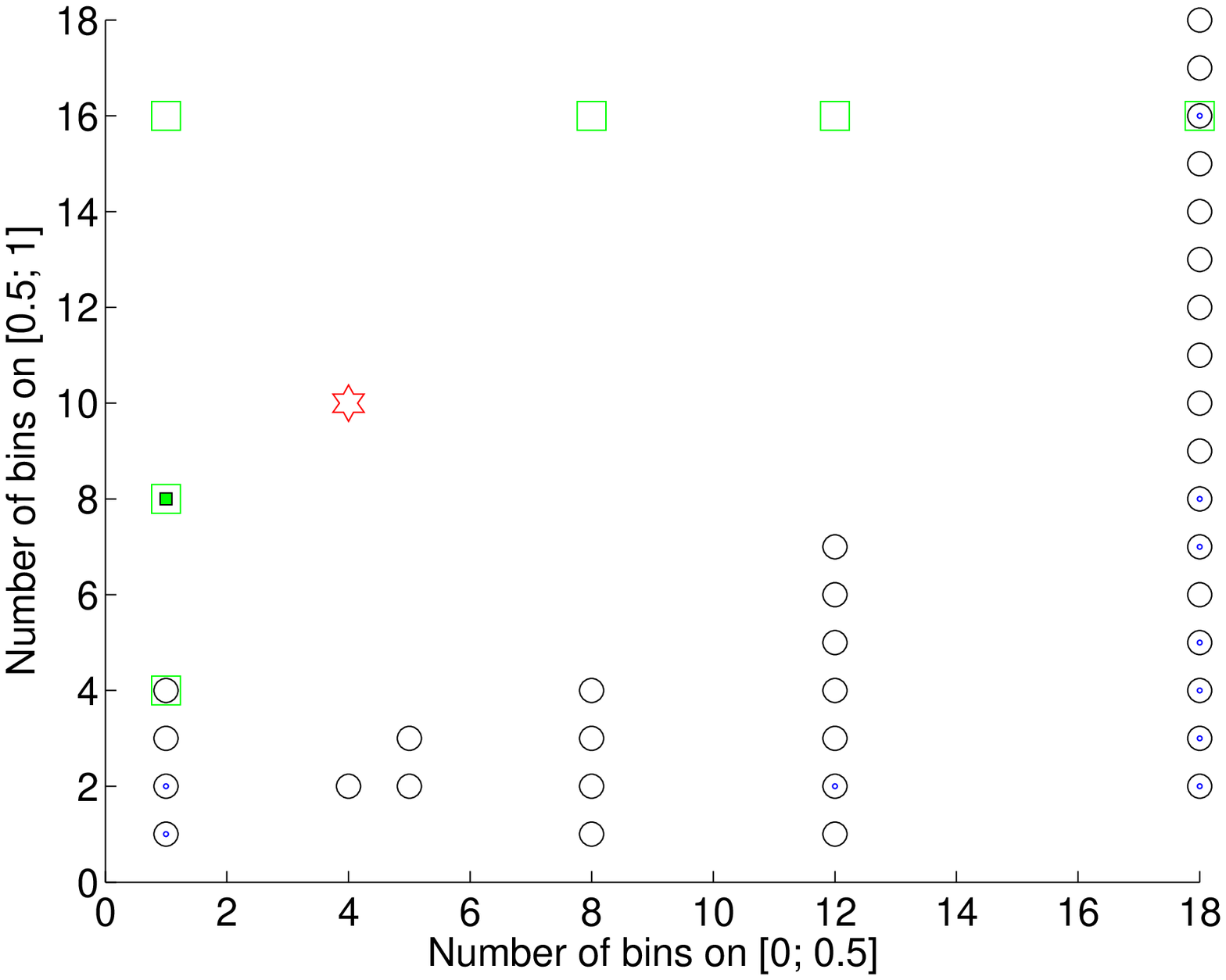}
\end{minipage}
\caption{ \label{fig.X1-005.penid+path}
Experiment X1--005.
Left: 
Ideal penalty \latin{vs.} $D_m$ for one particular sample. A similar picture holds in expectation. 
Right: 
The path of models that can be selected with penalties proportional to penLoo (green squares; the small square corresponds to pen=penLoo) is closest to the oracle (red star) than the path of models that can be selected with a dimensionality-based penalty (black circles; blue points correspond to linear penalties).
}

\end{figure}
According to \eqref{LL.eq.Epenid}, the right penalty is not proportional to $D_m$ but to $D_{m,1} \int_0^{1/2} \sigma^2(x) dx + D_{m,2} \int_{1/2}^1 \sigma^2(x) dx\,$.
The consequence of this fact is that any $m \in \Mdim$ is far from the oracle, as shown by the right of Figure~\ref{fig.X1-005.penid+path}. Indeed, minimizing the empirical risk over models of a given dimension $D$ leads to put more bins where the noise-level is larger, that is, to overfit locally (see also \cite{Arl_Cel:2009:segm} for a deeeper experimental study of this local overfitting phenomenon). 
Furthermore, Figure~\ref{fig.th1.path-density} shows that on $10\,000$ samples, $\mo$ is almost always far from $\Mdim\,$, and in particular from $\middim\,$.
\begin{figure}
\begin{minipage}[b]{\linewidth}
\includegraphics[width=\textwidth]{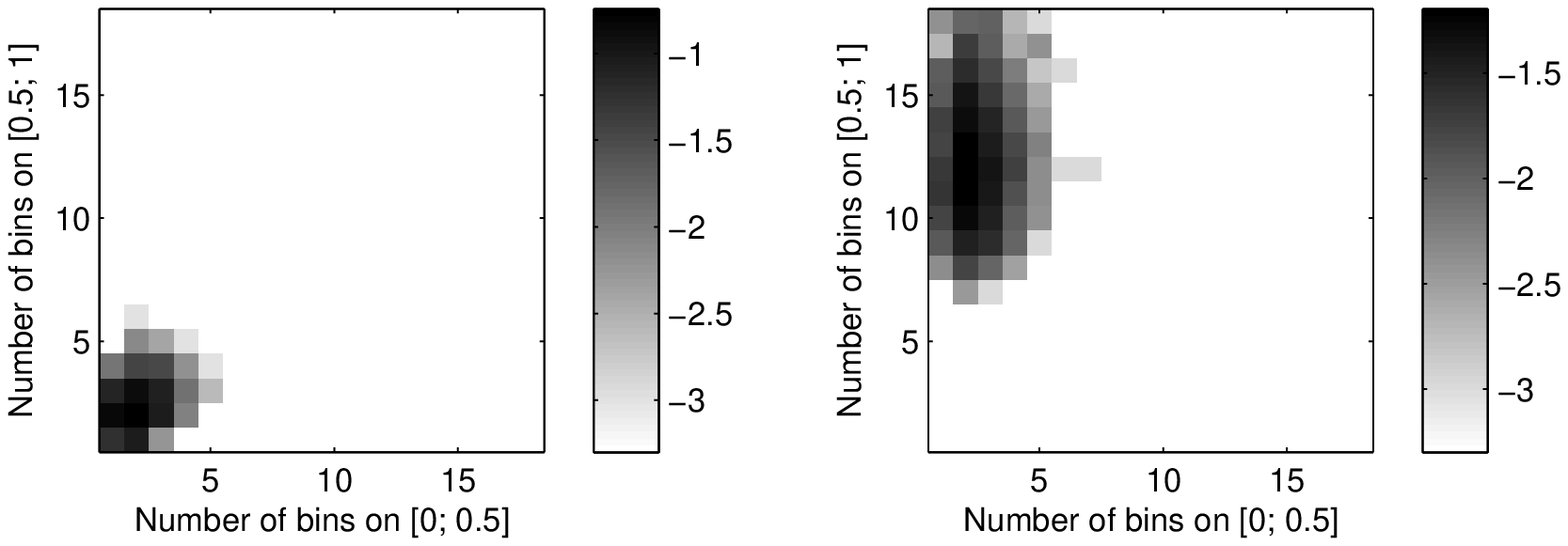}
\end{minipage}
\caption{Experiment \CadreEx. 
Left: $\log_{10} \Prob\paren{ m = \midlin}$ represented in $\R^2$ using $(D_{m,1},D_{m,2})$ as coordinates, where $\middim$ is defined by \eqref{eq.middim}; $N=10\,000$ samples have been simulated for estimating the probabilities.
Right: $\log_{10} \Prob\paren{ m = \mo}$ using the same representation and the same $N=10\,000$ samples. 
\label{fig.th1.path-density}
}
\end{figure}
On the contrary, using a resampling-based penalty (possibly multiplied by some factor $\Cov>0$) leads to avoid overfitting, and to select a model much closer to the oracle (see Figure~\ref{fig.X1-005.penid+path} right).

\subsection{Performance of linear penalties} \label{sec.main.subopt.linear}
Let us now focus on the most classical dimensionality-based penalties, that is, ``linear penalties'' of the form 
\begin{equation} \label{eq.linearpen} \pen(m) = K D_m \end{equation}
where $K>0$ can be data-dependent, but does not depend on $m$.
The first result of this subsection is that linear penalties satisfy an oracle inequality \eqref{eq.oracle} (with leading constant $\KoracleGenerique>1$) provided the constant $K$ in \eqref{eq.linearpen} is large enough. 
\begin{proposition} \label{pro.oracle.lin.Klarge}
Assume that \hypHist\ holds true. 
Then, if 
\[ \forall \mM_n \, , \quad \pen(m) = \frac{K D_m}{n} \qquad \mbox{with} \quad K > \norm{\sigma}_{\infty}^2 \enspace , \]
constants $\KProLinOKproba,\KProLinOKoracle>0$ exist such that with probability at least $1 - \KProLinOKproba n^{-2} \,$, 
\begin{equation} \label{eq.pro.oracle.lin.Klarge}
\perte{\ERM_{\mh}} \leq \KProLinOKoracle \inf_{\mM_n} \set{ \perte{\ERM_m} }
\enspace .
\end{equation}
The constant $\KProLinOKproba$ may depend on all the constants appearing in the assumptions \textup{(}that is, $\cM$, $\aM$, $c_{\mathrm{rich}}$, $A$, $\sigmin$, $\cbiasmaj$, $\cbiasmin$, $\betamin$, $\betamaj$, $\crXl$ and $K$; assuming $K \geq 2 \norm{\sigma}_{\infty}^2$, $\KProLinOKproba$ does not depend on $K$\textup{)}.
The constant $\KProLinOKoracle$ may only depend on $K$, $\sigmin$ and $\norm{\sigma}_{\infty}$; when $K \geq 2 \norm{\sigma}_{\infty}^2$, $\KProLinOKoracle$ can be made as close as desired to $K \sigmin^{-2} - 1$ at the price of enlarging $\KProLinOKproba\,$.
\end{proposition}
Proposition~\ref{pro.oracle.lin.Klarge} is proved in Section~\ref{sec.proof.pro.oracle.lin.Klarge}.
As a consequence of Proposition~\ref{pro.oracle.lin.Klarge}, if we can afford loosing a constant factor of order $\norm{\sigma}_{\infty}^2/\sigmin^2$ in the quadratic risk, a relevant (and computationally cheap) strategy is the following: 
First, estimate an upper bound on $\norm{\sigma}_{\infty}^2\,$. 
Second, plug it into the penalty $\pen(m) = K^{\prime} \norm{\sigma}_{\infty}^2 D_m / n\,$, where $K^{\prime}>1$ remains to be chosen. 
According to the proof of Proposition~\ref{pro.oracle.lin.Klarge} and previous results in the homoscedastic case, $K^{\prime}=2$ should be a good choice in general.
Let us now add a few comments.
\begin{remark}
\begin{enumerate}
\item The assumption set \hypHist\ is discussed in Remark~\ref{rk.hypHist} in Section~\ref{sec.urep.RP}, and can be relaxed in various ways, see \cite{Arl:2009:RP}.
\item Proposition~\ref{pro.oracle.lin.Klarge} can be generalized to other models than sets of piecewise constant functions. Indeed, the keystone of the proof of Proposition~\ref{pro.oracle.lin.Klarge} is that $2 \sigmin^2 \leq n D_m^{-1} \E\croch{\penid(m)} \leq 2 \norm{\sigma}_{\infty}^2$, which holds for instance when the design is fixed and the models $S_m$ are finite dimensional vector spaces. 
Therefore, using arguments similar to the ones of \cite{Bar_Bir_Mas:1999,Bar:2000} for instance, an oracle inequality like \eqref{eq.pro.oracle.lin.Klarge} can be proved when models are general vector spaces, assuming that the design $(X_i)_{1 \leq i \leq n}$ is deterministic. 
\end{enumerate}
\end{remark}

\medskip

The second result of this subsection shows that the condition $K>\norm{\sigma}_{\infty}^2/n$ in Proposition~\ref{pro.oracle.lin.Klarge} really prevents from strong overfitting. In particular, some example can be built where $C_p$ strongly overfits.
\begin{proposition} \label{pro.overfit.Mal}
Let us consider the framework of Section~\ref{HP.sec.regression} with $\X = [0, 1]$, $\M_n = \Mdeuxpas_n \,$ with maximal dimension $M_n = \sfloor{n/(\ln(n))}$, and assume that:
\begin{itemize}
\item \hypAb\ and \hypAn\ hold true \textup{(}see the definition of \hypHist\textup{)}, 
\item $\bayes \in \mathcal{H}(\alpha,R)$, that is, $\forall x_1, x_2 \in \X$, $\absj{\bayes(x_2) - \bayes(x_1)} \leq R \absj{x_2 - x_1}^{\alpha} \,$, for some $R>0$ and $\alpha \in (0,1]\,$,
\item $\mu = \P(X \in [0,1/2]) \in (0,1) \, $,
\item conditionally on $\set{ X \in [0,1/2] }$, $X$ has a density w.r.t. the Lebesgue measure which is lower bounded by $c_{X,\Leb}>0 \,$.
\end{itemize}
If in addition $\forall\mM_n \, $, $\pen(m) = K D_m / n$ with $K < \inf_{t \in [0,1/2]} \set{ \sigma(t)^2 } \,$, 
then constants $\KProCpFAILproba, \KProCpFAILdim, \KProCpFAILoracle>0$ exist such that with probability at least $1 - \KProCpFAILproba n^{-2}$, 
\begin{gather} \label{eq.pro.overfit.Mal.dim}
D_{\mh} \geq D_{\mh,1} \geq \frac{ \KProCpFAILdim n } { \ln(n) } \\ \label{eq.pro.overfit.Mal.risk}
\mbox{and} \quad 
\perte{\ERM_{\mh}} \geq \frac{\KProCpFAILoracle}{\paren{ \ln(n) }^2} \geq \ln(n) \inf_{\mM_n} \set{ \perte{\ERM_m} } \enspace . 
\end{gather}
The constants $\KProCpFAILproba$, $\KProCpFAILdim$, $\KProCpFAILoracle$ may depend on $A$, $\sigmin$, $\alpha$, $R$, $\mu$, $c_{X,\Leb}$, $\inf_{[0,1/2]} \set{\sigma^2}$ and $K$, but they do not depend on $n$.
\end{proposition}
Proposition~\ref{pro.overfit.Mal}, which is actually is a corollary of a more general result on minimal penalties---Theorem~2 in \cite{Arl_Mas:2009:pente}---is proved in Section~\ref{sec.proof.pro.overfit.Mal}.

Consider in particular the following example: $X \in [0,1]$ with a density w.r.t. $\Leb$ equal to $2\mu \un_{[0,1/2]} + 2 (1-\mu) \un_{(1/2,1]}$ for some $\mu \in (0,1)$ and $\sigma=\sigma_a \un_{[0,1/2]} + \sigma_b \un_{(1/2,1]}$ for some $\sigma_a \geq \sigma_b > 0 \,$. 
Then, the penalty $K D_m/n$ leads to overfitting as soon as $K < \sigma_a^2 = \norm{\sigma}_{\infty}^2 \, $, which shows that the lower bound $K > \norm{\sigma}_{\infty}^2$ appearing in the proof of Proposition~\ref{pro.oracle.lin.Klarge} cannot be improved in this example. 

Let us now consider $C_p$, that we naturally generalize to the heteroscedastic case by $\penMal(m) = K D_m / n$ with $K = 2 \E\croch{\sigma(X)^2}\,$. In the above example, $K = 2 \mu \sigma_a^2 + 2 (1-\mu) \sigma_b^2 \,$. So, the condition on $K$ in Proposition~\ref{pro.overfit.Mal} can be written
\[ \mu + (1-\mu) \frac{\sigma_b^2}{\sigma_a^2} < \frac{1}{2} \enspace ,  \]
which holds when 
\begin{equation} \label{eq.linepen.Cpoverfit.con} \mu \in (0,1/2) \quad \mbox{and} \quad \frac{\sigma_b^2}{\sigma_a^2} < \frac{\frac{1}{2} - \mu}{1 - \mu} \enspace . \end{equation}
Therefore, when \eqref{eq.linepen.Cpoverfit.con} holds, Proposition~\ref{pro.overfit.Mal} shows that $C_p$ strongly overfits. 

The conclusion of this subsection is that some linear penalties can be used with heteroscedastic data provided that we can estimate $\norm{\sigma}_{\infty}^2$ by some $\widehat{\sigma^2}_{\infty}$ such that $2 \widehat{\sigma^2}_{\infty} > \norm{\sigma}_{\infty}^2$ holds with probability close to 1.
Then the price to pay is an increase of the quadratic risk by a constant factor of order $\max(\sigma^2)/ \min(\sigma^2)$ in general.

\section{Simulation study} \label{sec.simus}
This section intends to compare by a simulation study the finite sample performances of the model selection procedures studied in the previous sections: dimensionality-based and resampling-based procedures.

\subsection{Experiments} 
We consider four experiments, called `X1--005' (as in Section~\ref{sec.main.subopt.illus}), `X1--005$\mu$02', `S0--1' and `XS1--05'. 
Data $(X_i,Y_i)_{1\leq i \leq n}$ are generated according to \eqref{HP.eq.donnees.reg} with $\varepsilon_i \sim \mathcal{N}(0,1)$ and $X_i$ has density w.r.t. $\Leb([0,1])$ of the form $2 \mu \un_{[0,1/2]} + 2 (1-\mu) \un_{(1/2,1]} \,$, where $\mu = \Prob(X_1 \leq 1/2) \in (0,1)\,$.
The functions $\bayes$ and $\sigma\,$, and the values of $n$ and $\mu$, depend on the experiment, see Table~\ref{tab.expts} and Figure~\ref{fig.reg-samples}; in experiment `XS1--05', the regression function is given by 
\begin{equation} \label{eq.bayes.XS1-05}
\bayes(x) = \frac{x}{4} \un_{x \leq 1/2} + \croch{ \frac{1}{8} + \frac{2}{3} \sin\paren{16 \pi x }} \un_{x > 1/2} \enspace .
\end{equation}
In each experiment, $N=10\,000$ independent data samples are generated, and the model collection is $\paren{S_m}_{m\in\Mdeuxpas_n}\,$, with different values of $M_n$ for computational reasons, see Table~\ref{tab.expts}.
The signal-to-noise ratio is rather small in the four experimental settings considered here, and the collection of models is quite large ($\card(\M_n) = 1 + (M_n/2)^2\,$). Therefore, we can expect overpenalization to be necessary (see Section~6.3.2 of \cite{Arl:2009:RP} for more details on overpenalization).

\begin{table} 
\caption{Parameters of the four experiments. \label{tab.expts}}
\begin{center}
\begin{tabular}
{p{0.14\textwidth}@{\hspace{0.025\textwidth}}p{0.20\textwidth}@{\hspace{0.025\textwidth}}p{0.18\textwidth}@{\hspace{0.025\textwidth}}p{0.18\textwidth}@{\hspace{0.025\textwidth}}p{0.20\textwidth}}
\hline\noalign{\smallskip}
Experiment      & X1--005                      & S0--1              & XS1--05           & X1--005$\mu$02     \\
\noalign{\smallskip}
\hline
\noalign{\smallskip}
$\bayes(x)$     & $x$                          & $\sin(\pi x)$      & see Eq.~\eqref{eq.bayes.XS1-05} & $x$ \\
$\sigma(x)$     & $(1+19\times\un_{x\leq 1/2})/20$ & $\un_{x  > 1/2}$ & $(1+\un_{x  \leq 1/2})/2$ & $(1+19\times \un_{x\leq 1/2})/20$  \\
$n$             & $200$                        & $200$              & $500$                 & $1000$ \\
$\Prob(X\leq 1/2)$ & $1/2$                     & $1/2$              & $1/2$                 & $1/5$ \\
$M_n$           & $\sfloor{n/(\ln(n))}=37$     & $\sfloor{n/(\ln(n))}=37$ & $\sfloor{n/(\ln(n))}=80$ & $\sfloor{n/(\ln(n))^2}=20$ \\
\hline
\end{tabular} 
\end{center}
\end{table}

\begin{figure}
\begin{center}
\begin{minipage}[b]{.46\linewidth}
\includegraphics[width=\textwidth]{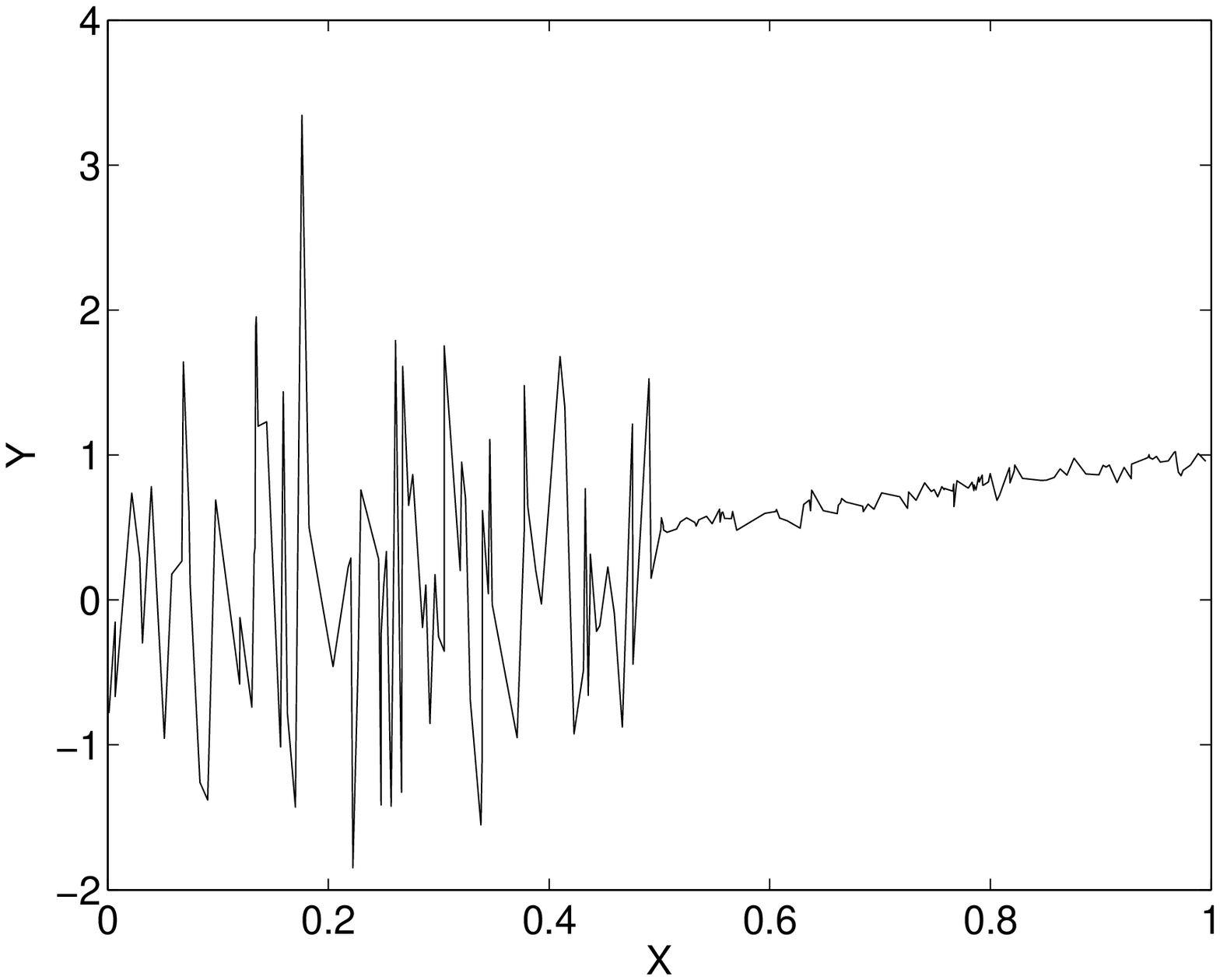}
\end{minipage}
\begin{minipage}[b]{.46\linewidth}
\includegraphics[width=\textwidth]{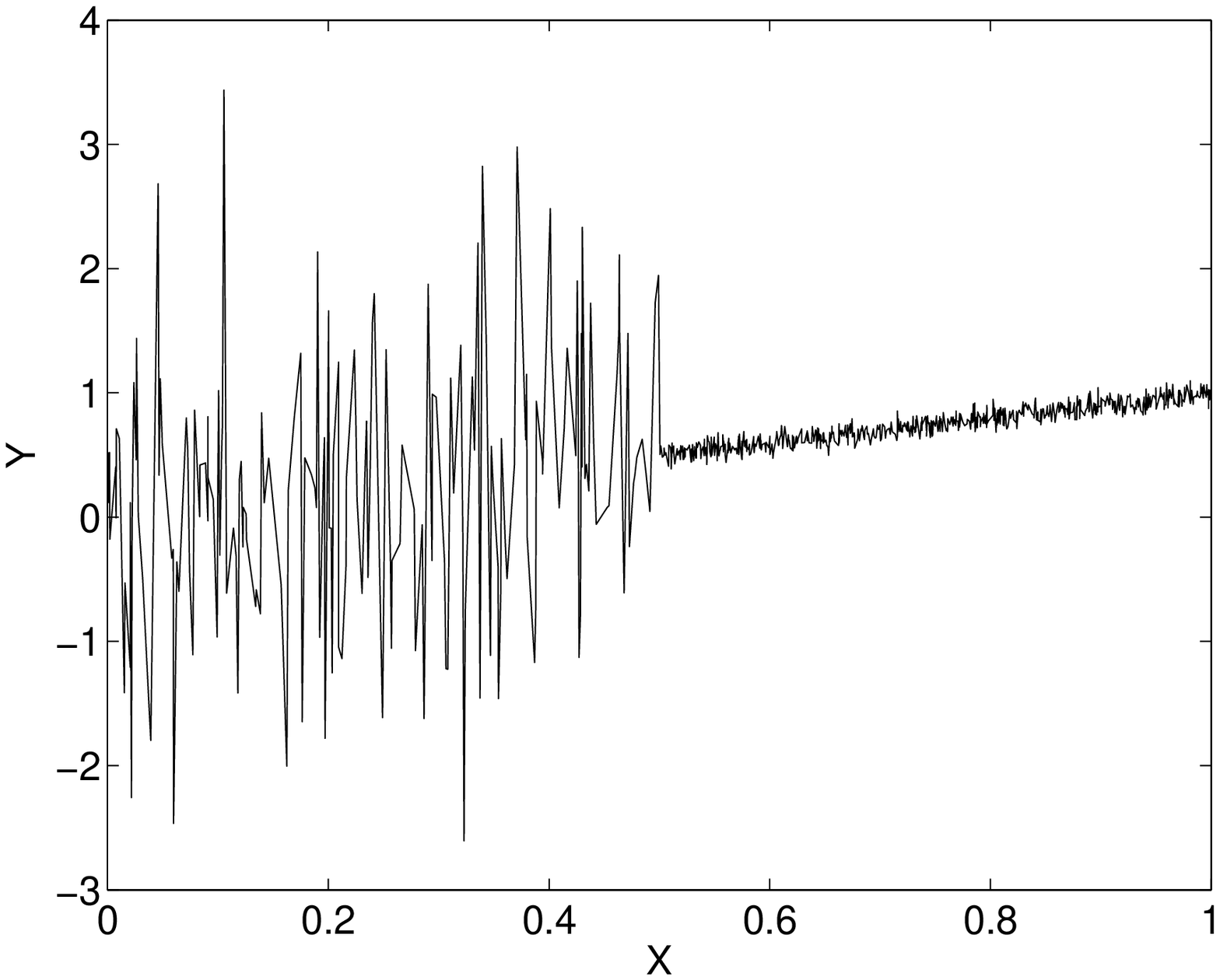}
\end{minipage} \hfill \\
Experiments X1--005 (left) and X1--005$\mu$02 (right): one data sample \vspace{0.2cm} \\
\begin{minipage}[b]{.46\linewidth}
\includegraphics[width=\textwidth]{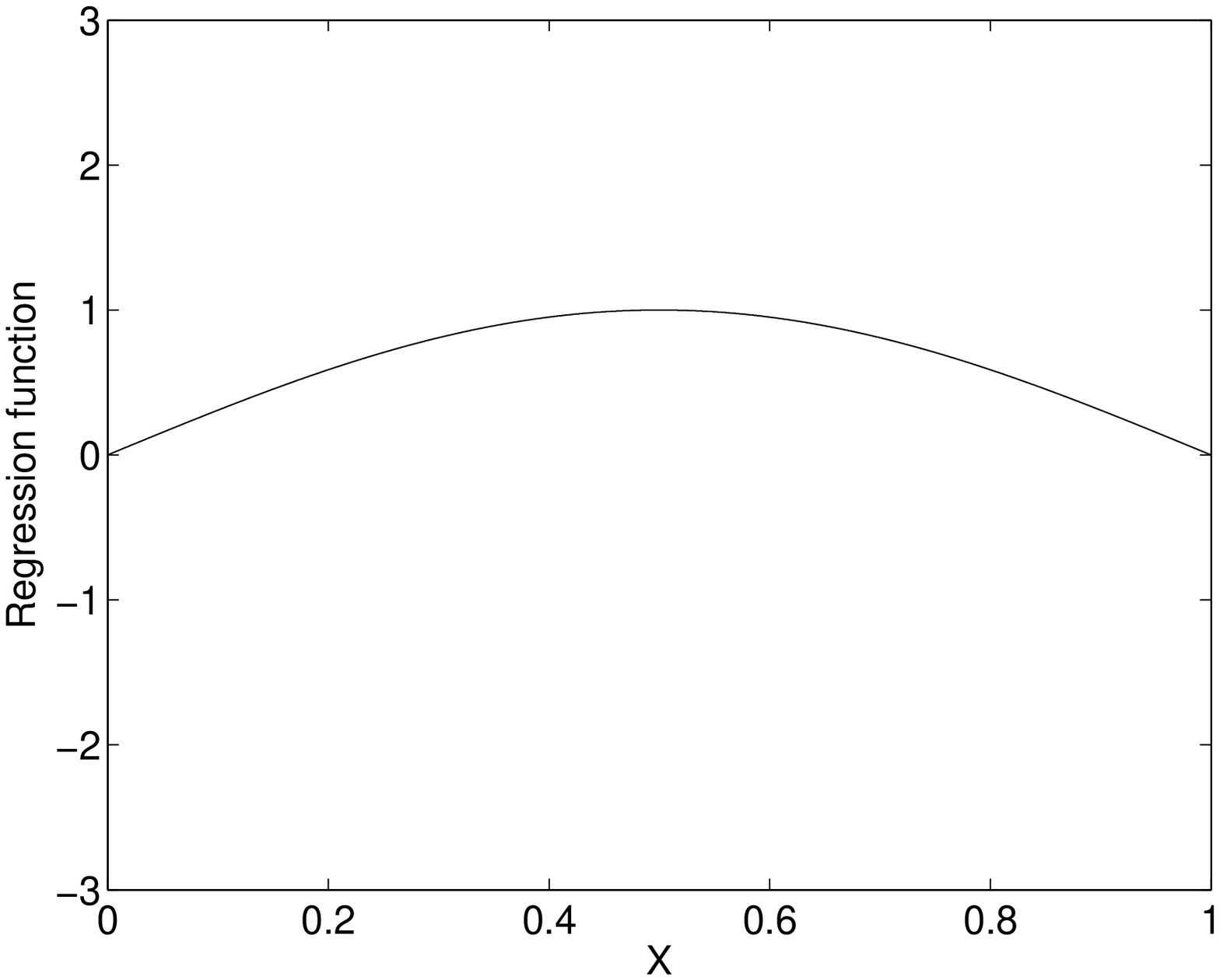}
\end{minipage}
\begin{minipage}[b]{.46\linewidth}
\includegraphics[width=\textwidth]{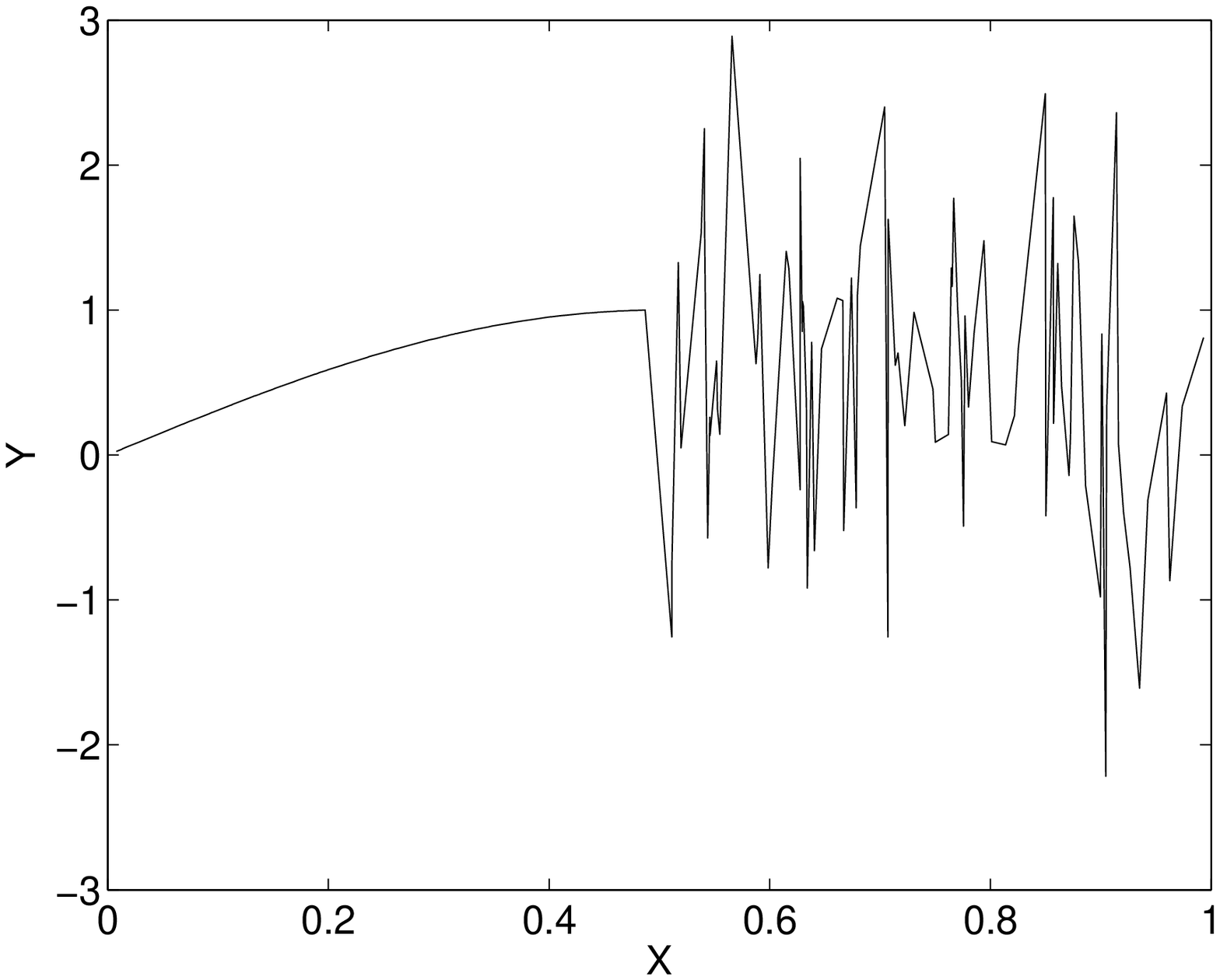}
\end{minipage} \hfill \\
Experiment S0--1 (left: regression function; right: one data sample)\vspace{0.2cm} \\
\begin{minipage}[b]{.46\linewidth}
\includegraphics[width=\textwidth]{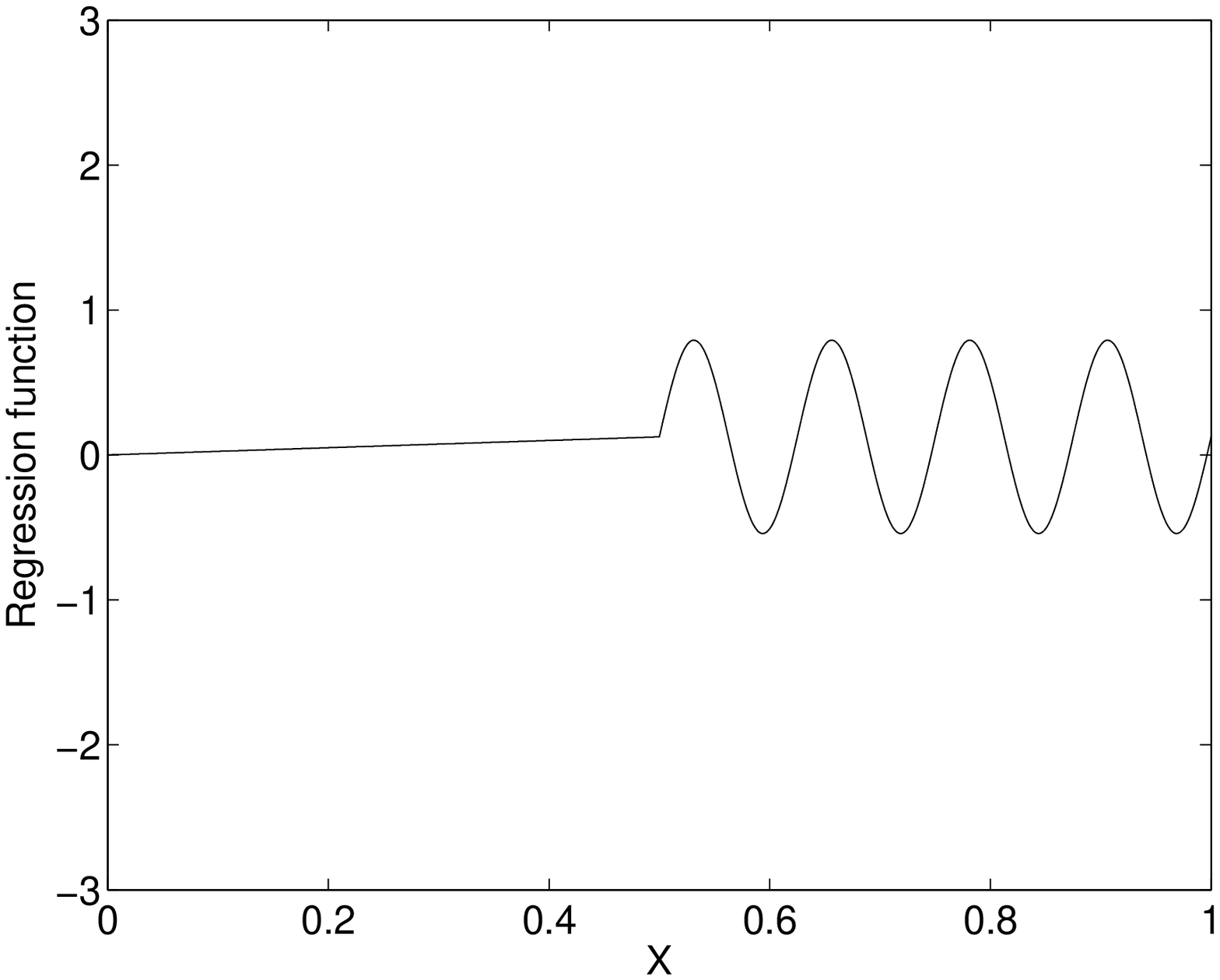}
\end{minipage}
\begin{minipage}[b]{.46\linewidth}
\includegraphics[width=\textwidth]{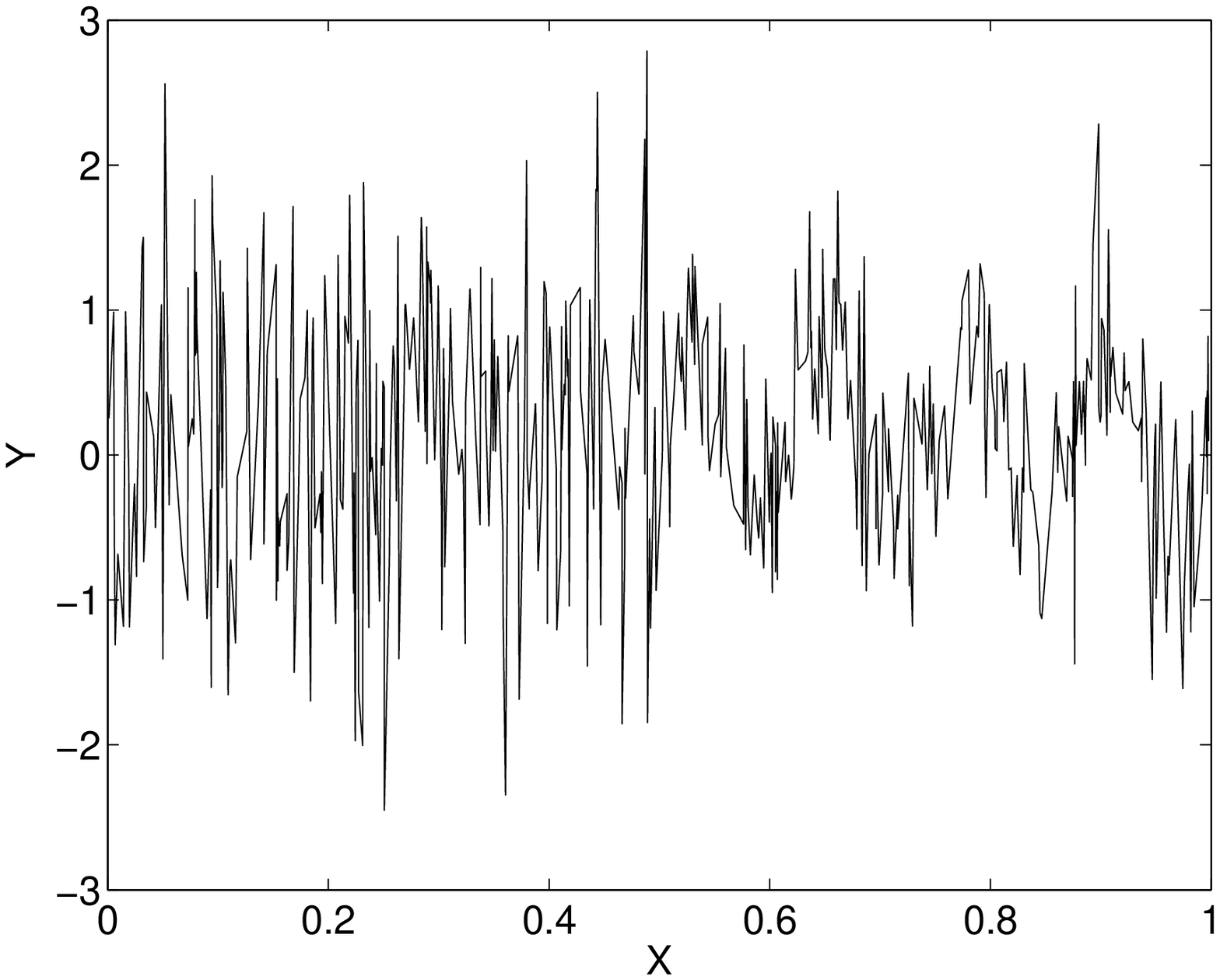}
\end{minipage} \hfill \\
Experiment XS1--05 (left: regression function; right: one data sample)
\end{center}
\caption{Regression functions and one particular data sample for the four experiments. \label{fig.reg-samples}}
\end{figure}

\medskip

\subsection{Procedures compared}
For each sample, the following model selection procedures are compared, where 
$\tau$ denotes a permutation of $\set{1, \ldots, n}$ such that $\paren{X_{\tau(i)}}_{1 \leq i \leq n}$ is nondecreasing.
\begin{itemize}
\item[(A)] Epenid: penalization with 
\[ \pen(m)=\E\croch{\penid(m)} \quad \mbox{ where } \quad \penid(m)=P\gamma\sparen{\ERM_m} - P_n\gamma\sparen{\ERM_m} \enspace , \]
as defined in Section~\ref{HP.sec.cadre.mod_selec}. This procedure makes use of the knowledge of the true distribution $P$ of data. Its model selection performances witness what performances could be expected (ideally) from penalization procedures adapting to heteroscedasticity.
\item[(B)] MalEst: penalization with $\penMal(m)$ where the variance is estimated as in Section~6 of \cite{Bar:2000}, that is, 
\[ \pen(m) = \frac{2 \sighsq D_m}{n} \quad \mbox{with} \quad \sighsq \egaldef \frac{1}{n} \sum_{i=1}^{n/2} \paren{Y_{\tau(2i)} - Y_{\tau(2i-1)}}^2 \enspace , \]
Replacing $\sighsq$ by $\E\croch{\sigma(X)^2}$ doesn't change much the performances of MalEst, see \refapp.
\item[(C)] MalMax: penalization with $\pen(m)=2\norm{\sigma}_{\infty}^2 D_m / n$ (using the knowledge of $\sigma$).
\item[(D)] HO: hold-out procedure, that is, 
\[ \mh \in \arg\min_{\mM_n} \set{P_n^{(I^c)} \gamma\paren{\ERM_m^{(I)}}} \enspace , \]
where $P_n^{(I^c)}$ and $\ERM^{(I)}$ are defined as in Section~\ref{sec.urep.RP}, and $I \subset\set{1, \ldots, n}$ is uniformly chosen among subsets of size $n/2$ such that $\forall k \in \set{1, \ldots, n/2}\,$, $\card(I\cap \set{\tau(2k-1),\tau(2k)})=1\,$.
\item[(E--F--G)] VFCV ($V$-fold cross-validation) with $V=2$, $5$ and $10\,$:
\[ \mh \in \arg\min_{\mM_n} \set{ \frac{1}{V} \sum_{j=1}^V P_n^{(B_j)} \gamma\paren{\ERM_m^{(B_j^c)}}} \enspace , \]
where $(B_j)_{1 \leq j \leq V}$ is a regular partition of $\set{1, \ldots, V}\,$, uniformly chosen among partitions such that 
$\forall j \in \set{1, \ldots, V}\,$, $\forall k \in \set{1, \ldots, n/V}\,$, 
$\card(B_j \cap \set{\tau(i) \telque kV-V+1 \leq i \leq kV})=1\,$.
\item[(H)] penHO: hold-out penalization, as defined in Section~\ref{sec.urep.RP}, with the same training set $I$ as in procedure (D).
\item[(I--J--K)] penVF ($V$-fold penalization) with $V=2$, $5$ and $10\,$, as defined by \eqref{eq.penVF}, with the same partition $(B_j)_{1 \leq j \leq V}$ as in procedures (E--F--G) respectively.
\item[(L)] penLoo (Leave-one-out penalization), that is, $V$-fold penalization with $V=n$ and $\forall j\,$, $B_j=\set{j}\,$.
\end{itemize}
Every penalization procedure was also performed with various overpenalization factors $\Cov \geq 1\,$, that is, with $\pen(m)$ replaced by $\Cov \times \pen(m)\,$. Only results with $\Cov \in \set{1, 2, 4}$ are reported in the paper since they summarize well the whole picture.
\begin{table} 
\caption{Short names for the procedures compared. \label{tab.proc}}
\begin{center}
\begin{minipage}[b]{.45\linewidth}
\begin{center}
\begin{tabular}
{p{0.30\textwidth}@{\hspace{0.05\textwidth}}p{0.30\textwidth}@{\hspace{0.05\textwidth}}p{0.30\textwidth}}
\hline\noalign{\smallskip}
A: Epenid & E: 2-FCV  & I: pen2-F \\
B: MalEst & F: 5-FCV  & J: pen5-F \\
C: MalMax & G: 10-FCV & K: pen10-F \\
D: HO     & H: penHO  & L: penLoo \\
\hline
\end{tabular} 
\end{center}
\end{minipage}
\end{center}
\end{table}

Furthermore, given each penalization procedure among the above (let us call it `Pen'), we consider the associated ideally calibrated penalization procedure `IdPen', which is defined as follows:
\begin{gather*}
\mideale_{\pen} = \mh_{\pen}\paren{ \Kid_{\pen} } 
\quad \mbox{where} \quad 
\forall K \geq 0 \, , \quad \mh_{\pen}(K) \in \arg\min_{\mM_n} \set{ P_n\gamma\paren{\ERM_m} + K \pen(m) } \\
\mbox{and} \quad \Kid_{\pen} \in \arg\min_{K \geq 0} \set{ \perte{\ERM_{\mh_{\pen}(K)}} } \enspace . 
\end{gather*}
In other words, $\pen(m)$ is used with the best distribution and data-dependent overpenalization factor $\Kid_{\pen}\,$.
Needless to say, `IdPen' makes use of the knowledge of $P\,$, and is only considered for experimental comparison.

When $\pen(m)=D_m\,$, the above definition defines the {\em ideal linear penalization procedure}, that we call `IdLin' (and the selected model is denoted by $\midlin$).
In addition, we consider the ideal dimensionality-based model selection procedure `IdDim', 
defined by \eqref{eq.middim}.

Finally, let us precise that in all the experiments, prior to performing any model selection procedure, models $S_m$ such that $\min_{\lamm} \card\set{i \telque X_i \in \Il} < 2$ are removed from $\paren{S_m}_{\mM_n}\,$.
Without removing interesting models, this preliminary step intends to provide a fair and clear comparison between penLoo (which was defined in \cite{Arl:2009:RP} including this preliminary step) and other procedures.

\medskip

The benchmark for comparing model selection performances of the procedures is
\begin{equation} \label{VFCV.def.Cor}
\Cor \egaldef \frac{ \E\croch{ \perte{\ERM_{\mh}} }} {\E\croch{ \inf_{\mM_n}  \perte{\ERM_m} }}  \enspace ,
\end{equation}
where both expectations are approximated by an average over the $N$ simulated samples.
Basically, $\Cor$ is the constant that should appear in an oracle inequality \eqref{eq.oracle} holding in expectation with $R_n=0$.
We also report the following uncertainty measure of our estimator of $\Cor\,$,
\begin{equation} \label{def.epsCor}
\epsCor \egaldef \frac{ \sqrt{ \var\paren{ \perte{\ERM_{\mh}} } } } { \sqrt{N} \E\croch{ \inf_{\mM_n}  \perte{\ERM_m} } } \enspace , 
\end{equation}
where $\var$ (resp. $\E$) is approximated by an empirical variance (resp. expectation) over the $N$ simulated samples.

\subsection{Results}

The (evaluated) values of $\Cor \pm \epsCor$ in the four experiments are given on Figures~\ref{fig.res.Cor-div1} and~\ref{fig.res.Cor-div2} (for procedures A--L) and in Table~\ref{LL.tab.deux} (for IdDim and the ideally calibrated penalization procedures).
In addition, results for experiment `X1--005' with various values of the sample size $n$ are presented on Figure~\ref{fig.X1-005.nvar}.
Note that a few additional results are provided in \refapp.
\begin{figure}
\begin{center}
\includegraphics[height=0.87\textwidth,angle=-90]{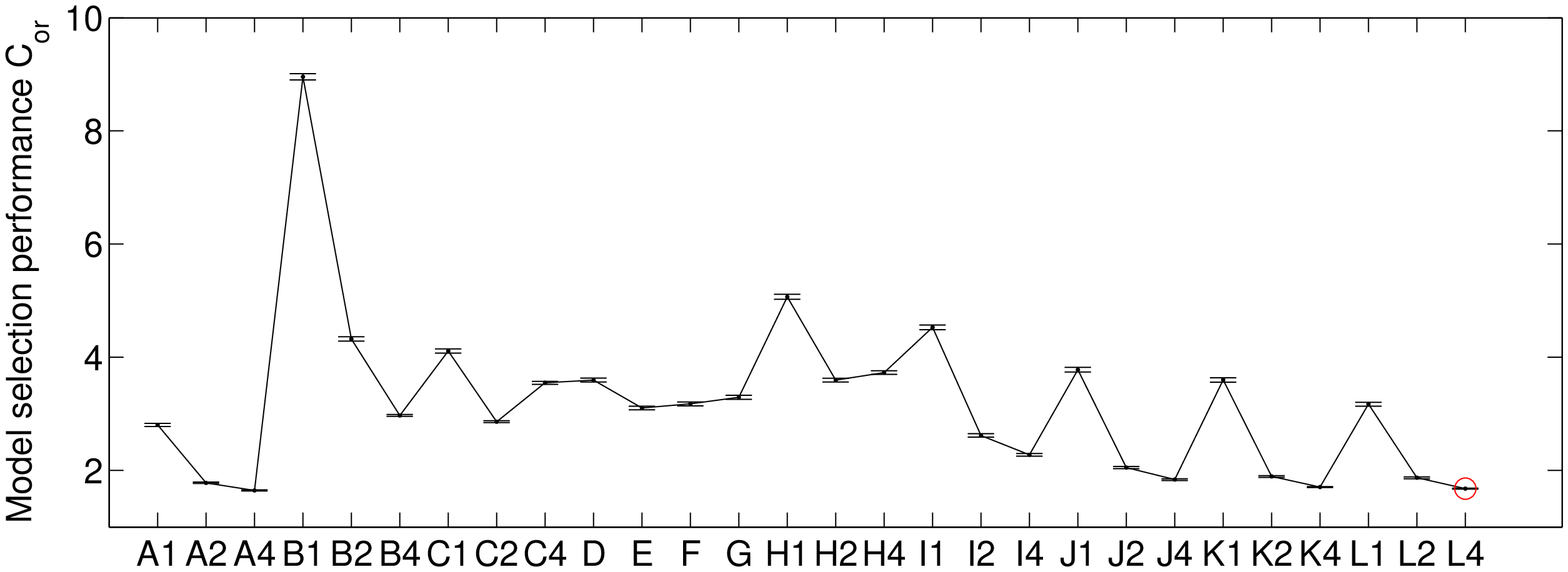}
\end{center}
\caption{Accuracy indices $\Cor$ in Experiment X1--005 for various model selection procedures. Error bars represent $\epsCor\,$. $\Cor$ is defined by \eqref{VFCV.def.Cor} and $\epsCor$ by \eqref{def.epsCor}. Red circles show which are the most accurate procedures, taking uncertainty into account and excluding procedures usign $\E\scroch{\penid}\,$. 
On the x-axis, procedures are named with a letter (whose meaning is given in Table~\ref{tab.proc}), plus a figure (for penalization procedures) equal to the value of the overpenalization constant $\Cov\,$; for instance, J2 means the 5-fold penalty multiplied by $2$. 
See also \refTabFigapp. 
\label{fig.res.Cor-div1}}
\end{figure}
\begin{figure}
\begin{center}
\includegraphics[height=0.87\textwidth,angle=-90]{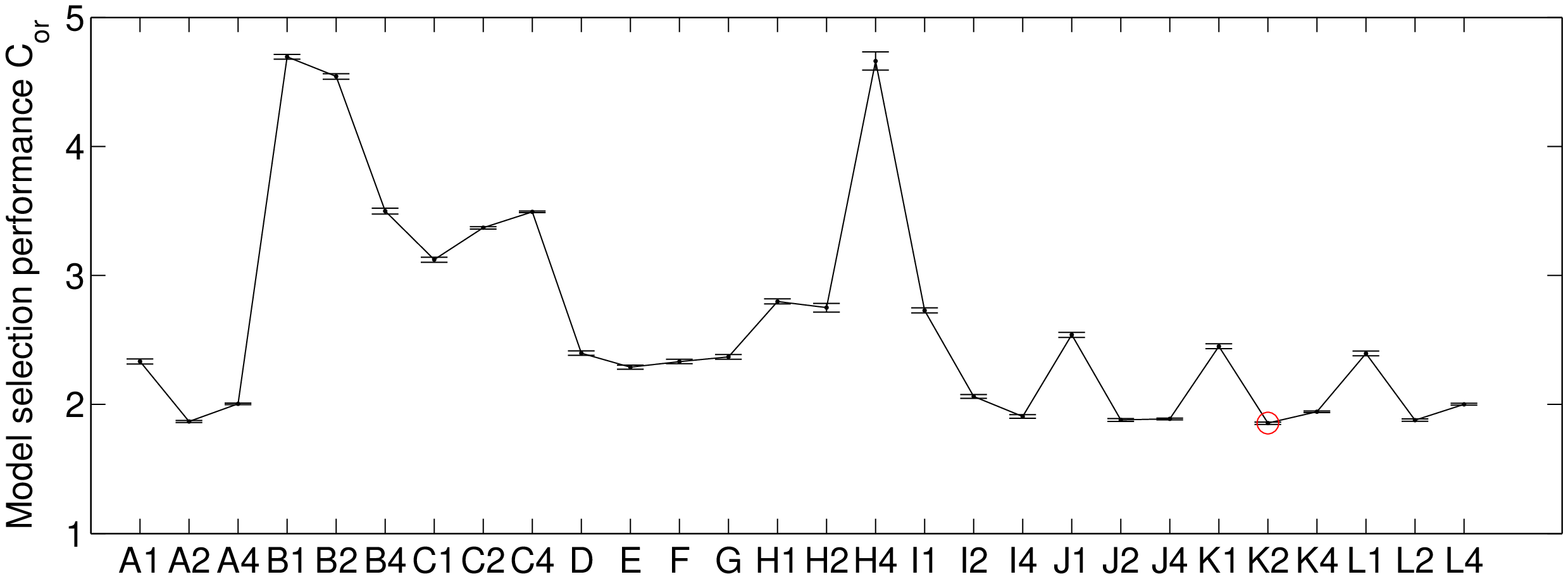}
\hfill \vspace{0.2cm} \\ Experiment X1--005$\mu$02 \vspace{0.2cm} \\
\includegraphics[height=0.87\textwidth,angle=-90]{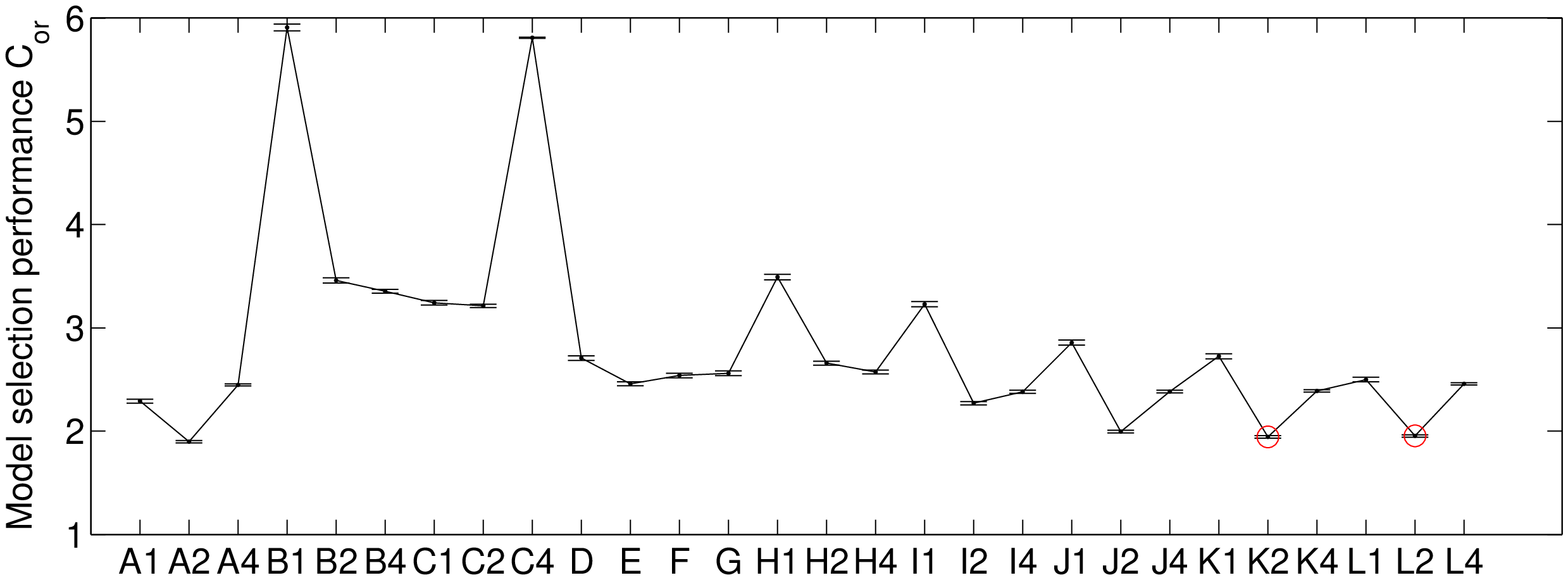}
\hfill \vspace{0.2cm} \\ Experiment S0--1 \vspace{0.2cm} \\
\includegraphics[height=0.87\textwidth,angle=-90]{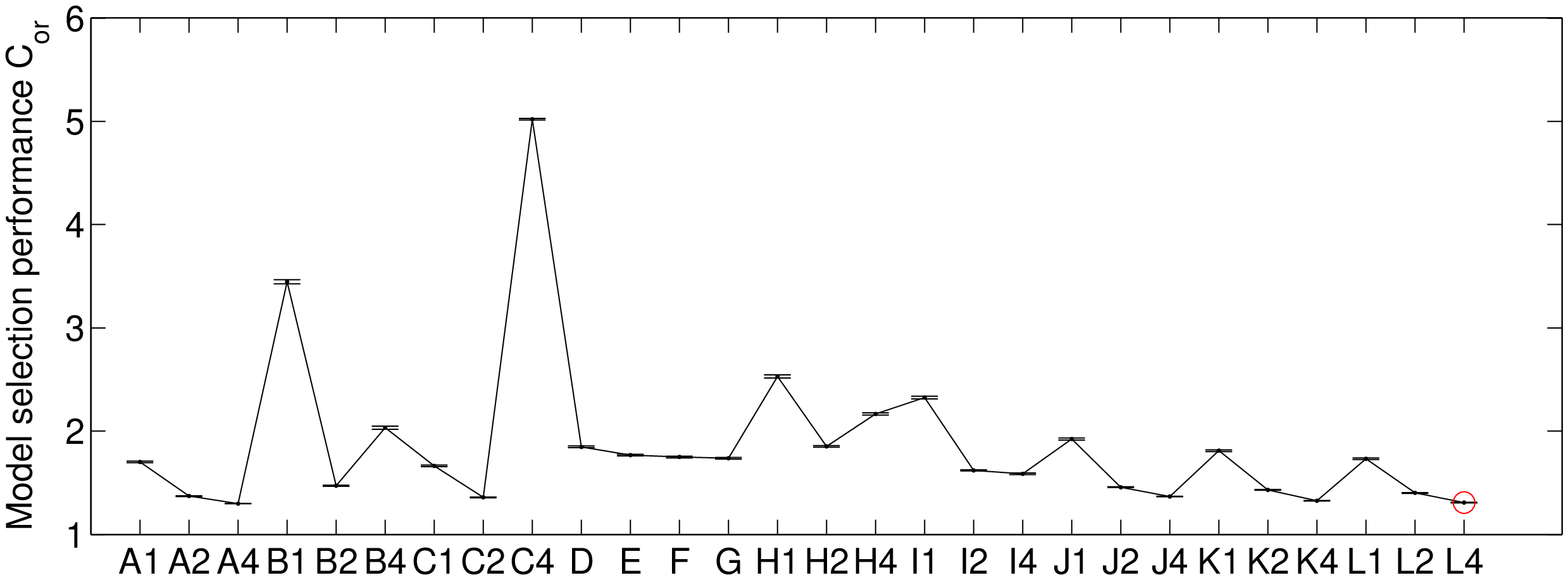}
\hfill \vspace{0.2cm} \\ Experiment XS1--05
\end{center}
\caption{Same as Figure~\ref{fig.res.Cor-div1} with the three other experiments. See also \refTabFigapp. \label{fig.res.Cor-div2}}
\end{figure}
\begin{table} 
\caption{Accuracy indices $\Cor$ for ``ideal'' procedures in four experiments, $\pm\epsCor\,$. See also \refFigappid.
\label{LL.tab.deux}}
\begin{center}
\begin{tabular}
{p{0.16\textwidth}@{\hspace{0.025\textwidth}}p{0.18\textwidth}@{\hspace{0.025\textwidth}}p{0.17\textwidth}@{\hspace{0.025\textwidth}}p{0.17\textwidth}@{\hspace{0.025\textwidth}}p{0.17\textwidth}}
\hline\noalign{\smallskip}
Experiment      & X1--005                  & S0--1                    & XS1--05                  & X1--005$\mu$02     \\
\noalign{\smallskip}
\hline
\noalign{\smallskip}
IdLin           & $2.065 \pm 0.010$        & $2.106 \pm 0.009$        & $1.308 \pm 0.002$        & $2.211 \pm 0.009$ \\
IdDim           & $1.507 \pm 0.009$        & $1.595 \pm 0.008$        & $1.262 \pm 0.002$        & $1.683 \pm 0.008$ \\
\noalign{\smallskip}
\hline
\noalign{\smallskip}
IdPenHO         & $2.158 \pm 0.020$        & $1.785 \pm 0.012$        & $1.509 \pm 0.005$        & $1.767 \pm 0.012$ \\
IdPen2F         & $1.454 \pm 0.011$        & $1.523 \pm 0.009$        & $1.303 \pm 0.003$        & $\meil{1.410 \pm 0.008}$ \\
IdPen5F         & $\meil{1.377 \pm 0.008}$ & $1.467 \pm 0.008$        & $1.244 \pm 0.002$        & $\meil{1.413 \pm 0.007}$ \\
IdPen10F        & $\meil{1.384 \pm 0.008}$ & $\meil{1.446 \pm 0.008}$ & $1.240 \pm 0.002$        & $\meil{1.414 \pm 0.007}$ \\
IdPenLoo        & $\meil{1.378 \pm 0.008}$ & $\meil{1.458 \pm 0.008}$ & $\meil{1.233 \pm 0.002}$ & $\meil{1.419 \pm 0.007}$ \\
\noalign{\smallskip}
\hline
\noalign{\smallskip}
IdEpenid        & $1.363 \pm 0.008$        & $1.401 \pm 0.008$        & $1.232 \pm 0.002$        & $1.410 \pm 0.007$ \\
\hline
\end{tabular} 
\end{center}
\end{table}

The first remark we can make from Figures~\ref{fig.res.Cor-div1} and~\ref{fig.res.Cor-div2} is that changing the overpenalization factor $\Cov$ can make large differences in $\Cor\,$, as it was pointed in \cite{Arl:2008a}. We do not address here the question of choosing $\Cov$ from data since it is beyond the scope of the paper; see Section~11.3.3 of \cite{Arl:2007:phd} for suggestions.

The choice of the overpenalization factor can be put aside in two ways. 
First, we can compare on Figures~\ref{fig.res.Cor-div1} and~\ref{fig.res.Cor-div2} the performances obtained with the best deterministic value of $\Cov$ for each procedure. Second, we can compare procedures with their best data-driven (and distribution dependent) overpenalization factor, as done in Table~\ref{LL.tab.deux}. 
Both ways yield the same qualitative conclusion in the four experiments---up to minor variations---which can be summarized as follows.

Firstly, most resampling-based procedures (that is, VFCV, penVF, penLoo) 
outperform dimensionality-based procedures (MalEst, which strongly overfits, and even MalMax), which confirms the theoretical results of Sections~\ref{sec.urep} and~\ref{sec.dimbased}. 
In particular, penLoo yields a significant improvement over MalEst and MalMax (by more than 25\% in three experiments, and by 2.3\% in XS1--05), and IdPenLoo similarly outperforms IdLin and IdDim (by more than 8.5\% in three experiments, and by 2.3\% in XS1--05).
Even penLoo with a well-chosen {\em deterministic} overpenalization factor $\Cov$ outperforms IdLin by 7\% to 18\% in experiments X1--005, S0--1 and X1--005$\mu$02, and penLoo equals the performances of IdLin in experiment XS1--05 (compare Table~\ref{LL.tab.deux} with \refTabapp).
Figure~\ref{fig.X1-005.nvar} illustrates the same phenomenon: when the sample size increases, the model selection performance of IdLin remains approximately constant (close to~2) while the model selection performance $\Cor$ of penLoo constantly decreases (with $\Cov=1.25$ because overpenalization is still needed for $n=3\,000$ and we could not consider larger sample sizes for computational reasons).
The reasons for this clear advantage of resampling-based procedures are the same in the four experiments: As pointed out in Section~\ref{sec.main.subopt.illus} for experiment X1--005, no dimensionality-based model selection procedure can select a model close enough to the oracle $\mo\,$. In particular, figures similar to Figure~\ref{fig.th1.path-density} hold in experiments S0--1, X1--005$\mu$02, and XS1--05, see \refFigpath.

Secondly, improving over dimensionality-based procedures requires a significant increase of the computational cost. 
Indeed, PenHO performs significantly worse than MalMax in experiments X1--005 and XS1--05, while penHO and MalMax have similar performances in experiments S0--1 and X1--005$\mu$02. 
Furthermore, IdPenHO performs worse than IdDim in the four experiments, and even worse than IdLin in experiments X1--005 and XS1--05.
In order to obtain sensibly better performances than dimensionality-based procedures, our experiments show that the computational cost must at least be increased to the one of 5-fold penalization.
This phenomenon certainly comes from the small signal-to-noise ratio, which makes it difficult to estimate precisely the penalty shape by resampling, whereas MalMax can provide reasonably good performances thanks to underfitting.

Finally, let us add that penVF outperforms VFCV (and similarly penHO outperforms HO), provided the (deterministic) overpenalization factor is well-chosen, as shown in a previous paper \cite{Arl:2008a}. 

\begin{figure}
\begin{minipage}[b]{.53\linewidth}
\includegraphics[width=\textwidth]{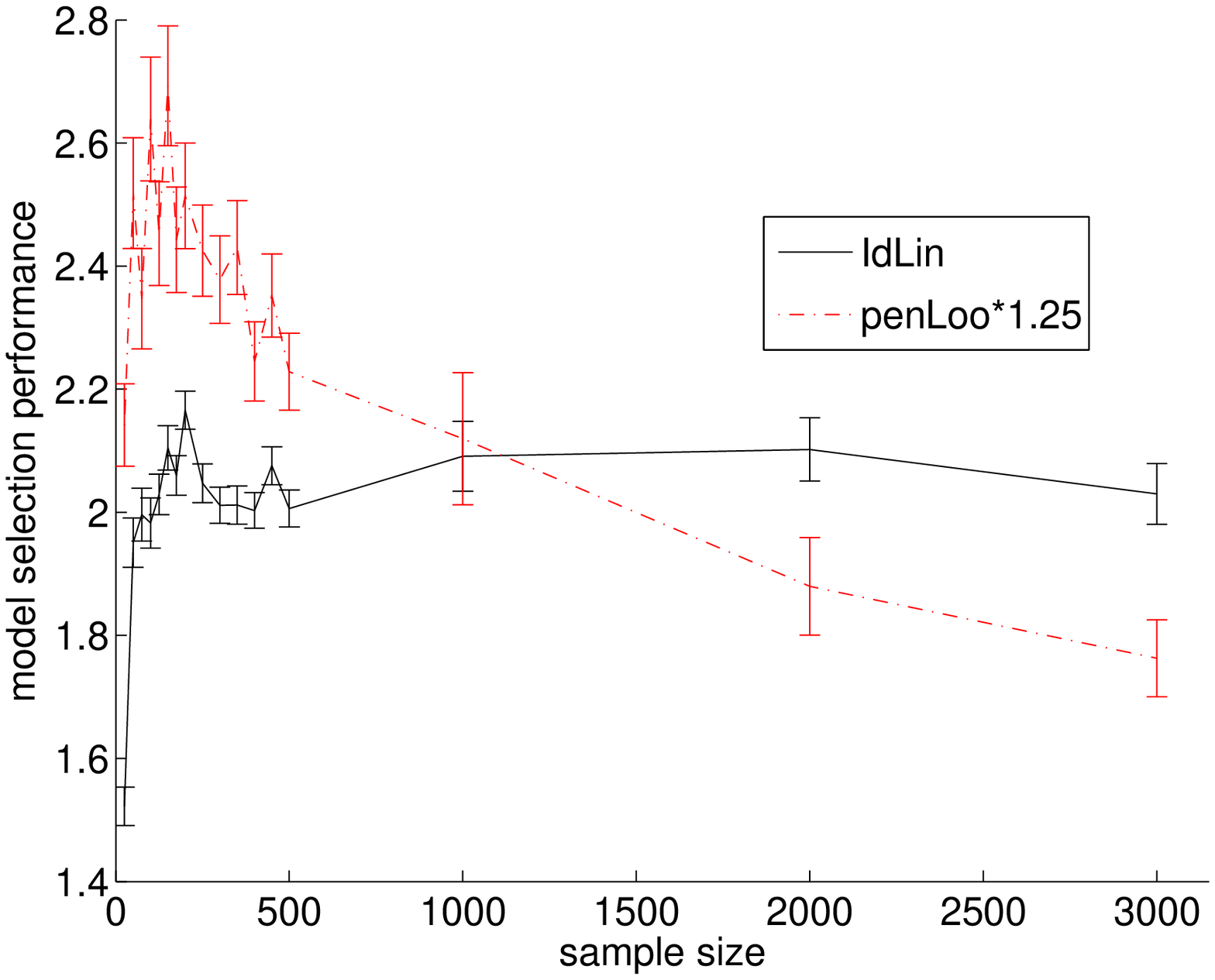}
\caption{Model selection performance \Cor\ of IdLin and penLoo (multiplied by $\Cov=1.25$) as a function of $n$ for experiment X1--005. Each estimated value of \Cor\ is obtained with $N=1\,000$ data samples for $n \leq 500$ and $N=200$ data samples for $n \geq 1\,000$ for computational reasons. Error bars represent $\epsCor\,$. 
\label{fig.X1-005.nvar}
}
\end{minipage}
\hspace{.025\linewidth}
\begin{minipage}[b]{.43\linewidth}
\includegraphics[width=\textwidth]{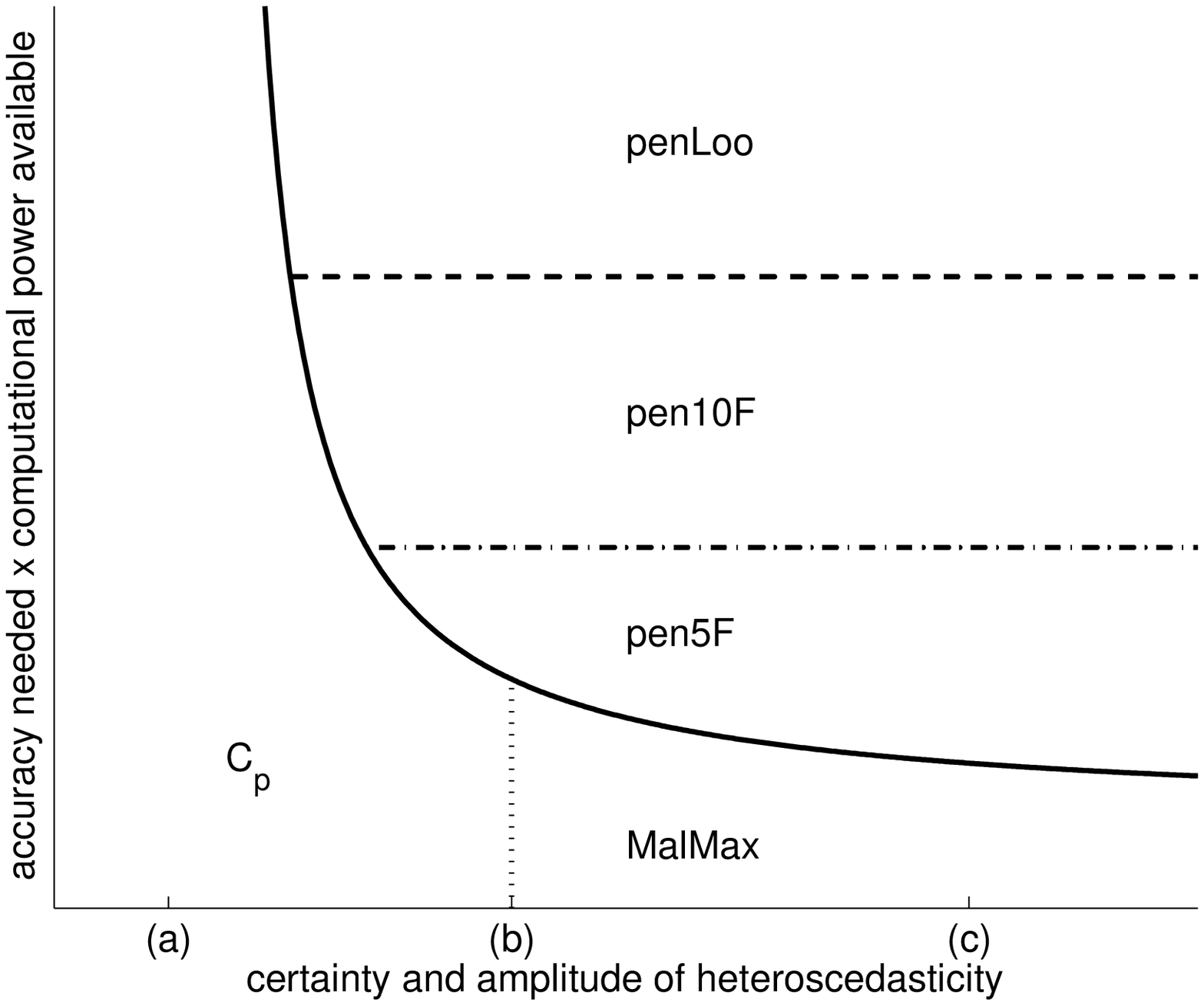}
\caption{Summary of Section~\ref{sec.concl}: best penalty as a function of the computational power available, the level of accuracy needed, and the knowledge about heteroscedasticity of data. The limit between linear and resampling penalties goes up when the SNR goes down.
\label{fig.schema}
}
\end{minipage}
\end{figure}

\section{Conclusion: How to choose the penalty?} \label{sec.concl}
Combining the theoretical results of Sections~\ref{sec.urep} and~\ref{sec.dimbased} with the conclusions of the experiments of Section~\ref{sec.simus}, we can propose an answer to the main question raised in this paper: Which penalty should be used for which model selection problem? A visual summary of this conclusion is proposed on Figure~\ref{fig.schema}. 

Three main factors must be taken into account to answer this question:
\begin{enumerate}
\item the prior knowledge on the noise-level $\sigma(\cdot)$, 
\item the trade-off between computational power and statistical accuracy desired,
\item the signal-to-noise ratio (SNR).
\end{enumerate}

What is known about $\sigma(\cdot)$ appears as the determinant factor:
\begin{itemize}
\item[(a)] If $\sigma(\cdot)$ is known to be constant, $C_p$ clearly is the best procedure, compared to cross-validation or resampling-based procedures which cannot take into account this information about data.
\item[(d)] If $\sigma(\cdot)$ is non-constant but completely known, then the expectation of the ideal penalty $\E\croch{\penid(m)}$ is entirely known and should be used, following the unbiased risk estimation principle.
\end{itemize}
Note that $C_p$ or AIC are still often used in case (d), mainly by non-statisticians that probably do not know (or do not trust) model selection procedures adapting to heteroscedasticity. This paper provides clear  theoretical arguments to show them what improvement they could obtain by using a properly chosen procedure.

\medskip

Choosing a penalty is less simple in the following two intermediate cases, where a trade-off must be found between the precise knowledge on $\sigma$, computational power and statistical accuracy:
\begin{itemize}
\item[(b)] $\sigma(\cdot)$ is probably (almost) constant, but this information is questionable
\item[(c)] $\sigma(\cdot)$ is known to be non-constant, without prior information on its shape
\end{itemize}

If the computational power has no limits (or, equivalently, if accuracy is crucial), resampling penalties (or $V$-fold penalties with $V$ large) should be used in both cases. 

Nevertheless, when the computational power is limited, one has to take into account that $V$-fold penalties with too small $V$ poorly estimate the shape of the penalty, so that they may be outperformed by MalMax (that is, $C_p$ with $\sigma^2$ replaced by some upper-bound on $\norm{\sigma}_{\infty}^2$), in particular in case (b).
Similarly, if the final user does not matter loosing a (small) constant in the quadratic risk, MalMax could be used instead of resampling in case (b), and even in case (c) provided $(\max\sigma/\min\sigma)^2$ is small enough. Of course, using MalMax requires either the knowledge of an upper bound on $\norm{\sigma}_{\infty}^2$, or to be able to estimate one (for instance assuming $\sigma(\cdot)$ does not jump too much).

\medskip

The picture can also change depending on the SNR. When the SNR is small, overpenalization is usually required. Therefore, choosing a proper overpenalization level can be more important than estimating the shape of the penalty, so that MalMax (possibly with an enlarged penalty) is quite a reasonable choice in case (b), and even in case (c) depending on the computational power.
On the contrary, when the SNR is large, $V$-fold penalties (even with rather small $V$, such as $V=5$) yield a significant improvement over any dimensionality-based penalty.

\medskip

A natural question arises from this conclusion: How to calibrate precisely the constant in front of the penalty? 
Birg\'e and Massart \cite{Bir_Mas:2006} proposed an optimal (and computationally cheap) data-driven procedure answering this question, based upon the concept of minimal penalties (see \cite{Arl_Mas:2009:pente} for the heteroscedastic regression framework). Nevertheless, theoretical results on Birg\'e and Massart's procedure are not accurate enough to determine whether it takes into account the need for overpenalization when the SNR is small. 
Therefore, understanding precisely how we should overpenalize as a function of the SNR seems a quite important question from the practical point of view, which is still widely open, up to the best of our knowledge.

\section{Proofs} \label{LL.sec.proof}

Before proving Theorem~\ref{th.shape} and Propositions~\ref{pro.oracle.lin.Klarge} and~\ref{pro.overfit.Mal}, let us define some notation and recall probabilistic results from other papers \cite{Arl:2009:RP,Arl:2008a,Arl_Mas:2009:pente} that are used in the proof. 

\subsection{Notation} \label{LL.sec.proof.nota}
In the rest of the paper, $L$ denotes an absolute constant, not
necessarily the same at each occurrence.
When $L$ is not universal, but depends on $p_1, \ldots, p_k$, it is written $L_{p_1,\ldots, p_k}$.

Define, for every model $\mM_n \,$, 
\begin{gather*} 
p_1(m) \egaldef P \gamma\paren{\ERM_m} - P \gamma\paren{\bayes_m} \qquad p_2(m) \egaldef P_n \gamma\paren{\bayes_m} - P_n \gamma\paren{\ERM_m} \quad \mbox{and} \quad \\
\delc(m) \egaldef (P_n - P) \paren{ \gamma\paren{\bayes_m} - \gamma\paren{\bayes}} 
\enspace . 
\end{gather*} 

\subsection{Probabilistic tools: expectations} \label{LL.sec.proof.expe}
\begin{proposition}[Proposition~1 and Lemma~7 in \cite{Arl:2008a}] \label{VFCV.pro.EcritVFCV-EcritID}
Let $S_m$ be the model of histograms associated with the partition $\Lambda_m\,$.
Then,
\begin{align} \label{eq.Ep1}
\E \croch{ p_1(m) } &= \frac{1}{n} \sum_{\lamm} \paren{1 + \delta_{n, \pl}} \sigl^2  
\\ \label{eq.Ep2}
\E \croch{ p_2(m) } &= \frac{1}{n} \sum_{\lamm} \sigl^2  
\\ \notag 
\mbox{where} \quad \sigl^2 &\egaldef \E\croch{ \paren{Y - \bayes(X)}^2 \sachant X \in \Il} = \carre{\sigld} + \carre{\sigla} \enspace ,
\end{align}
$\pl = \Prob\paren{X \in \Il}$ and $\delta_{n,p}$ only depends on $(n,p)\,$. 
Moreover, $\delta_{n,p}$ is small when the product $np$ is large: 
\[ \absj{\delta_{n, p} } \leq \min\set{ L_1, L_2 (np)^{-1/4} } \enspace , \]
where $L_1$ and $L_2$ are absolute constants.
\end{proposition}
Note that $\delta_{n,p}$ can be made explicit: $\delta_{n,p} = np \E\croch{Z^{-1} \un_{Z>0}} - 1 $
where $Z$ is a binomial random variable with parameters $(n,p)\,$.
\begin{remark} \label{rk.histo.pbdef} The regressogram estimator $\ERM_m$ is not defined when 
$\card\set{ i \telque X_i \in \Il}=0$ for some $\lamm$, which occurs with positive probability.
Therefore, a convention for $p_1(m)$ as to be chosen on this event (which has a small probability, see Claim~\ref{claim:nopb} in the proof of Theorem~\ref{th.shape}) so that $p_1(m)$ has a finite expectation (see \cite{Arl:2008a} for details). 
This convention is purely formal, since the statement of Theorem~\ref{th.shape} does not involve the expectation of $p_1(m)$. The important point is that the same convention is used in Proposition~\ref{VFCV.pro.conc.penid} below.
\end{remark}

\subsection{Probabilistic tools: concentration inequalities} \label{sec.proof.conc}
We state in this section some concentration results on the components of the ideal penalty, using for $p_1(m)$ the same convention as in Proposition~\ref{VFCV.pro.EcritVFCV-EcritID}.

\begin{proposition}[Proposition~10 in \cite{Arl:2009:RP}, proved in Section~4 of \cite{Arl:2008b:app}]
\label{VFCV.pro.conc.penid}
Let $\gamma>0$. Assume that $\min_{\lamm} \set{n \pl} \geq B_n \geq 1$, $\norm{Y}_{\infty} \leq A < \infty$ and 
\[ D_m^{-1} \sum_{\lamm} \E\croch{\sigma(X)^2 \sachant X \in \Il} \geq Q > 0 \enspace . \]
Then, an event of probability at least $1 - L n^{-\gamma}$ exists on which 
\begin{align} \label{VFCV.eq.conc.p1.min}
p_1(m) &\geq \E \croch{p_1(m)} - L_{A,Q,\gamma}  \croch{   \paren{\ln n}^2 D_m^{-1/2} +  e^{-L B_n} }  \E\croch{p_2(m)}
\\
\label{VFCV.eq.conc.p1.maj}
p_1(m) &\leq \E \croch{p_1(m)} + L_{A,Q,\gamma}  \croch{  \paren{\ln n}^2 D_m^{-1/2}  +  \sqrt{D_m} e^{-L B_n} }  \E\croch{p_2(m)}
\end{align}
\begin{equation}
\label{VFCV.eq.conc.p2}
\absj{p_2(m) - \E [p_2(m)]} \leq L_{A,Q,\gamma} D_m^{-1/2} \ln(n) \E\croch{p_2(m)} \enspace .
\end{equation}
\end{proposition}

\begin{lemma}[Proposition~8 in \cite{Arl_Mas:2009:pente}]
 \label{VFCV.le.conc.delta.borne}
Assume that $\norm{Y}_{\infty} \leq A < \infty$. Then for any $x \geq 0$, an event of probability at least $1 - 2 e^{-x}$ exists on which
\begin{equation}\label{VFCV.eq.conc.delta.borne.2}
\forall \eta >0 \, , \qquad \absj{\delc(m)} \leq \eta \perte{\bayes_m} + \left( \frac{4}{\eta} + \frac{8}{3} \right) \frac{A^2 x}{n}  \enspace .\end{equation}
\end{lemma}

\begin{lemma}[Lemma~12 in \cite{Arl:2008a}] \label{le.bernstein.binom}
Let $(\pl)_{\lamm}$ be non-negative real numbers of sum 1, $(n\phl)_{\lamm}$ a multinomial vector of parameters
$(n;(\pl)_{\lamm})$, and $\gamma>0$. Assume that $\card(\Lambda_m) \leq n$ and $\min_{\lamm} \set{n \pl} \geq B_n >0$. Then, an event of probability at least $1 - L n^{-\gamma}$ exists on which
\begin{align}
\label{VFCV.eq.threshold} 
\min_{\lamm} \set{n \phl} \geq \frac{\min_{\lamm} \set{n \pl}}{2} - 2 (\gamma+1) \ln n \enspace . 
\end{align}
\end{lemma}

\subsection{Proof of Theorem~\ref{th.shape}} \label{sec.proof.main}
In the following, $\LThm = L_{E,\siga,\sigb,\norm{\sigma^{\prime}}_{\infty}, \norm{\sigma^{\prime\prime}}_{\infty}}$ denotes any constant depending on $E$, $\siga$, $\sigb$, $\norm{\sigma^{\prime}}_{\infty}$ and $\norm{\sigma^{\prime\prime}}_{\infty}$ only.
The outline of the proof of Theorem~\ref{th.shape} is the following:
\begin{itemize}
\item From Lemma~\ref{claim:1}, it is sufficient to prove that for every $D$, 
\[ \inf_{m \in \Mdim(D)} \set{ \perte{\ERM_m} } \geq \KThmDimFAILoracle \inf_{\mM_n} \set{ \perte{\ERM_m} } \] holds with probability at least $1 - \KThmDimFAILproba n^{-2} \,$, with $\KThmDimFAILoracle>1\,$.
\item Prove that all regressogram estimators are well-defined with a large probability (Claim~\ref{claim:nopb}).
\item Compute explicitly the bias of each models (Claim~\ref{claim:bias}).
\item Provide a good approximation of the excess loss and the empirical risk of each model on a large probability event (Claim~\ref{claim:control}).
\item Upper bound the excess loss of the oracle model (Claim~\ref{claim:upperbound}). 
\item Lower bound the excess loss of small models (Claim~\ref{claim:smallmod}).
\item Prove that all models $m$ having an excess loss close to the one of the oracle are close to the oracle model (Claim~\ref{claim:gooddim}).
\item Conclude the proof by showing that for every model close to the oracle, a model with a smaller empirical risk can be found.
\end{itemize}

As pointed out by Remark~\ref{rk.histo.pbdef}, the regressogram estimator associated with model $S_m$ is not well defined if for some $\lamm$, no $X_i$ belongs to $\Il$. The following claim shows this only happens with a small probability, hence this possible problem can be put aside for proving Theorem~\ref{th.shape}.
\begin{claim} \label{claim:nopb}
An event $\Omega_1$ of probability at least $1 - L n^{-2}$ exists on which 
\[ \forall\mM_n \, , \, \forall \lamm \, , \, \qquad \card\set{ i \telque X_i \in \Il} \geq 1 \enspace . \]
Hence, on $\Omega_1$, all estimators $\paren{\ERM_m}_{\mM_n}$ are well-defined.
\end{claim}
\begin{proof}[proof of Claim~\ref{claim:nopb}]
For every $\mM_n \, $, let us apply Lemma~\ref{le.bernstein.binom} with $B_n = \paren{\ln(n)}^2$ and $\gamma=4$.
An event of probability at least $1 - L n^{-4}$ exists such that 
\[ \min_{\lamm} \card\set{ i \telque X_i \in \Il} \geq \frac{\paren{\ln(n)}^2}{2} - 10 \ln n \enspace . \]
This lower bound is positive provided that $n \geq L$.
Therefore, the result holds on the intersection $\Omega_1$ of these $\card(\M_n) \leq n^2$ events.
\end{proof}

\medskip

The next step is to use the results recalled in Sections~\ref{LL.sec.proof.expe} and ~\ref{sec.proof.conc} in order to control the excess loss and the empirical risk of each model. This leads to Claims~\ref{claim:bias} and~\ref{claim:control} below.
\begin{claim} \label{claim:bias}
Define  
\[ \alpha_1 = \frac{1}{48} \int_0^{1/2} \paren{\bayes^{\prime}(x)}^2 dx \qquad \mbox{and} \qquad \alpha_2 = \frac{1}{48} \int_{1/2}^1 \paren{\bayes^{\prime}(x)}^2 dx \enspace . \]
For every $\mM_n \, $, some $\kappa_{m,b,1},\kappa_{m,b,2} \in \R$ exist such that 
\begin{equation} \label{eq.claim:bias}
\perte{\bayes_m} = \alpha_1 D_{m,1}^{-2} \paren{1 + \kappa_{m,b,1}} + \alpha_2 D_{m,2}^{-2} \paren{1 + \kappa_{m,b,2}} 
\end{equation}
and $\absj{\kappa_{m,b,i}} \leq L_{\norm{ \bayes^{\prime} }_{\infty}, \, \norm{ \bayes^{\prime\prime} }_{\infty} } D_{m,i}^{-1} \,$.
\end{claim}
\begin{proof}[proof of Claim~\ref{claim:bias}]
Since $X$ is uniformly distributed on $[0,1]$, 
\begin{equation} \label{eq.pr.claim:bias.0} 
\perte{\bayes_m} = \int_0^1 \paren{\bayes_m (X) - \bayes(X)}^2 = \sum_{\lamm} \int_{\Il} \paren{ \bayes(x) - \bayes_{\lambda} }^2 dx 
\end{equation}
where $\bayes_{\lambda}$ is the average of $\bayes$ on $\Il\,$.
We now fix some $\lamm$.  
Let $c_{\lambda}$ denote the center of the interval $\Il\,$, and $\absj{\Il}$ the length of $\Il\,$.
Then,
\begin{equation} \label{eq.pr.claim:bias.1} 
\int_{\Il} \paren{ \bayes(x) - \bayes_{\lambda} }^2 dx  = \paren{ \bayes(c_{\lambda}) - \bayes_{\lambda} }^2  + \int_{\Il} \paren{ \bayes(x) - \bayes(c_{\lambda})  }^2 dx \enspace . 
\end{equation}
In addition, since $\bayes$ is twice continuously differentiable, for every $x \in \Il$, some $g(x) \in \Il$ exists such that 
\begin{equation} \label{eq.pr.claim:bias.2}
\bayes(x) - \bayes(c_{\lambda}) 
=  (x - c_{\lambda}) \bayes^{\prime}(c_{\lambda}) + \frac{1}{2} (x - c_{\lambda})^2 \bayes^{\prime\prime}(g(x))
\enspace . 
\end{equation}
On the one hand, integrating \eqref{eq.pr.claim:bias.2} over $\Il$ leads to 
\begin{equation} \label{eq.pr.claim:bias.3}
\paren{ \bayes_{\lambda} - \bayes(c_{\lambda}) }^2 \leq L \norm{\bayes^{\prime\prime}}_{\infty}^2 \absj{\Il}^4 \enspace . 
\end{equation}
On the other hand, integrating the square of \eqref{eq.pr.claim:bias.2} over $\Il$ leads to 
\begin{equation} \label{eq.pr.claim:bias.4}
\absj{ \int_{\Il} \paren{ \bayes(x) - \bayes(c_{\lambda}) }^2 dx 
-  \frac{ \bayes^{\prime 2}(c_{\lambda}) \absj{\Il}^3} {12} } 
\leq L \absj{\Il}^4 \norm{\bayes^{\prime\prime}}_{\infty} \paren{ \norm{\bayes^{\prime}}_{\infty} + \norm{\bayes^{\prime\prime}}_{\infty} }
\enspace . 
\end{equation}
Combining \eqref{eq.pr.claim:bias.1} with \eqref{eq.pr.claim:bias.3} and \eqref{eq.pr.claim:bias.4} then shows that for every $\lamm$,
\begin{equation} \label{eq.pr.claim:bias.5}  
\absj{ \int_{\Il} \paren{ \bayes(x) - \bayes_{\lambda} }^2 dx - \frac{ \bayes^{\prime 2}(c_{\lambda}) \absj{\Il}^3} {12} } 
\leq L \absj{\Il}^4 \norm{\bayes^{\prime\prime}}_{\infty} \paren{ \norm{\bayes^{\prime}}_{\infty} + \norm{\bayes^{\prime\prime}}_{\infty} }
\enspace . \end{equation}
Furthermore, for every $\lamm$, 
\begin{align} \notag 
\absj{ \absj{\Il} \paren{ \bayes^{\prime}(c_{\lambda}) }^2 - \int_{\Il} \paren{ \bayes^{\prime}(x) }^2 dx } 
\leq 
\int_{\Il} \absj{ \paren{ \bayes^{\prime}(x) }^2 - \paren{ \bayes^{\prime}(c_{\lambda}) }^2 } dx 
&\leq
2 \norm{\bayes^{\prime}}_{\infty}
\int_{\Il} \absj{ \bayes^{\prime}(x) - \bayes^{\prime}(c_{\lambda})} dx
\\
&\leq 
2 \norm{\bayes^{\prime}}_{\infty} \norm{\bayes^{\prime\prime}}_{\infty} \absj{\Il}^2 
\label{eq.pr.claim:bias.5b} \enspace .
\end{align}
Using \eqref{eq.pr.claim:bias.0}, combining \eqref{eq.pr.claim:bias.5b} with \eqref{eq.pr.claim:bias.5} and summing over $\lamm$ implies 
\begin{align*}
 \absj{ \perte{\bayes_m}  - \frac{1}{12} \sum_{\lamm} \paren{ \absj{\Il}^2 \int_{\Il} \paren{ \bayes^{\prime}(x) }^2  dx } }
&\leq 
L \norm{\bayes^{\prime\prime}}_{\infty} \paren{ \norm{\bayes^{\prime}}_{\infty}  + \norm{\bayes^{\prime\prime}}_{\infty}  } \sum_{\lamm} \absj{\Il}^4 \\
&\leq L \norm{\bayes^{\prime\prime}}_{\infty} \paren{ \norm{\bayes^{\prime}}_{\infty}  + \norm{\bayes^{\prime\prime}}_{\infty}  } \paren{ D_{m,1}^{-3} + D_{m,2}^{-3} }
\end{align*}
and the result follows.
\end{proof}

\begin{claim} \label{claim:control}
Define $ \beta_1 = 2 \carre{\siga}$ and $\beta_2 = 2 \carre{\sigb} \,$. 
An event $\Omega$ of probability at least $1 - L n^{-2}$ exists on which for every $m = (D_{m,1},D_{m,2}) \in \M_n \, $, 
\begin{gather} 
\perte{\ERM_m} = \paren{  \frac{\alpha_1}{D_{m,1}^2} + \frac{\alpha_2}{D_{m,2}^2} + \frac{\beta_1 D_{m,1}} {n} +  \frac{\beta_2 D_{m,2} } {n} } \paren{1 + \kappa_{m,0}} \label{eq.claim.control.loss}
\\
\mbox{and} \qquad  
P_n \gamma\paren{ \ERM_m } - P\gamma\paren{\bayes} = 
\croch{ \frac{\alpha}{D_{m,1}^2} + \frac{\alpha}{D_{m,2}^2}  } \paren{1 + \kappa_{m,1}}
- \croch{ \frac{\beta_1  D_{m,1}} {n} + \frac{\beta_2 D_{m,2} } {n} } \paren{1 + \kappa_{m,2}} 
\enspace , \label{eq.claim.control.emprisk} 
\end{gather}
where $(\kappa_{m,i})_{i=0,1,2, \, \mM_n}$ satisfy 
\[ \max \set{ \absj{ \kappa_{m,i} } \telque i \in \set{0,1,2} , \, \mM_n  , \, \mbox{and }  \min\set{D_{m,1}, D_{m,2}} \geq \paren{\ln n}^6 } \leq \LThm \paren{\ln n}^{-1/2} \enspace . \]
\end{claim}
\begin{proof}[proof of Claim~\ref{claim:control}]
Using the notation of Section~\ref{LL.sec.proof.nota}, 
\[ \perte{\ERM_m} = \perte{\bayes_m} + p_1(m) \qquad \mbox{and} \qquad P_n \gamma\paren{ \ERM_m } - P\gamma\paren{\bayes} = \perte{\bayes_m} - p_2(m) + \delc(m) \enspace . \]
Let us first compute the expectation of each term. Recall that $\E\scroch{\delc(m)}=0$.
The bias $\perte{\bayes_m}$ is controlled thanks to Claim~\ref{claim:bias}.
By Proposition~\ref{VFCV.pro.EcritVFCV-EcritID}, $\E\croch{p_1(m)}$ and $\E\croch{p_2(m)}$ mostly depend on 
\begin{align*} \sigl^2 &= \E\croch{\paren{Y - \bayes(X)}^2 \sachant X \in \Il} \\
&= \E\croch{\paren{\bayes(X) - \bayes_m(X)}^2 \sachant X \in \Il} + \E\croch{\carre{\sigma(X)} \sachant X \in \Il} \\
&= \frac{1}{\Leb(\Il)} \int_{\Il} \paren{\bayes(x) - \bayes_m(x)}^2 dx + \frac{1}{\Leb(\Il)} \int_{\Il} \carre{\sigma(X)} dx \enspace . \end{align*}
Precisely,
\begin{align*}
\E\croch{p_2(m)} &= \frac{1}{n} \sum_{\lamm} \sigl^2 \\
&= \frac{2 D_{m,1} } {n} \int_0^{1/2} \croch{\paren{\bayes(X) - \bayes_m (x)}^2 + \carre{\sigma(x)} } dx + \frac{2 D_{m,2}} {n} \int_{1/2}^1 \croch{\paren{\bayes(X) - \bayes_m (x)}^2 + \carre{\sigma(x)} } dx \\
&= 
\frac{\beta_1 D_{m,1}} {n} + \frac{\beta_2 D_{m,1}} {n} 
+ R(m,n) \\
\mbox{where} \quad 0 \leq  R(m,n) &= 
\frac{2}{n} \paren{ D_{m,1} \int_0^{1/2} \paren{\bayes(X) - \bayes_m (x)}^2 dx
+ D_{m,2} \int_{1/2}^1 \paren{\bayes(X) - \bayes_m (x)}^2 dx } \leq 
\frac{ \perte{\bayes_m} } { \paren{ \ln(n) }^2 }
\enspace , 
\end{align*}
since $D_{m,i} \leq n / \sparen{ 2 \sparen{ \ln(n) }^2 } \,$.
Similarly,
\begin{align*}
\E\croch{p_1(m)} = \paren{ \frac{\beta_1 D_{m,1}} {n} + \frac{\beta_2 D_{m,1}} {n} } \paren{ 1 + \delta_n } + R^{\prime} (m,n) \\
\mbox{where} \quad 0 \leq  R^{\prime} (m,n) \leq 
\frac{ \perte{\bayes_m} } { \paren{ \ln(n) }^2 } \qquad \mbox{and} \qquad \absj{\delta_n} \leq L \paren{\ln n}^{-1/2} \enspace . 
\end{align*}

\medskip

It now remains to prove that $p_1(m) - \E\croch{p_1(m)}$ and $\E\croch{p_2(m)} - p_2(m) + \delc(m)$ are close to zero on a large probability event. 
The condition on $\sigma(\cdot)$ imply that the last assumption of Proposition~\ref{VFCV.pro.conc.penid} holds since 
\begin{align*}
D_m^{-1} \sum_{\lamm} \E\croch{\sigma(X)^2 \sachant X \in \Il} 
= \frac{2 D_{m,1} \carre{\siga} + 2 D_{m,2} \carre{\sigb} } {D_{m,1} + D_{m,2}} 
\geq 2 \min\set{ \carre{\siga}, \carre{\sigb}} \defegal Q > 0 
\enspace . 
\end{align*}

Let $\Omega_2$ be the event on which, for every $\mM_n \, $, \eqref{VFCV.eq.conc.p1.min}--\eqref{VFCV.eq.conc.p2} hold with $\gamma=4$ and $B_n = \ln(n)^2$, and \eqref{VFCV.eq.conc.delta.borne.2} holds with $x = 4 \ln(n)$ and $\eta = \sparen{\ln(n)}^{-1}\,$.
Since $\card(\M_n) \leq n^2$, Proposition~\ref{VFCV.pro.conc.penid} and Lemma~\ref{VFCV.le.conc.delta.borne} show that $\Prob\paren{\Omega_2} \geq 1 - L n^{-2}$.
Therefore, the probability of $\Omega = \Omega_1 \cap \Omega_2$ is larger than $1 - \KThmDimFAILproba n^{-2}$ for some absolute constant $\KThmDimFAILproba\,$.

On $\Omega$, for every $\mM_n$ such that $\min\set{D_{m,1}, D_{m,2}} \geq \sparen{\ln(n)}^6$, we then have 
\begin{align*}  
\absj{p_1(m) - \E \croch{p_1(m)}} &\leq \LThm \paren{\ln n}^{-1} \E\croch{p_2(m)}
\\ 
\absj{p_2(m) - \E [p_2(m)]} &\leq \LThm \paren{\ln n}^{-2} \E\croch{p_2(m)} \\
\mbox{and} \qquad 
\absj{\delc(m)} &\leq \frac{\perte{\bayes_m}}{\ln n} + \frac{\LThm \paren{\ln n}^2}{n} \end{align*}
as soon as $n \geq L\,$. Enlarging the constant $\KThmDimFAILproba$ so that $1- \KThmDimFAILproba n^{-2} \leq 0$ when $n$ is too small yields the result.
\end{proof}

\medskip

\begin{claim} \label{claim:upperbound}
On $\Omega$, 
\begin{equation} \label{eq.claim:upperbound} 
\inf_{\mM_n} \set{\perte{\ERM_m}} \leq \frac{3 }{2^{2/3} n^{2/3}} \paren{ \alpha_1^{1/3} \beta_1^{2/3} + \alpha_2^{1/3} \beta_2^{2/3} } \paren{1 + \LThm \paren{\ln n}^{-1/2}} \enspace . 
\end{equation}
\end{claim}
\begin{proof}[proof of Claim~\ref{claim:upperbound}]
Let $\mob \in \M_n$ be any model such that 
\[ \absj{D_{\mob,1} - \paren{ \frac{2 \alpha_1 n }{\beta_1} }^{1/3} } \leq 1 \qquad \mbox{and} \qquad \absj{D_{\mob,2} - \paren{ \frac{2 \alpha_2 n }{\beta_2} }^{1/3} } \leq 1 \enspace . \]
As soon as $n \geq \LThm$, such an $\mob$ exists and satisfies $\min\set{D_{\mob,1}, D_{\mob,2}} > \paren{\ln n}^{6}\,$. The result follows from Claim~\ref{claim:control}. 
\end{proof}
\begin{claim} \label{claim:smallmod}
For every $\mM_n$ such that $\min\set{ D_{m,1}, D_{m,2}} \leq \paren{\ln(n)}^6$, 
\[ \perte{\ERM_m} \geq \frac{L_{\alpha_1, \alpha_2 , \norm{ \bayes^{\prime} }_{\infty} , \norm{ \bayes^{\prime\prime} }_{\infty} }}{\paren{\ln(n)}^{12}} \enspace . \]
\end{claim}
In particular, Claims~\ref{claim:upperbound} and~\ref{claim:smallmod} show that for every $\KThmDimFAILoracle>0\,$, when $\min\set{ D_{m^{\prime},1}, D_{m^{\prime},2}} \leq \paren{\ln(n)}^6$ and $n \geq L_{\hypThm,\KThmDimFAILoracle} \, $, 
\[ \perte{\ERM_{m^{\prime}}} \geq \KThmDimFAILoracle \inf_{\mM_n} \set{ \perte{\ERM_m} }  \enspace . \] 
\begin{proof}[proof of Claim~\ref{claim:smallmod}]
First, note that $\perte{\ERM_m} \geq \perte{\bayes_m} $, and by Claim~\ref{claim:bias}, 
for every $\mM_n \,$,
\begin{align*}
\perte{\bayes_m} \geq \alpha_1 D_{m,1}^{-2} \paren{1 - L_{\norm{ \bayes^{\prime} }_{\infty}, \, \norm{ \bayes^{\prime\prime} }_{\infty} } D_{m,1}^{-1} } 
+ \alpha_2 D_{m,2}^{-2} \paren{1 - L_{\norm{ \bayes^{\prime} }_{\infty}, \, \norm{ \bayes^{\prime\prime} }_{\infty} } D_{m,2}^{-1} }
\enspace .
\end{align*}
If $\mM_n$ satisfies $\min\set{ D_{m,1}, D_{m,2}} \geq L_1 \paren{\norm{ \bayes^{\prime} }_{\infty}, \, \norm{ \bayes^{\prime\prime} }_{\infty} } \,$, then, the lower bound is larger than 
\[ 
\frac{\alpha_1}{2 D_{m,1}^2} + \frac{\alpha_2}{2 D_{m,2}^2}
\geq 
\frac{ \min\set{ \alpha_1, \alpha_2} } {2 \min\set{ D_{m,1}, D_{m,2}}^{2} } \enspace . 
\]
Now fix some $\mM_n$ such that $\min\set{ D_{m,1}, D_{m,2}} \leq \paren{\ln(n)}^6$.
Some $m^{\prime} \in \M_n$ exists such that $S_m \subset S_{m^{\prime}}$, $L_1
 \leq D_{m^{\prime},1} \leq 2 \max\set{ L_1, \paren{ \ln(n) }^6 } \,$; indeed, either $m=m^{\prime}$ satisfies the condition, or $m^{\prime}$ can be obtained from $m$ by doubling the number of bins in $[0,1/2]$ (resp. $(1/2,1]$) until the required condition is fulfilled.
Then, 
\[ 
\perte{\bayes_m} 
\geq \perte{\bayes_{m^{\prime}}}
\geq \frac{ \min\set{ \alpha_1, \alpha_2} } {2 \min\set{ D_{m,1}, D_{m,2}}^{2} } 
\geq \frac{ L_{ \norm{ \bayes^{\prime} }_{\infty}, \, \norm{ \bayes^{\prime\prime} }_{\infty}  } } {\paren{ \ln(n) }^{12}} \enspace . 
\]
\end{proof}
\begin{claim} \label{claim:gooddim}
Define for every $\mM_n \, $, 
\[ C_{m,1} \egaldef D_{m,1} \paren{ \frac{ \beta_1 } { 2 \alpha_1 n }}^{1/3} > 0 \qquad \mbox{and} \qquad  C_{m,2} \egaldef D_{m,2} \paren{ \frac{ \beta_2 } { 2 \alpha_2 n }}^{1/3} > 0 \enspace . \]
Let $\Delta \in (0,1]$ and define 
\[ \eta_{\Delta} \egaldef \frac{\Delta^2} {17 \paren{1 + \paren{\frac{\beta_1} {\beta_2}}^{2/3} }} 
 \enspace . \]
Then, on $\Omega$, any $\mM_n$ such that 
\begin{equation}
\label{eq.hyp.claim:gooddim}
\perte{\ERM_m} \leq (1 + \eta_{\Delta} ) \inf_{\mM_n} \set{ \perte{\ERM_m} } 
\end{equation}
must satisfy
\begin{equation} \label{eq.claim:gooddim}
\max\set{ \absj{C_{m,1} - 1}, \absj{C_{m,2} - 1} } \leq \Delta  
\end{equation}
as soon as $n \geq L_{\hypThm,\Delta}\,$.
\end{claim}
\begin{proof}[proof of Claim~\ref{claim:gooddim}]
Assume that $\Omega$ holds and let $\mM_n$ be satisfying \eqref{eq.hyp.claim:gooddim}. 
From Claim~\ref{claim:smallmod}, we know that $\min\set{ D_{m,1}, D_{m,2}} > \paren{\ln(n)}^6$ for $n \geq L_{\hypThm,\Delta}$.

Define, for every $x > -1$, $f(x) = 2^{-2/3} (1+x)^{-2} + 2^{1/3} (1+x) \,$ and for every $x \geq 0$, $g(x) = \min\set{1,(x-1)^2}$. 
Then, \eqref{eq.claim.control.loss} in Claim~\ref{claim:control} and Lemma~\ref{VFCV.le.tech.1} below yield
\begin{align*}
\perte{\ERM_m} &\geq \frac{1}{n^{2/3}} \paren{ \alpha_1^{1/3} \beta_1^{2/3} f(C_{m,1} - 1) + \alpha_2^{1/3} \beta_2^{2/3} f(C_{m,2} - 1) } \paren{ 1 - \LThm \paren{\ln n}^{-1/2}} \\
&\geq \frac{3 }{2^{2/3} n^{2/3}} \paren{ \alpha_1^{1/3} \beta_1^{2/3} +  \alpha_2^{1/3} \beta_2^{2/3} }  
+ \frac{3 }{2^{14/3} n^{2/3}} \paren{ 1 - \LThm \paren{\ln n}^{-1/2} }  
\paren{ \alpha_1^{1/3} \beta_1^{2/3} g(C_{m,1}) + \alpha_2^{1/3} \beta_2^{2/3}  g(C_{m,2}) } \enspace .
\end{align*}
Hence, \eqref{eq.claim:upperbound} and \eqref{eq.hyp.claim:gooddim} imply
\begin{align*} 
16 \paren{ \alpha_1^{1/3} \beta_1^{2/3} + \alpha_2^{1/3} \beta_2^{2/3} } \frac{ \eta_{\Delta} + L_{\hypThm,\Delta} \paren{\ln n}^{-1/2} } {  1 - \LThm  \paren{\ln n}^{-1/2} }
&\geq \paren{ \alpha_1^{1/3} \beta_1^{2/3}  g(C_{m,1}) + \alpha_2^{1/3} \beta_2^{2/3}  g(C_{m,2}) } \enspace . 
\end{align*}

In particular, when $n \geq L_{\hypThm,\Delta} \,$, 
\begin{align*} 
g(C_{m,1}) \leq 16 \Delta^2 /17 < 1  \qquad \mbox{and} \qquad g(C_{m,2}) \leq 16 \Delta^2 /17 < 1
\enspace , \end{align*}
which implies \eqref{eq.claim:gooddim}.
\end{proof}

\begin{lemma} \label{VFCV.le.tech.1}
Let $f: (-1, +\infty) \flens \R$ be defined by $f(x) = 2^{-2/3} (1+x)^{-2} + 2^{1/3} (1+x)$. Then, for every $x>-1$,
\[ f(x) \geq 3 \times 2^{-2/3} + 3 \times 2^{-14/3}\min\set{x^2 , 1} \enspace . \]  
\end{lemma}
\begin{proof}[proof of Lemma~\ref{VFCV.le.tech.1}]
We apply the Taylor-Lagrange theorem to $f$ (which is infinitely differentiable) at order two, between 0 and $x$. The result follows since $f(0)=3 \times 2^{-2/3}$, $f^{\prime}(0)=0$ and $f^{\prime\prime}(t) = 6 \times 2^{-2/3} \times  (1+t)^{-4} \geq 3 \times 2^{1/3 - 4} $ if $t \leq 1$.
If $t>1$, the result follows from the fact that $f^{\prime} \geq 0$ on $[0,+\infty)$.
\end{proof}

We now can conclude the proof of Theorem~\ref{th.shape}.
Let us assume that on $\Omega$, $\mM_n$ satisfies \eqref{eq.hyp.claim:gooddim} for some $\Delta>0$ to be chosen later. 
Without loss of generality, we can assume that $\siga>\sigb\,$, hence $\beta_1 > \beta_2\,$.
By Claim~\ref{claim:gooddim}, we have
\[ \paren{ \frac{ 2 \alpha_1 n } { \beta_1} }^{1/3} \paren{1 - \Delta} \leq D_{m,1} \leq \paren{ \frac{ 2 \alpha_1 n } { \beta_1} }^{1/3} \paren{1 + \Delta} \]
and
\[ \paren{ \frac{ 2 \alpha_2 n } { \beta_2} }^{1/3} \paren{1 - \Delta} \leq D_{m,2} \leq \paren{ \frac{ 2 \alpha_2 n } { \beta_2} }^{1/3} \paren{1 + \Delta} \enspace . \]
Therefore,
\[ D_{m,1} \leq \kappa D_{m,2} \qquad \mbox{with} \qquad \kappa = \frac{ \alpha_1^{1/3} \beta_2^{1/3} \paren{ 1+\Delta } } { \alpha_2^{1/3} \beta_1^{1/3} \paren{ 1-\Delta } } \enspace . \]
Since $\beta_1>\beta_2$, we can choose $\Delta=\Delta(\beta_1, \beta_2)>0$ such that 
\[ \kappa \leq \kappa^{\prime} = (\alpha_1 / \alpha_2)^{1/3} (\beta_2 / \beta_1)^{1/6}  < \overline{\kappa} = (\alpha_1 / \alpha_2)^{1/3} \enspace . \] 
Note that $D_{m,1} \leq \kappa D_{m,2}$ is equivalent to $D_{m,1} \leq \kappa D_m / (1 + \kappa)$. 
Therefore, some $m^{\prime} \in \M_n$ exists such that $D_m = D_{m^{\prime}}$ and 
$-1 \leq  D_{m^{\prime},1} - D_m \overline{\kappa} / (1+\overline{\kappa}) \leq 0$. 
Then, \eqref{eq.claim.control.emprisk} implies that 
\begin{align*} 
P_n \gamma\paren{ \ERM_m } - P_n \gamma\paren{ \ERM_{m^{\prime}} } 
&= 
\croch{ \frac{\alpha_1}{D_{m,1}^2} + \frac{\alpha_2}{D_{m,2}^2}  } \paren{1 + \kappa_{m,1}}
- \croch{ \frac{\alpha_1}{D_{m^{\prime},1}^2} + \frac{\alpha_2}{D_{m^{\prime},2}^2}  } \paren{1 + \kappa_{m^{\prime},1}}
\\&\qquad 
- \croch{ \frac{\beta_1  D_{m,1}} {n} + \frac{\beta_2 D_{m,2} } {n} } \paren{1 + \kappa_{m,2}} 
+ \croch{ \frac{\beta_1  D_{m^{\prime},1}} {n} + \frac{\beta_2 D_{m^{\prime},2} } {n} } \paren{1 + \kappa_{m^{\prime},2}} 
\\
&\geq 
\croch{ \alpha_1  D_{m,1}^{-2} + \alpha_2 D_{m,2}^{-2} - \alpha_1  D_{m^{\prime},1}^{-2} - \alpha_2 D_{m^{\prime},2}^{-2}} 
\\& \qquad
+ \frac{\beta_1 (D_{m^{\prime},1} - D_{m,1})} {n} + \frac{\beta_2 (D_{m^{\prime},2} - D_{m,2}) } {n} 
- \LThm \paren{\ln(n)}^{-1/2} n^{-2/3} \enspace .
\end{align*}
Now, remark that the bias term is smaller for $S_{m^{\prime}}$ than for $S_m$ since $x \flapp \alpha_1 x^{-2} + \alpha_2 (D_m-x)^{-2}$ is decreasing on $(0,D_m \overline{\kappa} / (1+\overline{\kappa})]$.
Therefore, using the definition of $m^{\prime}$,
\begin{align*} 
P_n \gamma\paren{ \ERM_m } - P_n \gamma\paren{ \ERM_{m^{\prime}} } 
&\geq \frac{\beta_1 (D_{m^{\prime},1} - D_{m,1})} {n} + \frac{\beta_2 (D_{m^{\prime},2} - D_{m,2}) } {n} 
- \LThm \paren{\ln(n)}^{-1/2} n^{-2/3}
\\
&= \frac{(\beta_1 - \beta_2) (D_{m^{\prime},1} - D_{m,1})} {n} - \LThm \paren{\ln(n)}^{-1/2} n^{-2/3}
\\
&\geq \frac{(\beta_1 - \beta_2) \croch{ D_m \paren{ \frac{1}{1+\kappa} - \frac{1}{1+\overline{\kappa}}}-1 }} {n} - \LThm \paren{\ln(n)}^{-1/2} n^{-2/3} 
\\
&\geq \LThm n^{-2/3} - \LThm \paren{\ln(n)}^{-1/2} n^{-2/3}  > 0
\end{align*}
as soon as $n \geq \LThm$.
Therefore, $m \notin \Mdim$, which concludes the proof of Theorem~\ref{th.shape}, with $\KThmDimFAILoracle=1+\eta_{\Delta}\,$. \qed

\subsection{Proof of Proposition~\ref{pro.oracle.lin.Klarge}} \label{sec.proof.pro.oracle.lin.Klarge}
Let $\LProA = L_{\cM, \aM, c_{\mathrm{rich}}, A, \sigmin, \cbiasmaj, \cbiasmin, \betamin, \betamaj, \crXl}$ denote a constant (varying from line to line) that may only depend on the constants appearing in the assumptions of Proposition~\ref{pro.oracle.lin.Klarge} (except the constant $K$).

According to \eqref{eq.Ep2} in Proposition~\ref{VFCV.pro.EcritVFCV-EcritID}, 
\[ \E \croch{ p_2(m) } = \frac{1}{n} \sum_{\lamm} \sigl^2 \geq \frac{D_m \sigmin^2}{n} \quad 
\mbox{and} \quad \E \croch{ p_2(m) } \leq \frac{D_m \norm{\sigma}_{\infty}^2}{n} + \frac{1}{n} \sum_{\lamm} \carre{\sigld} \enspace . \]
Now, using \hypArXl\ and \hypAp, 
\[ \cbiasmaj D_m^{-\betamaj} \geq \perte{\bayes_m} = \sum_{\lamm} \pl \carre{\sigld} \geq \frac{\crXl}{D_m} \sum_{\lamm} \carre{\sigld} \]
so that 
\begin{equation} \notag 
\E\croch{p_2(m)} \leq \frac{D_m}{n} \paren{ \norm{\sigma}_{\infty}^2 + \frac{\cbiasmaj} {\crXl D_m^{\betamaj}} } \enspace . 
\end{equation}
Therefore, for every $\mM_n$ such that $D_m \geq \ln(n) \,$, 
\[ c_1(K,n) \E\croch{p_2(m)} \leq \pen(m) \leq c_2 \E\croch{p_2(m)} \]
with $c_2 = K / \sigmin^2$ and
\[ 
c_1(K,n) = \frac{K}{\norm{\sigma}_{\infty}^2 + \frac{\cbiasmaj} {\crXl \paren{\ln(n)}^{\betamaj}}} 
\geq \frac{K}{\norm{\sigma}_{\infty}^2} \times \frac{1}{1 + \frac{\cbiasmaj} {\crXl \paren{\ln(n)}^{\betamaj} \norm{\sigma}_{\infty}^2 }}
\geq \frac{K}{\norm{\sigma}_{\infty}^2} \paren{1 - \paren{\ln(n)}^{-\betamaj/2}}
\geq \frac{1}{2} \paren{ 1 + \frac{K}{\norm{\sigma}_{\infty}^2} }
 \]
as soon as $n \geq L_{\hypProA,K} \, $.
Then, Theorem~5 in \cite{Arl_Mas:2009:pente} shows that with probability at least $1 - \LProA n^{-2}$, 
\[ \perte{\ERM_{\mh}} \leq 
\paren{ \frac{1 + (\frac{K}{\sigmin^2}  - 2)_+}{
\min\set{1, \frac{1}{2} \paren{ \frac{K}{\norm{\sigma}_{\infty}^2} - 1 }}}}
\inf_{\mM_n} \set{\perte{\ERM_m}}  \]
which concludes the proof. 
When $K \geq 2 \norm{\sigma}_{\infty}^2$, the leading constant of the oracle inequality is smaller than 
\[ \paren{ \frac{1 + (\frac{K}{\sigmin^2}  - 2)_+}{1 - \paren{\ln(n)}^{-\betamaj/2}} + \paren{\ln(n)}^{-1/5} } \]
which can be made as close as possible from $K \sigmin^{-2}  - 1$ provided $n \geq L_{\hypProA,K} \, $.
\qed

\subsection{Proof of Proposition~\ref{pro.overfit.Mal}} \label{sec.proof.pro.overfit.Mal}

Let $\LProB = L_{A,\sigmin,\alpha,R,\mu,c_{X,\Leb},\inf_{[0,1]}{\sigma}}$ denote a constant (varying from line to line) that may only depend on the constants appearing in the assumptions of Proposition~\ref{pro.overfit.Mal} (except the constant $K$).

Since $\M_n = \Mdeuxpas_n$ and the penalty can be written $\pen(m) = \frac{K D_{m,1}}{n} + \frac{K D_{m,2}}{n}$, the model selection problem can actually be split into two separate model selection problems: one for the $N$ data points for which $X_i \in [0,1/2]$, the other for the $n-N$ data points for which $X_i \in (1/2,1] \,$.

For proving Proposition~\ref{pro.overfit.Mal}, we can focus on the first problem only, that is, we are given $N$ data points independent with distribution $\loi\paren{(X,Y) \sachant X \in [0,1/2]}$, where $N$ is itself a random variable whose distribution is binomial with parameters $n$ and $\P(X \in [0,1/2]) = \mu \,$. The goal is to select a model $\mt$ among the family $\Mt(n,N)$ of regular histograms on $[0,1/2]$ with a number of bins between 1 and $n/(2 \ln(n))^2 \,$.
Note that from Bernstein's inequality (see for instance Proposition~2.9 in \cite{Mas:2003:St-Flour}), we have with probability at least $1 - 2 n^{-2}$ that 
\[ n \mu + 2 \sqrt{\mu n \ln(n)} + \frac{2 \ln(n)}{3} \geq N \geq n \mu - 2 \sqrt{\mu n \ln(n)} - \frac{2 \ln(n)}{3} \enspace . \]
In particular, on some event $\Omega_n$ of probability at least $ 1 - 2 n^{-2}\,$, 
\[ \mbox{if } n \geq L_{\mu} \, , \qquad  \mbox{then } \quad \absj{ \frac{ N } {n \mu} - 1 } \leq  n^{-1/4}  \enspace . \]

Now, on $\Omega_n \, $, we apply Theorem~2 in \cite{Arl_Mas:2009:pente}.
First, let us check that the assumptions of Theorem~2 in \cite{Arl_Mas:2009:pente} are satisfied: \hypAb\ and \hypAn\ are assumed in Proposition~\ref{pro.overfit.Mal}; the upper bound on the bias of the models holds because $\bayes \in \mathcal{H}(\alpha,R) \, $; the uniform lower bound on $\P(X \in \Il)$ holds because $\P(X\in[0,1/2])=\mu>0$ and $X$ has a lower bounded density w.r.t. $\Leb([0,1/2])\,$.
Finally, we need an upper bound on $\pen(\mt) = K D_{\mt}/N \, $: Using the proof of Proposition~\ref{pro.oracle.lin.Klarge}, we have 
\[ \forall \mt \in \Mt(n,N) \, , \quad \frac{\pen(\mt)}{\E\croch{p_2(\mt,N)}} \leq \frac{K D_{\mt} N } {N D_{\mt} \inf_{t \in [0,1/2]} \set{ \sigma(t)^2 } } = \frac{ K } { \inf_{t \in [0,1/2]} \set{ \sigma(t)^2 } } < 1 \enspace . \]
So, Theorem~2 in \cite{Arl_Mas:2009:pente} shows that $D_{\mh,1} \geq L_{\hypProB,K} N \paren{\ln(N)}^{-2} \geq L_{\hypProB,K} n \paren{\ln(n)}^{-1}$ with probability at least $1 - L_{\hypProB,K} N^{-2}  = 1 - L_{\hypProB,K} n^{-2} \,$. 
The lower bound \eqref{eq.pro.overfit.Mal.risk} on the risk also follows from Theorem~2 in \cite{Arl_Mas:2009:pente} and its proof.
\qed

\section*{Acknowledgments}
The author would like to thank gratefully Pascal Massart for several fruitful discussions, and Francis Bach for pointing out an idea that led to improve Theorem~\ref{th.shape}.
The author acknowledges the support of the French Agence Nationale de la Recherche (ANR) under reference ANR-09-JCJC-0027-01.

\appendix 

\section{Appendix} \label{sec.app}

\begin{table} 
\caption{Accuracy indices $\Cor$ for each procedure in the four experiments $\pm\epsCor\,$. In each column, the most accurate procedures (taking the uncertainty $\epsCor$ into account) are bolded.
\label{LL.tab.un}}
\begin{center}
\begin{tabular}
{p{0.16\textwidth}@{\hspace{0.025\textwidth}}p{0.18\textwidth}@{\hspace{0.025\textwidth}}p{0.17\textwidth}@{\hspace{0.025\textwidth}}p{0.17\textwidth}@{\hspace{0.025\textwidth}}p{0.17\textwidth}}
\hline\noalign{\smallskip}
Experiment      & X1--005      & S0--1         & XS1--05      & X1--005$\mu$02     \\
\noalign{\smallskip}
\hline
\noalign{\smallskip}
Mal                       & $8.961 \pm 0.055$ & $6.056 \pm 0.033$ & $3.424 \pm 0.020$ & $4.693 \pm 0.019$ \\
Mal$\times 1.25$          & $7.279 \pm 0.057$ & $5.086 \pm 0.034$ & $2.347 \pm 0.015$ & $4.692 \pm 0.019$ \\
Mal$\times 2   $          & $4.101 \pm 0.036$ & $3.242 \pm 0.022$ & $1.452 \pm 0.005$ & $4.562 \pm 0.021$ \\
Mal$\times 3   $          & $3.074 \pm 0.019$ & $2.675 \pm 0.013$ & $1.338 \pm 0.003$ & $4.033 \pm 0.023$ \\
Mal$\times 4   $          & $2.862 \pm 0.015$ & $3.214 \pm 0.017$ & $1.891 \pm 0.014$ & $3.420 \pm 0.022$ \\
\noalign{\smallskip}
\hline
\noalign{\smallskip}
Mal$_{\infty}$            & $4.110 \pm 0.036$ & $3.242 \pm 0.022$ & $1.664 \pm 0.008$ & $3.122 \pm 0.020$ \\
Mal$_{\infty}\times 1.25$ & $3.371 \pm 0.024$ & $2.794 \pm 0.015$ & $1.452 \pm 0.005$ & $3.334 \pm 0.015$ \\
Mal$_{\infty}\times 2   $ & $2.862 \pm 0.015$ & $3.214 \pm 0.017$ & $1.358 \pm 0.004$ & $3.370 \pm 0.009$ \\
Mal$_{\infty}\times 3   $ & $3.033 \pm 0.019$ & $5.035 \pm 0.015$ & $3.493 \pm 0.019$ & $3.430 \pm 0.007$ \\
Mal$_{\infty}\times 4   $ & $3.549 \pm 0.025$ & $5.810 \pm 0.006$ & $5.020 \pm 0.009$ & $3.493 \pm 0.006$ \\
\noalign{\smallskip}
\hline
\noalign{\smallskip}
HO                        & $3.598 \pm 0.036$ & $2.707 \pm 0.021$ & $1.848 \pm 0.008$ & $2.398 \pm 0.018$ \\
2FCV                      & $3.104 \pm 0.032$ & $2.458 \pm 0.019$ & $1.767 \pm 0.007$ & $2.289 \pm 0.016$ \\
5FCV                      & $3.176 \pm 0.035$ & $2.538 \pm 0.021$ & $1.749 \pm 0.008$ & $2.332 \pm 0.018$ \\
10FCV                     & $3.291 \pm 0.037$ & $2.559 \pm 0.022$ & $1.738 \pm 0.008$ & $2.369 \pm 0.018$ \\
\noalign{\smallskip}
\hline
\noalign{\smallskip}
penHO                     & $5.070 \pm 0.045$ & $3.492 \pm 0.027$ & $2.529 \pm 0.014$ & $2.798 \pm 0.020$ \\
penHO$\times 1.25$        & $4.393 \pm 0.041$ & $3.072 \pm 0.024$ & $2.152 \pm 0.012$ & $2.626 \pm 0.019$ \\
penHO$\times 2   $        & $3.595 \pm 0.034$ & $2.659 \pm 0.020$ & $1.853 \pm 0.008$ & $2.751 \pm 0.034$ \\
penHO$\times 3   $        & $3.516 \pm 0.032$ & $2.558 \pm 0.018$ & $1.972 \pm 0.009$ & $3.634 \pm 0.055$ \\
penHO$\times 4   $        & $3.729 \pm 0.033$ & $2.573 \pm 0.018$ & $2.166 \pm 0.011$ & $4.663 \pm 0.070$ \\
\noalign{\smallskip}
\hline
\noalign{\smallskip}
pen2F                     & $4.530 \pm 0.043$ & $3.229 \pm 0.025$ & $2.325 \pm 0.013$ & $2.729 \pm 0.019$ \\
pen2F$\times 1.25$        & $3.649 \pm 0.037$ & $2.769 \pm 0.022$ & $1.945 \pm 0.010$ & $2.451 \pm 0.018$ \\
pen2F$\times 2   $        & $2.619 \pm 0.028$ & $2.270 \pm 0.017$ & $1.619 \pm 0.005$ & $2.062 \pm 0.014$ \\
pen2F$\times 3   $        & $2.273 \pm 0.023$ & $2.222 \pm 0.015$ & $1.539 \pm 0.005$ & $1.932 \pm 0.013$ \\
pen2F$\times 4   $        & $2.275 \pm 0.022$ & $2.381 \pm 0.016$ & $1.586 \pm 0.007$ & $1.907 \pm 0.014$ \\
\noalign{\smallskip}
\hline
\noalign{\smallskip}
pen5F                     & $3.779 \pm 0.041$        & $2.857 \pm 0.024$        & $1.925 \pm 0.010$        & $2.540 \pm 0.019$ \\
pen5F$\times 1.25$        & $2.794 \pm 0.031$        & $2.331 \pm 0.018$        & $1.646 \pm 0.007$        & $2.193 \pm 0.016$ \\
pen5F$\times 2   $        & $2.051 \pm 0.019$        & $1.995 \pm 0.012$        & $1.457 \pm 0.004$        & $1.880 \pm 0.011$ \\
pen5F$\times 3   $        & $1.777 \pm 0.013$        & $2.119 \pm 0.011$        & $1.388 \pm 0.003$        & $\meil{1.860 \pm 0.009}$ \\
pen5F$\times 4   $        & $1.838 \pm 0.015$        & $2.384 \pm 0.013$        & $1.366 \pm 0.003$        & $1.887 \pm 0.008$ \\
\noalign{\smallskip}
\hline
\noalign{\smallskip}
pen10F                    & $3.599 \pm 0.040$        & $2.726 \pm 0.024$        & $1.810 \pm 0.009$        & $2.451 \pm 0.019$ \\
pen10F$\times 1.25$       & $2.726 \pm 0.031$        & $2.215 \pm 0.018$        & $1.594 \pm 0.006$        & $2.125 \pm 0.016$ \\
pen10F$\times 2   $       & $1.893 \pm 0.016$        & $\meil{1.944 \pm 0.012}$ & $1.451 \pm 0.004$        & $\meil{1.854 \pm 0.010}$ \\
pen10F$\times 3   $       & $1.709 \pm 0.011$        & $2.132 \pm 0.011$        & $1.358 \pm 0.003$        & $1.879 \pm 0.008$ \\
pen10F$\times 4   $       & $1.706 \pm 0.011$        & $2.389 \pm 0.011$        & $1.327 \pm 0.002$        & $1.943 \pm 0.007$ \\
\noalign{\smallskip}
\hline
\noalign{\smallskip}
penLoo                    & $3.171 \pm 0.034$        & $2.499 \pm 0.021$        & $1.731 \pm 0.008$        & $2.395 \pm 0.019$ \\
penLoo$\times 1.25$       & $2.529 \pm 0.027$        & $2.118 \pm 0.016$        & $1.548 \pm 0.006$        & $2.065 \pm 0.015$ \\
penLoo$\times 2   $       & $1.870 \pm 0.014$        & $\meil{1.954 \pm 0.012}$ & $1.401 \pm 0.003$        & $1.879 \pm 0.009$ \\
penLoo$\times 3   $       & $1.701 \pm 0.010$        & $2.183 \pm 0.011$        & $1.378 \pm 0.003$        & $1.931 \pm 0.007$ \\
penLoo$\times 4   $       & $\meil{1.679 \pm 0.010}$ & $2.457 \pm 0.011$        & $\meil{1.308 \pm 0.002}$ & $2.002 \pm 0.006$ \\
\noalign{\smallskip}
\hline
\noalign{\smallskip}
Epenid                    & $2.805 \pm 0.029$        & $2.291 \pm 0.019$        & $1.702 \pm 0.008$        & $2.333 \pm 0.019$ \\
Epenid$\times 1.25$       & $2.304 \pm 0.023$        & $1.943 \pm 0.014$        & $1.513 \pm 0.005$        & $2.035 \pm 0.015$ \\
Epenid$\times 2   $       & $1.780 \pm 0.012$        & $1.897 \pm 0.011$        & $1.371 \pm 0.003$        & $1.868 \pm 0.009$ \\
Epenid$\times 3   $       & $1.687 \pm 0.009$        & $2.161 \pm 0.011$        & $1.312 \pm 0.002$        & $1.938 \pm 0.007$ \\
Epenid$\times 4   $       & $1.646 \pm 0.009$        & $2.448 \pm 0.010$        & $1.299 \pm 0.002$        & $2.005 \pm 0.006$ \\
\hline
\end{tabular} 
\end{center}
\end{table}

\begin{figure}
\begin{center}
\includegraphics[width=0.85\textwidth]{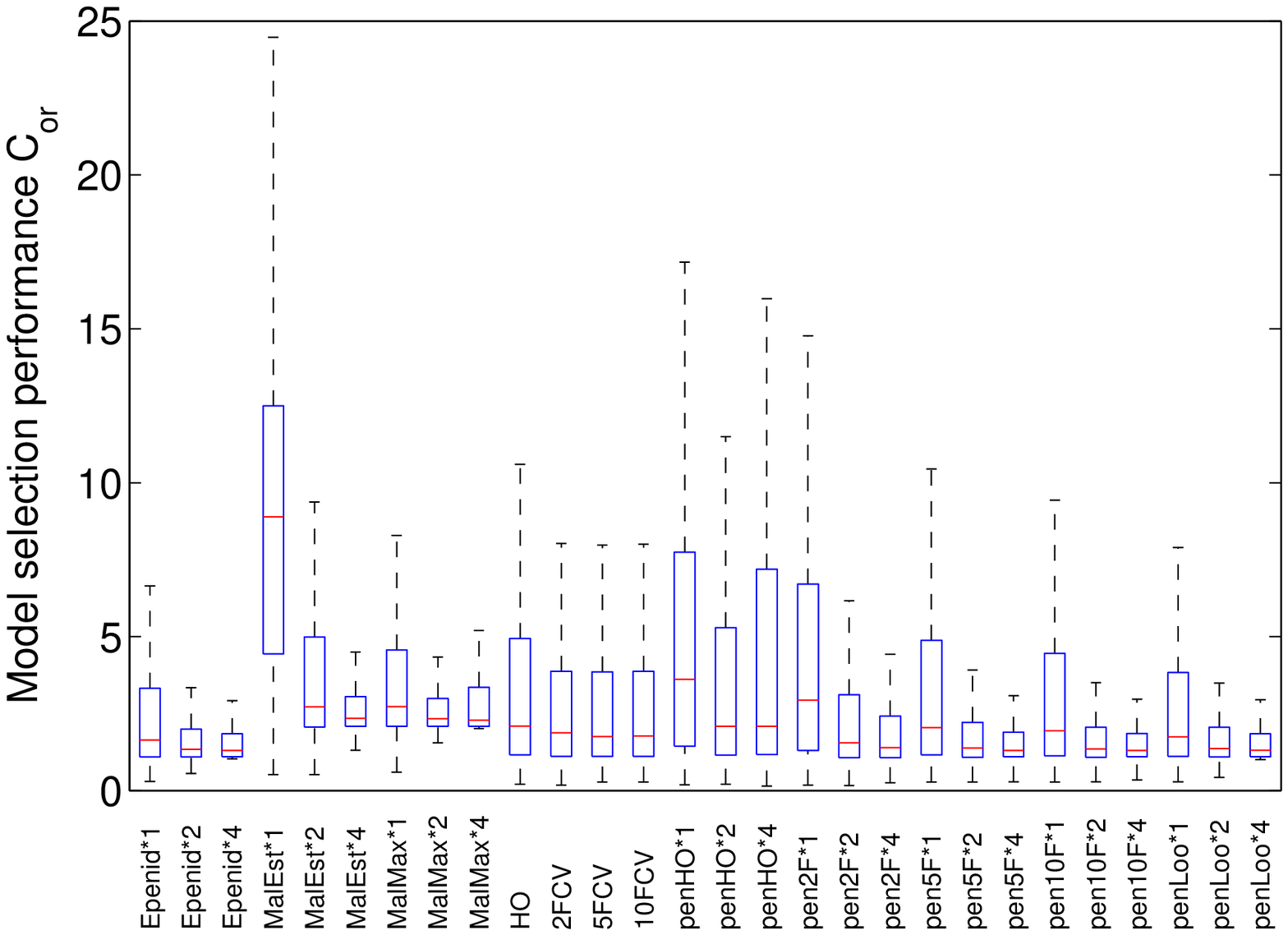}
 \hfill \\ Experiment X1--005 \vspace{0.2cm} \\
\includegraphics[width=0.85\textwidth]{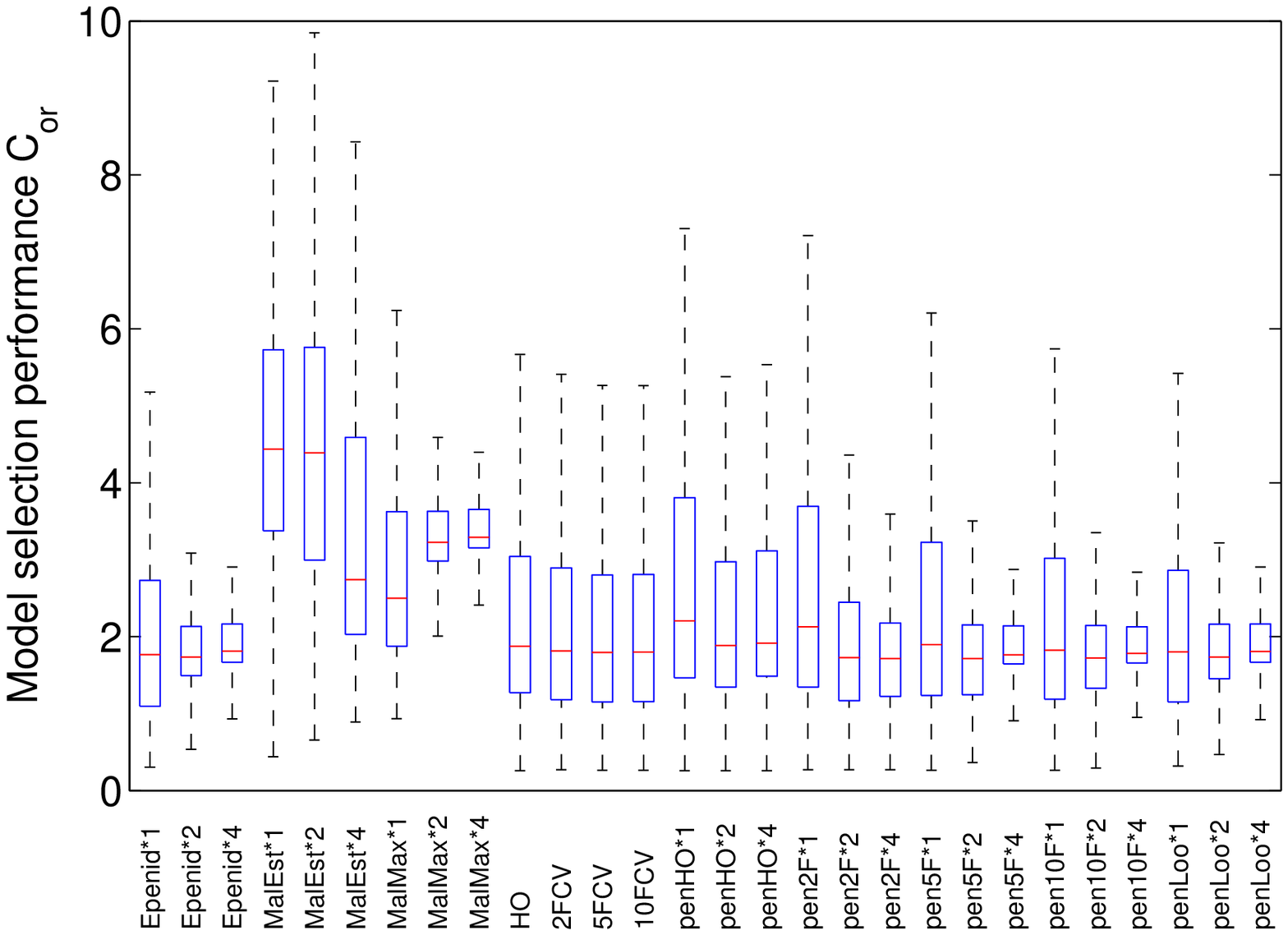}
 \hfill \\
Experiment X1--005$\mu$02 
\end{center}
\caption{Box plot of $\perte{\ERM_{\mh}}$ divided by the estimated value of  $\E\croch{\perte{\ERM_{\mo}}}$ for various algorithms in experiments X1--005 and X1--005$\mu$02. \label{fig.res.boxB-div1}}
\end{figure}
\begin{figure}
\begin{center}
\includegraphics[width=0.85\textwidth]{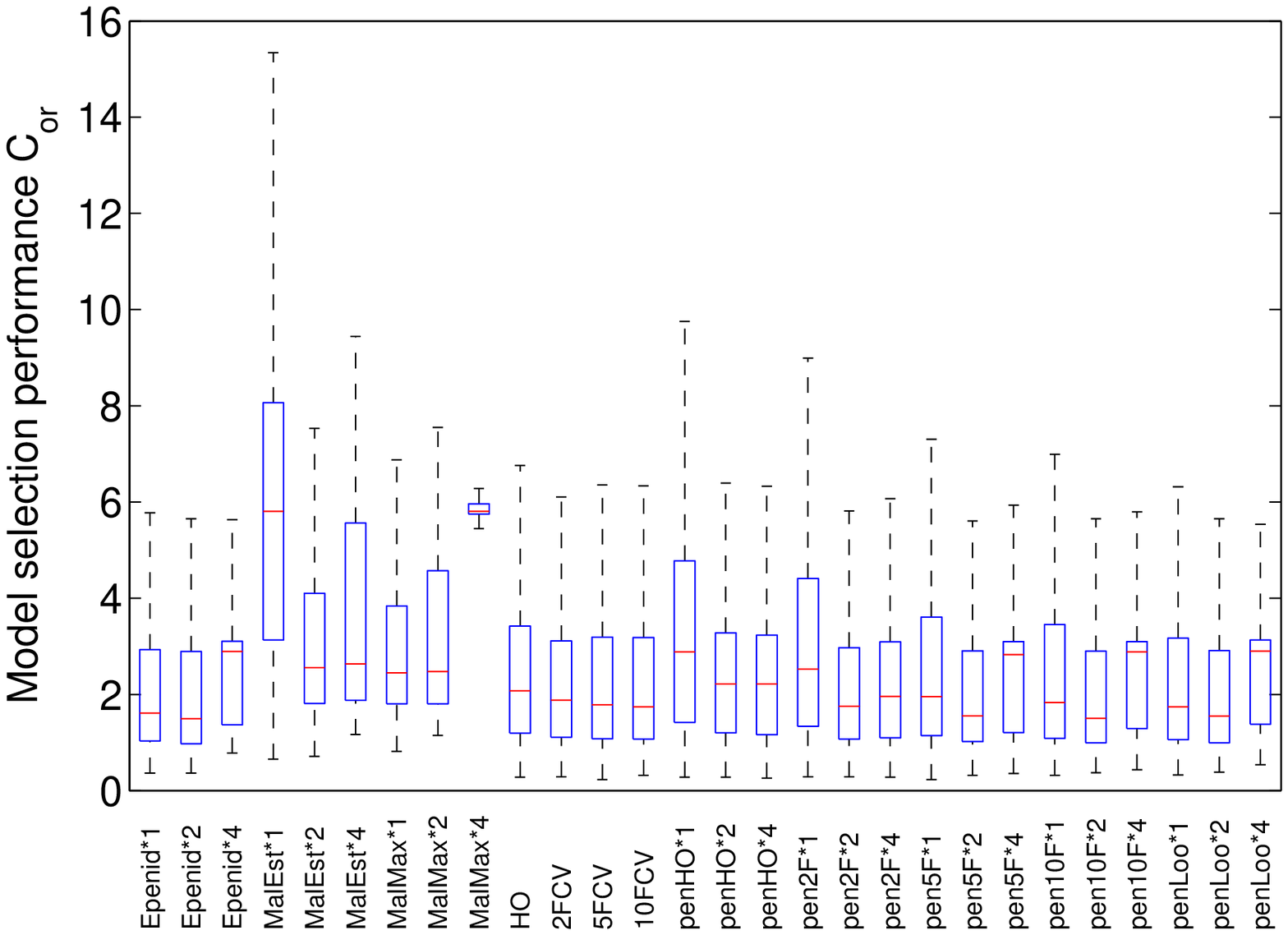}
\hfill \\ Experiment S0--1 \vspace{0.2cm} \\
\includegraphics[width=0.85\textwidth]{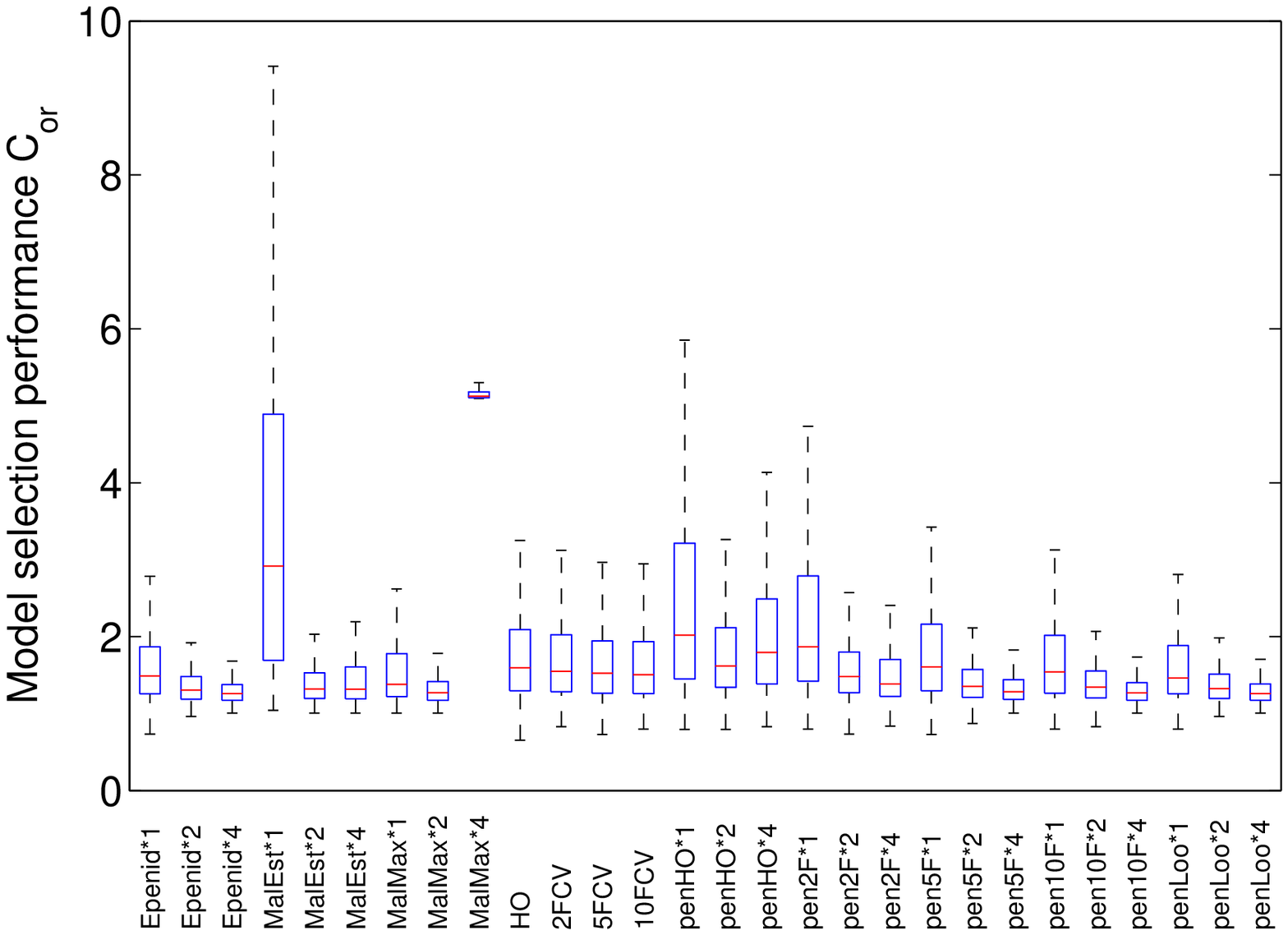}
\hfill \\
Experiment XS1--05
\end{center}
\caption{Box plot of $\perte{\ERM_{\mh}}$ divided by the estimated value of  $\E\croch{\perte{\ERM_{\mo}}}$ for various algorithms in experiments S0--1 and XS1--05. \label{fig.res.boxB-div2} }
\end{figure}

\begin{figure}
\begin{center}
\includegraphics[width=0.8\textwidth]{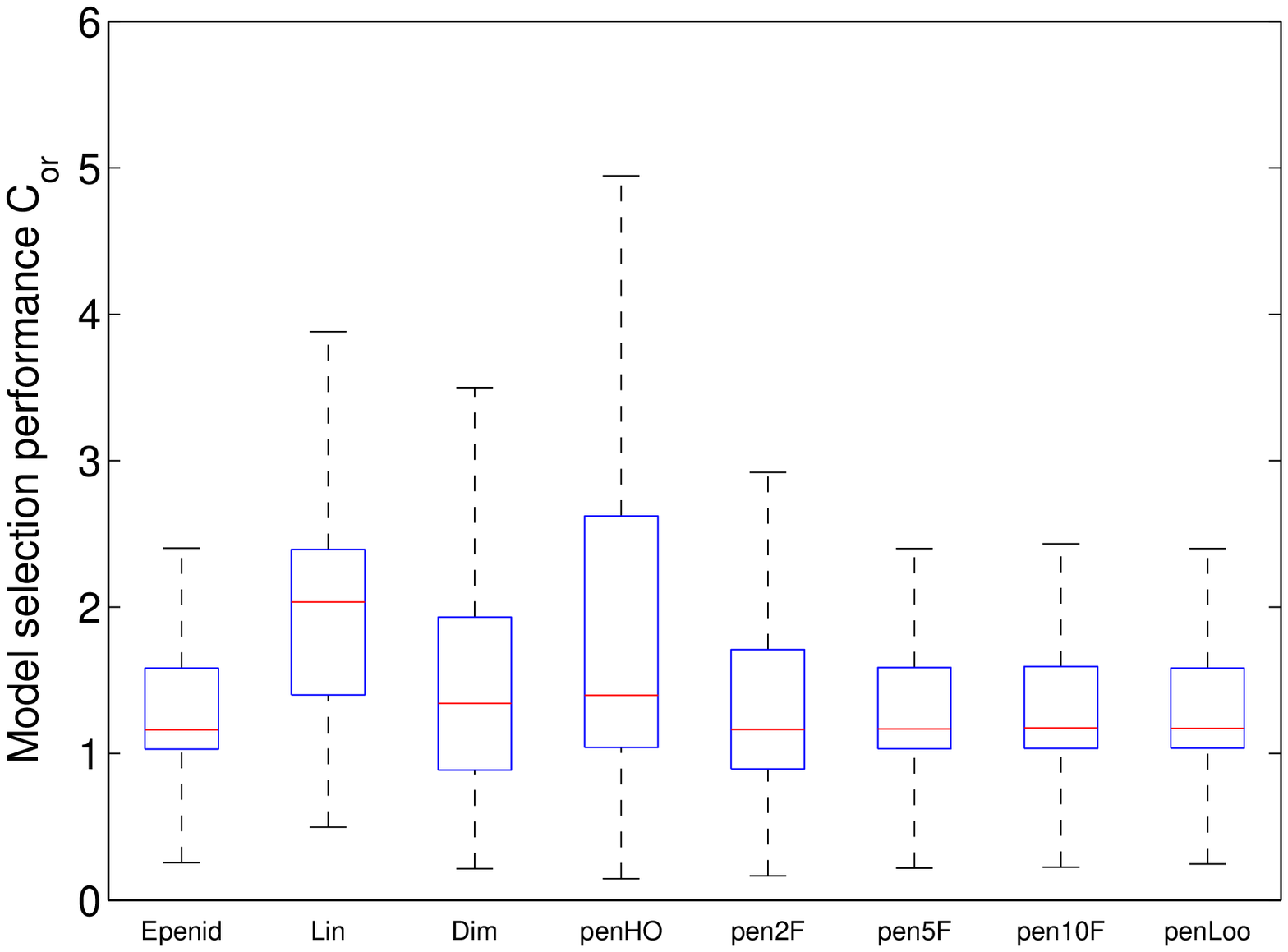}
 \hfill \\ Experiment X1--005 \vspace{0.2cm} \\
\includegraphics[width=0.8\textwidth]{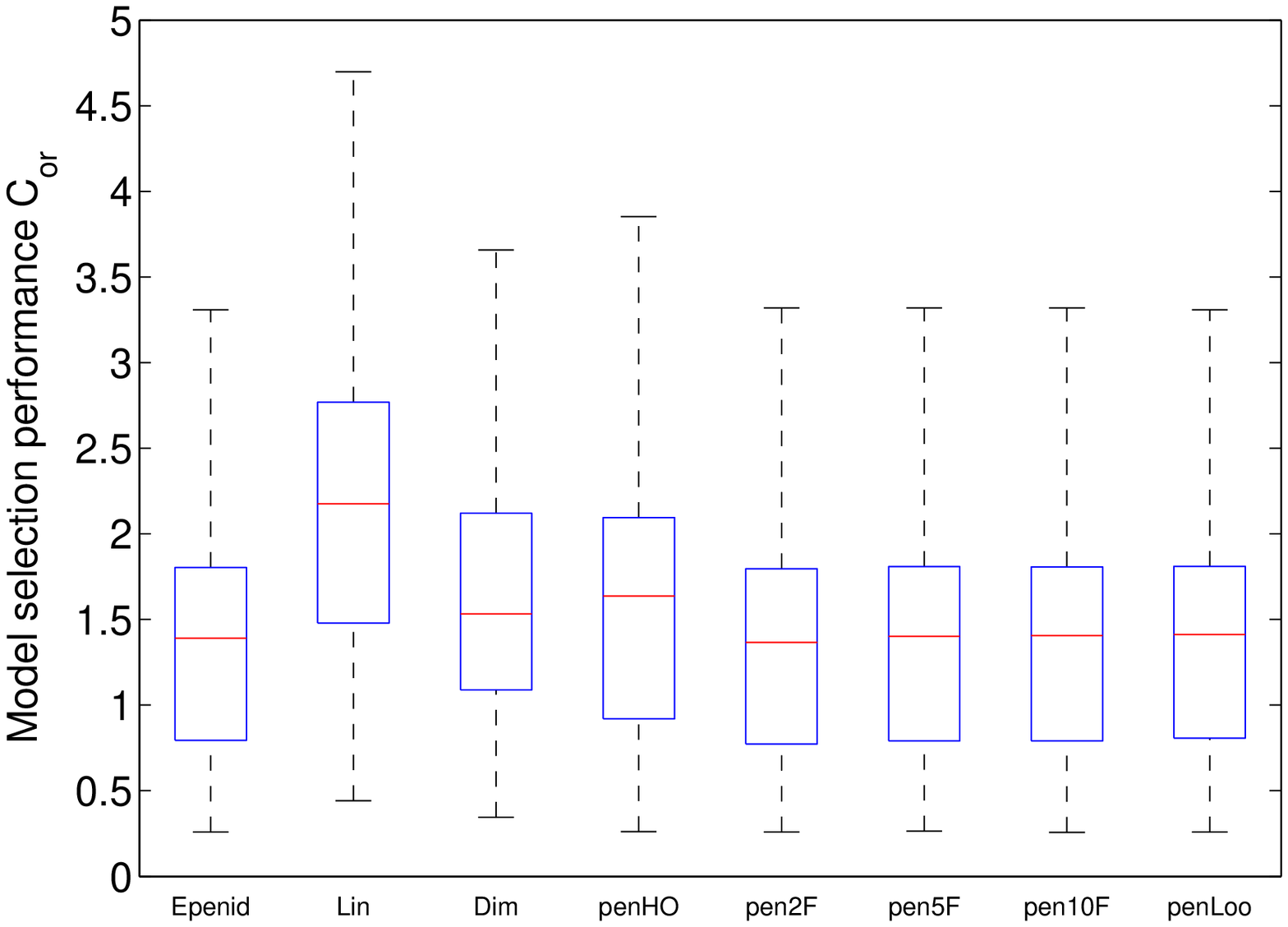}
 \hfill \\
Experiment X1--005$\mu$02 
\end{center}
\caption{Box plot of $\perte{\ERM_{\mh}}$ divided by the estimated value of  $\E\croch{\perte{\ERM_{\mo}}}$ for algorithms Id$\star$ (that is, penalties with the optimal data-driven overpenalization factor) in experiments X1--005 and X1--005$\mu$02. \label{fig.res.boxB-Id1}}
\end{figure}
\begin{figure}
\begin{center}
\includegraphics[width=0.8\textwidth]{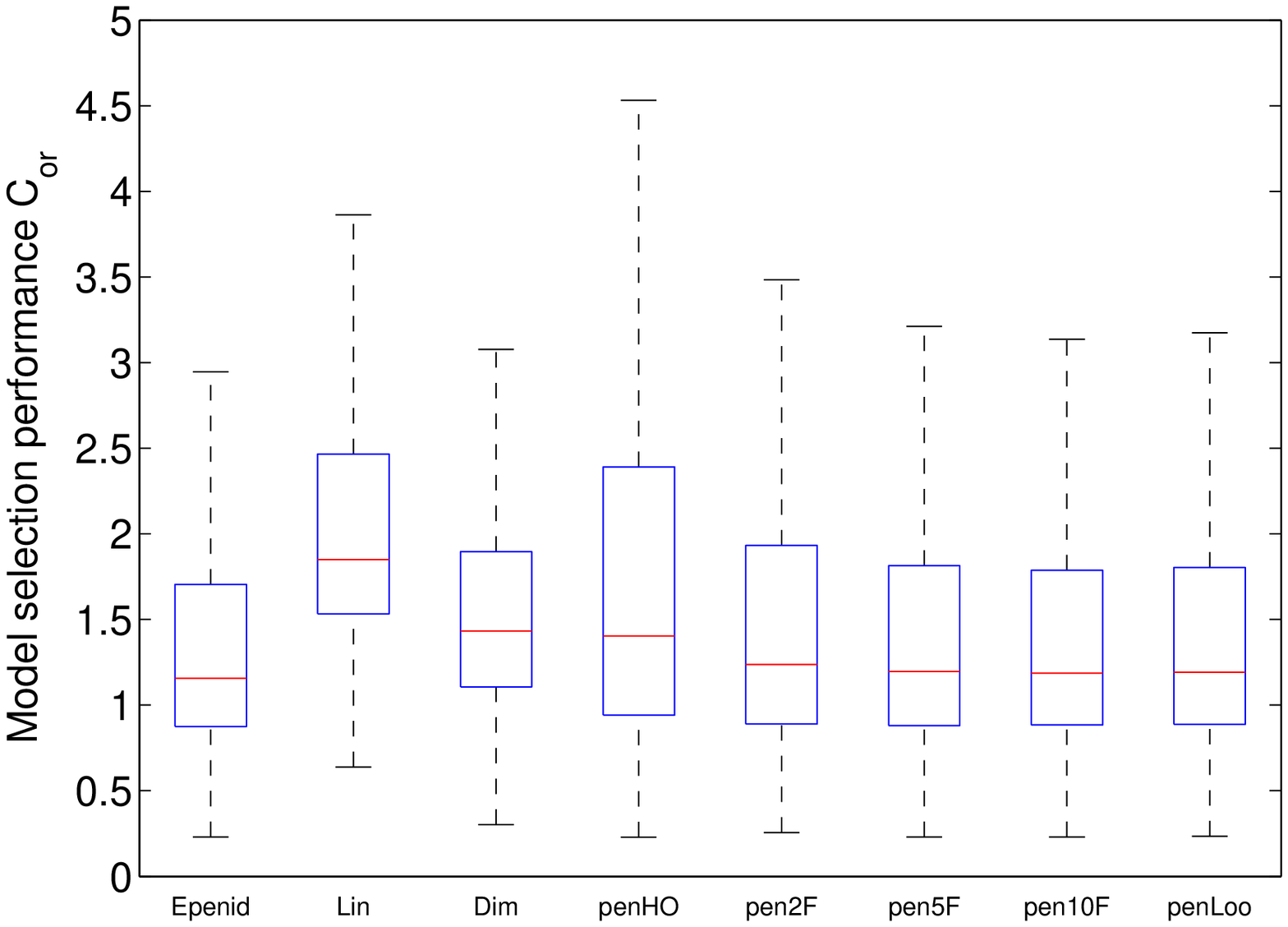}
\hfill \\ Experiment S0--1 \vspace{0.2cm} \\
\includegraphics[width=0.8\textwidth]{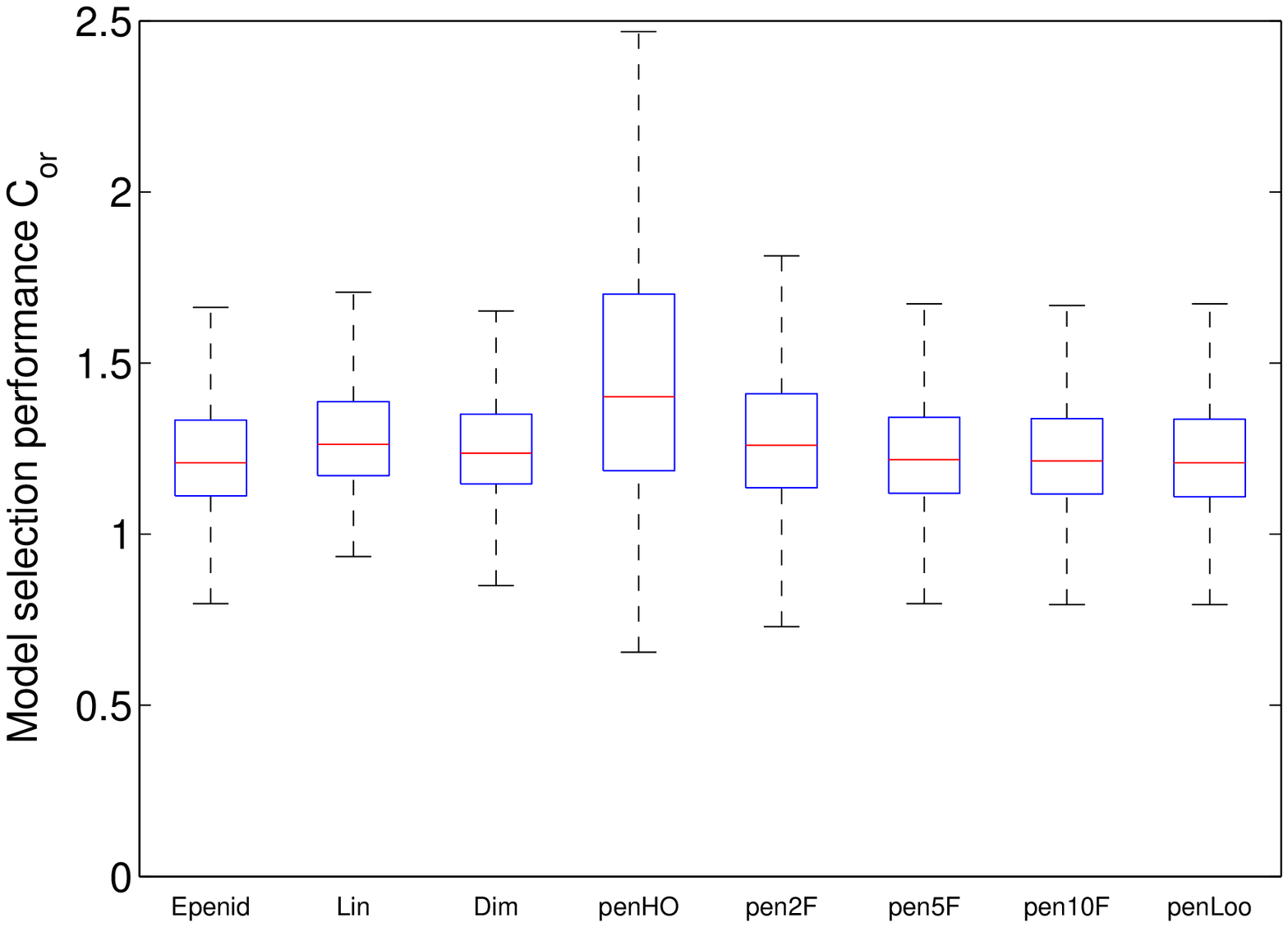}
\hfill \\
Experiment XS1--05
\end{center}
\caption{Box plot of $\perte{\ERM_{\mh}}$ divided by the estimated value of  $\E\croch{\perte{\ERM_{\mo}}}$ for algorithms Id$\star$ (that is, penalties with the optimal data-driven overpenalization factor) in experiments S0--1 and XS1--05. \label{fig.res.boxB-Id2} }
\end{figure}

\begin{figure}
\begin{center}
\begin{minipage}[b]{\linewidth}
\includegraphics[width=\textwidth]{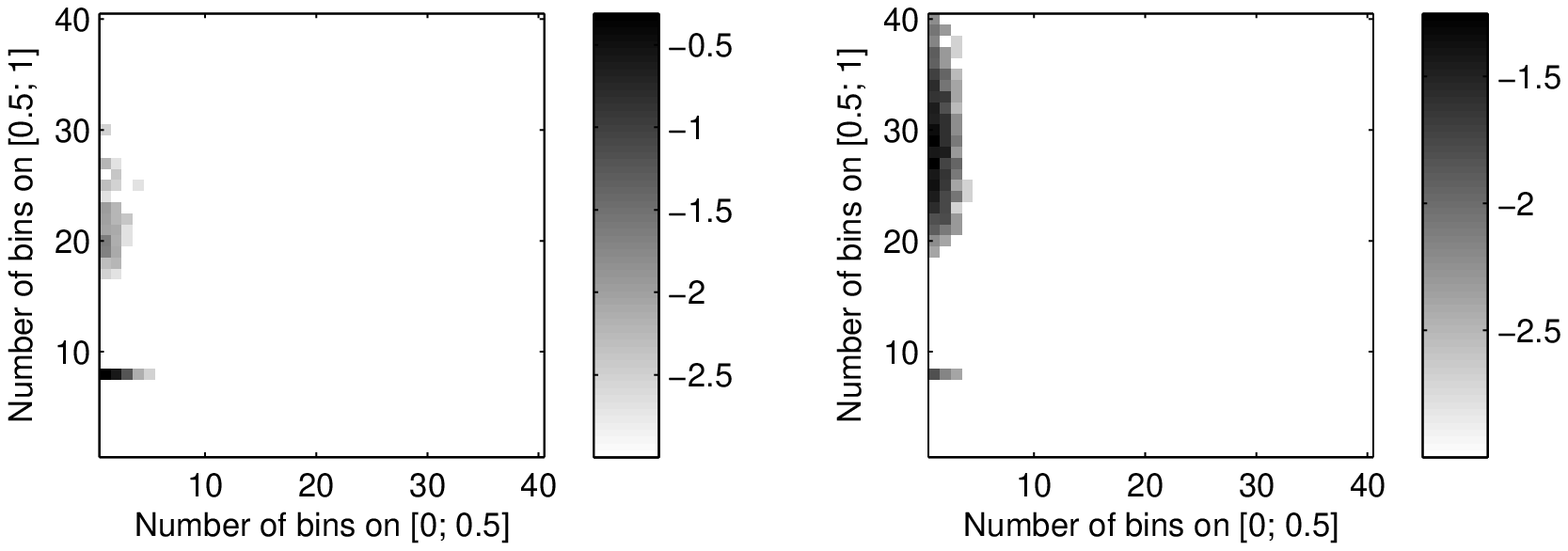}
\end{minipage} \hfill \\
\hspace{-1cm} $\middim$ \hspace{3cm} Experiment XS1--05 \hspace{2cm} $\mo$\vspace{0.2cm} \\
\begin{minipage}[b]{\linewidth}
\includegraphics[width=\textwidth]{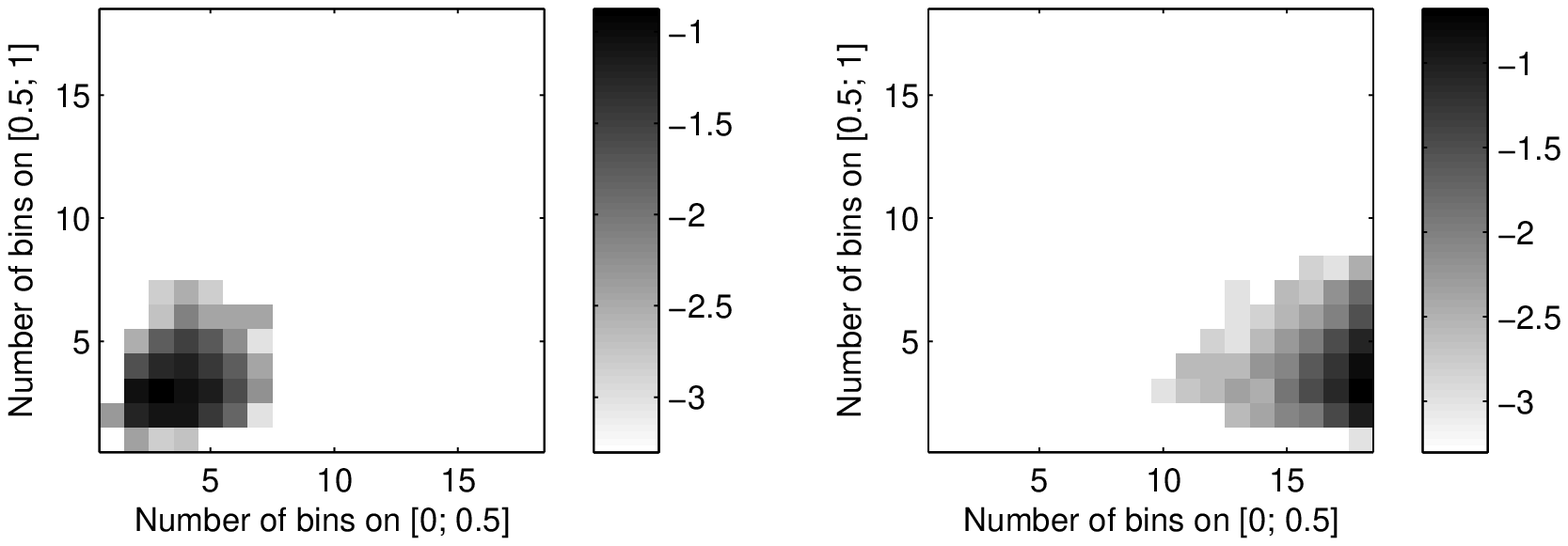}
\end{minipage} \hfill \\
\hspace{-1cm} $\middim$ \hspace{3cm} Experiment S0--1 \hspace{2cm} $\mo$\vspace{0.2cm} \\
\begin{minipage}[b]{\linewidth}
\includegraphics[width=\textwidth]{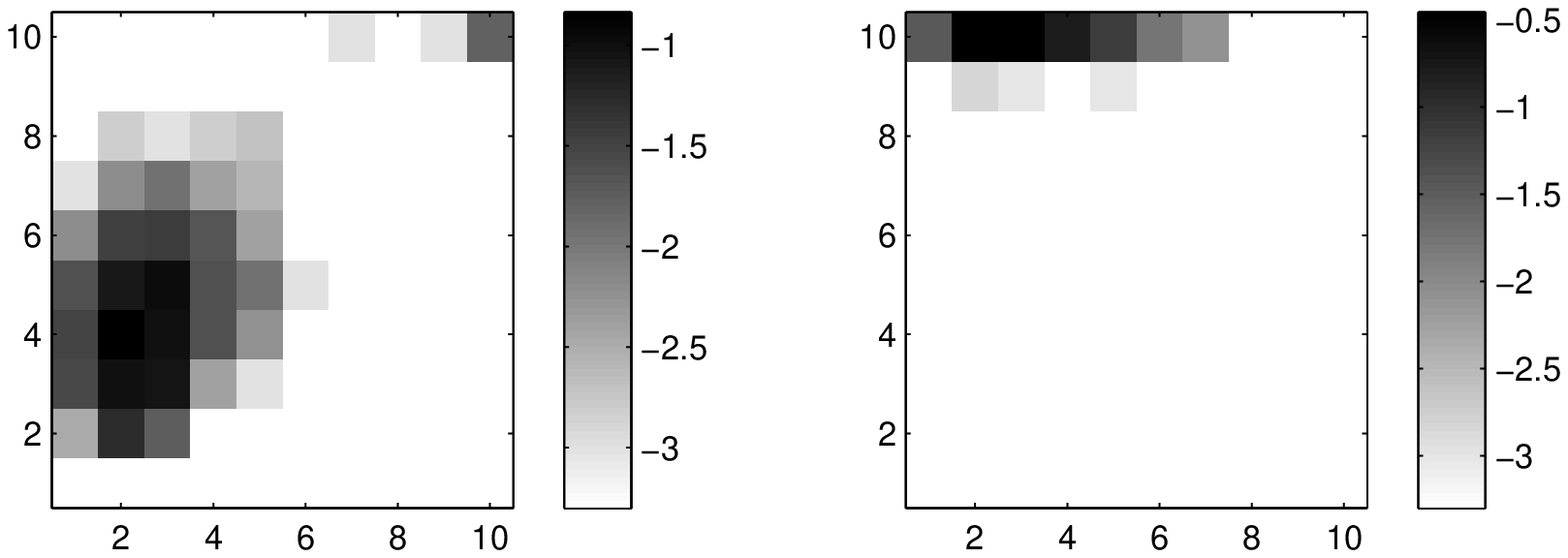}
\end{minipage} \hfill \\
\hspace{-1cm} $\middim$ \hspace{3cm} Experiment X1--005$\mu$02 \hspace{2cm} $\mo$\vspace{0.2cm} \\
\end{center}
\caption{Same as \Figpathmain\ for the three other experiments. 
Left: $\log_{10} \Prob\paren{ m = \middim}$ represented in $\R^2$ using $(D_{m,1},D_{m,2})$ as coordinates, where $\middim$ is defined by \refeqmiddim; $N=10\,000$ samples have been simulated for estimating the probabilities.
Right: $\log_{10} \Prob\paren{ m = \mo}$ using the same representation and the same samples. 
The distributions of \mo\ and \middim\ are almost disjoint.
\label{fig.path-density}
}
\end{figure}

\end{document}